\documentclass[]{elsarticle}
	\usepackage{amsmath,amsfonts,amssymb,amsthm}
	\newcounter{example}[section]
	\newenvironment{example}[1][]{\refstepcounter{example}\par\medskip
		\textbf{Test Problem~\theexample. #1} \rmfamily}{\medskip}

	\newtheorem{definition}{Definition}[section]
	\newtheorem{theorem}{Theorem}[section]

	\newtheorem{lemma}{Lemma}[section]
	\newtheorem{remark}{Remark}
	\usepackage[titletoc,title]{appendix}
	\usepackage{booktabs}
	\usepackage{caption}
	\usepackage{subcaption}
	\usepackage{lineno,hyperref}

	\journal{}
	
\begin{document}
	
	\begin{frontmatter}
	
	\title{Entropy stable discontinuous Galerkin schemes for the Relativistic Hydrodynamic Equations}
	
	\author{Biswarup Biswas}
	\author{Harish Kumar}
	\address{Department of Mathematics, Indian Institute of Technology Delhi, New Delhi, India}
	\begin{abstract}
    In this article, we present entropy stable discontinuous Galerkin numerical schemes for equations of special relativistic hydrodynamics with the ideal equation of state. The numerical schemes use the summation by parts (SBP) property of Gauss-Lobatto quadrature rules. To achieve entropy stability for the scheme, we use two-point entropy conservative numerical flux inside the cells and a suitable entropy stable numerical flux at the cell interfaces. The resulting semi-discrete scheme is then shown to entropy stable. Time discretization is performed using SSP Runge-Kutta methods. Several numerical test cases are presented to validate the accuracy and stability of the proposed schemes.
	\end{abstract}
	
	
	\end{frontmatter}
	
	\section{Introduction}
    Relativistic hydrodynamic equations are used to model astrophysical flow problems when the speed of the fluid is comparable to the speed of light, and hence relativistic effects can not be ignored. Some examples are including $\gamma$-ray bursts, astrophysical jets, core-collapse supernovae, and formation of black holes etc. (see     \cite{landau1987,begelman1984theory,mirabel1999sources,zensus1997parsec,bottcher2012relativistic,anile}).

In this article, we consider the special relativistic hydrodynamic (RHD) equations with ideal equations of states. The RHD system of equations is a set of nonlinear hyperbolic PDEs.  Due to nonlinearity in flux, the existence of the classical solutions for most applications is not possible (see \cite{godlewski_book1}). In fact, solutions often contain discontinuities even for smooth initial conditions. Hence, weak solutions are considered. The weak solutions containing discontinuities can be characterized by the solutions satisfying the Rankine-Hugoniot jump condition along the curve of discontinuity. Even though weak solutions allow discontinuities, it is often possible to construct non-unique solutions to the Cauchy problem. An additional criterion in the form of entropy stability is imposed to rule out physically irrelevant solutions. This leads to uniqueness for scalar conservation laws; however, recently, it has been shown that even the entropy stable solutions are non-unique (see \cite{camillo2015}). Still, entropy stability is one of the few nonlinear stability estimates available for systems of hyperbolic conservation laws. Hence, we aim to design numerical schemes such that the computed weak solutions are entropy stable.

Due to the presence of nonlinearity in the flux, the exact solutions for most applications are not possible. Hence, we need to use numerical methods. For the hyperbolic conservation laws, often finite volume methods are used (see \cite{godlewski_numerical_1996}). These methods evolve cell averages of the solution using the numerical flux at the cell edges. The numerical fluxes are based on either the exact Riemann solver or approximate Riemann solvers. To achieve a higher order of accuracy, TVD, ENO, or WENO based reconstructions are used (see \cite{HARTEN1983357,JIANG1996202}). Another popular numerical schemes are discontinuous Galerkin schemes, which are developed in  \cite{chavent1989,cockburn1989,cockburn1988}.

Numerical methods for the RHD equations have seen rapid development in the last few years. One of the first numerical scheme was proposed bt Wilson \cite{wilson1972numerical,wilson1979numerical}. Also, Riemann solvers for RHD equations are presented in several papers (see, \cite{marti1991numerical,marti1994analytical,dai1997iterative,ibanez1999riemann}). Higher-order schemes are developed in \cite{aloy1999genesis,marti1996extension,mignone2005piecewise}, using  piece-wise parabolic reconstruction methods . Furthermore, a scheme using TVD reconstruction is presented in \cite{choi2005numerical}. Higher-order ENO and WENO schemes are developed in (\cite{weno2007Tchekhovskoy,dolezal1995relativistic}). In  \cite{font2008numerical}, authors present a review of a wide range of numerical schemes by studying there performance on a large class of problems. More recently, several robust finite difference and DG schemes have been developed in \cite{wu2015high,zhao2017central,qin2016bound}.

Higher-order entropy stable numerical schemes for hyperbolic conservation laws have been developed in the recent past (see \cite{tadmor2003entropy,ismail2009affordable,chandrashekar2013kinetic,kumar2012entropy,sen2018entropy}). These schemes are based on entropy conservative numerical flux, which is modified to add appropriate higher-order accurate diffusion operator. For the RHD equations \cite{deepak2019entropy} has proposed such a scheme. In the case of DG methods, an entropy stable numerical scheme for shallow water equations was presented in \cite{Gassner2016}. More recently, Chen and Shu \cite{chen2017entropy} has presented a general framework to design entropy stable DG methods, which is also extended to equations of magnetohydrodynamics \cite{Liu2018}.

In this article, we propose an approximation of the RHD equations using entropy stable higher-order nodal DG schemes. We utilize the framework in \cite{chen2017entropy}, which exploits the SBP property of the Gauss-Lobatto quadrature rules. We then use entropy conservative numerical flux for RHD from \cite{deepak2019entropy}, inside each cell, and a suitable entropy stable numerical flux at the cell interfaces. This results in a semi-discrete entropy stable, higher-order accurate, and consistent DG scheme. The rest of this article is organized as follows:

In the next Section \ref{sec:rhd}, we present RHD equations with the ideal equation of states. We also present the entropy framework for the RHD equations. In Section \ref{sec:1dschemes}, we present entropy stable DG schemes for one dimensional case, which are then extended to two dimensions in Section \ref{sec:2dscheme}. Numerical results for one and two-dimensional test cases are presented in Section \ref{sec:num}. Finally, we present concluding remarks in Section \ref{sec:con}.

\section{Relativistic hydrodynamic equations}
\label{sec:rhd}
In the laboratory frame of reference, equations of ideal relativistic hydrodynamics can be written as follows (see \cite{landau1987,anile,Ryu_2006}):
\begin{subequations}\label{rhd_eqs1}
	\begin{align}
	\dfrac{\partial D}{\partial t} + \nabla \cdot (D \mathbf{u})&= 0,\\
	\dfrac{\partial \mathbf{m}}{\partial t} + \nabla \cdot (\mathbf{m} \mathbf{u} + p \mathbf{I})&= 0,\\
	\dfrac{\partial E}{\partial t} + \nabla \cdot\mathbf{m}&= 0.
	\end{align}
\end{subequations}
Here, $\mathbf{w}=(D,\mathbf{m},E)$ is the vector of conservative variables with  $D$ as the fluid mass density, $\mathbf{m}$ is the momentum density, and $E$ is the energy density. The vector of the primitive variable is given by  $(\rho, \mathbf{u},p)$, where $\rho$ is the rest-mass density, $\mathbf{u}$ is the fluid velocity vector, and $p$ the is the kinetic pressure. The conservative and primitive variables are connected via relations,
\begin{equation}
D=\Gamma \rho,\;\mathbf{m}=\rho h \Gamma^2 \mathbf{u},\;E=\rho h \Gamma^2-p,
\end{equation}
where $h$ is the special enthalpy and $\Gamma$ is the Lorentz factor, given by,
\begin{equation}
\label{eq:Gamma}
\Gamma = \frac{1}{\sqrt{1-|\mathbf{u}|^2}}, \qquad \text{with} \qquad |\mathbf{u}|<1
\end{equation}
The system \eqref{rhd_eqs1}, is closed assuming ideal equations of state, given by,
\begin{equation*}
h=1+\dfrac{\gamma}{\gamma-1}\dfrac{p}{\rho}.
\end{equation*}
One of the key difficulties in the case of RHD equations is to find primitive variables from the conservative variable. In this article, we follow the process prescribed in \cite{deepak2019entropy}. In conservative variable $\mathbf{w}$, the RHD equations \eqref{rhd_eqs1} can written as, 
\begin{equation}\label{CLaw}
\dfrac{\partial \mathbf{w}}{\partial t}+\dfrac{\partial \mathbf{f_1}}{\partial x}+\dfrac{\partial \mathbf{f_2}}{\partial y}=0.
\end{equation}
where the fluxes are given by,
\begin{equation*}
\mathbf{f_1}=\begin{bmatrix}
D u_x\\
m_x u_x+p\\
m_y u_x\\
m_x
\end{bmatrix} , \,\text{and}\,
\mathbf{f_2}=\begin{bmatrix}
D u_y\\
m_x u_y\\
m_y u_y+p\\
m_y
\end{bmatrix}.
\end{equation*}
We consider the following solution set of physical admissible states:
\begin{equation}
\label{eq:Omega}
\Omega=\left\{(D,\mathbf{m},E)\;\;| \;  \rho>0,\;\; p>0, \;\;|\mathbf{u}|<1)\right\}.
\end{equation}
Hence, solutions are physically admissible only if density and pressure are positive, and the magnitude of velocity is less than unity. We now have the following result:
\begin{lemma}[see \cite{anile,Ryu_2006}]
The RHD system \eqref{rhd_eqs1} is hyperbolic for states in $\Omega$, with real eigenvalues and a complete set of corresponding eigenvectors.
\end{lemma}
The expression for the complete set of eigenvalues can be found in \cite{deepak2019entropy}.	
Let us recall the following definition:
\begin{definition}
	A convex function $\mathcal{U}$ is said to be an entropy function for conservation laws \eqref{CLaw} if there exist smooth functions $\mathcal{F}_j(j=1,2)$ such that \begin{equation}
	\mathcal{F}_j'=\mathcal{U}'(\mathbf{w})\mathbf{f}_{j}'(\mathbf{w}).
	\end{equation}
\end{definition}
For RHD equations \eqref{CLaw}, the entropy functions and corresponding entropy flux functions are given by,
\begin{equation}
\mathcal{U}=-\dfrac{\rho \Gamma s}{\gamma-1},\,\mathcal{F}_1=-\dfrac{\rho \Gamma s u_x}{\gamma-1},\,\mathcal{F}_2=-\dfrac{\rho \Gamma s u_y}{\gamma-1},
\end{equation}
where $s=\ln(p \rho^{-\gamma})$ is the thermodynamic entropy. Let us introduce the entropy variables, $\mathbf{v}(\mathbf{w})={\partial_{\mathbf{w}} \mathcal{U}} $, and the entropy potentials, $\psi^\alpha(\mathbf{w})=\mathbf{v}^\top(\mathbf{w}) \cdot \mathbf{f}^\alpha(\mathbf{w})-\mathcal{F}^\alpha(\mathbf{w}) $ for $\alpha \in \{x,y\}$. A simple calculation results in
\begin{equation*}
\mathbf{v}= \begin{pmatrix}
\frac{\gamma- s}{\gamma -1} +\beta \\
u_x \Gamma \beta \\ 
u_y \Gamma \beta \\
- \Gamma \beta
\end{pmatrix}, \ \text{and }
{\psi}^\alpha=\rho \Gamma u_\alpha  \ \text{ for } \alpha \in \{x,y\}, \beta=\frac{\rho}{p}.
\end{equation*}
We have the following result from \cite{deepak2019entropy}:
\begin{lemma}
	The smooth solutions of \eqref{CLaw} satisfies,
	\begin{equation}\label{en_eq}
	\dfrac{\partial \mathcal{U}}{\partial t}+\dfrac{\partial \mathcal{F}_1}{\partial x}+\dfrac{\partial \mathcal{F}_2}{\partial y}=0.
	\end{equation}
\end{lemma}
For non-smooth solutions we have the entropy inequality,
\begin{equation}\label{en_ineq}
\dfrac{\partial \mathcal{U}}{\partial t}+\dfrac{\partial \mathcal{F}_1}{\partial x}+\dfrac{\partial \mathcal{F}_2}{\partial y}\leq0.
\end{equation}

We aim to design DG schemes that satisfy a semi-discrete version of \eqref{en_ineq}.

\section{Entropy stable DG schemes: One dimensional schemes}	
\label{sec:1dschemes}
We will first present the one dimensional case, i.e. we consider,
\begin{equation}\label{CLaw_1d}
\dfrac{\partial \mathbf{w}}{\partial t}+\dfrac{\partial \mathbf{f_1}}{\partial x}=0.
\end{equation}
The spatial domain is partitioned into $N$ elements $I_i=\left[x_{i-\frac{1}{2}},\,x_{i+\frac{1}{2}}\right]\, (1\leq i\leq N)$. Then in DG methods, we aim to find the solution, \begin{equation*}
\mathbf{w}_h\in \mathbf{V}^{k}_h:=\displaystyle \left\{\mathbf{v}_h:\mathbf{v}_h|_{I_i}\in \left[\mathcal{P}^k(I_i)\right]^4,\,1\leq i\leq N \right\},
\end{equation*} such that for all $\mathbf{v}_h\in \mathbf{V}^{k}_h$ and for all $1\leq i\leq N$,
\begin{equation}\label{dgntform}
\begin{split}
\int_{I_i} \dfrac{\partial \mathbf{w}_h^T}{\partial t} \mathbf{v}_h dx-&\int_{I_i} \mathbf{f_1}(\mathbf{w}_h)^T\dfrac{d \mathbf{v}_h}{d x}  dx\\ &+ \hat{\mathbf{f_1}}^T_{i+1/2} \mathbf{v}_h(x_{i+1/2}^-)-\hat{\mathbf{f_1}}^T_{i-1/2} \mathbf{v}_h(x_{i-1/2}^+)=0,
\end{split}
\end{equation} where $\hat{\mathbf{f_1}}_{i+1/2}$ is an numerical flux depends on the numerical solutions  at element interface, that is, $\hat{\mathbf{f_1}}_{i+1/2}=\hat{\mathbf{f_1}}\left(\mathbf{w}_h(x_{i+1/2}^-), \mathbf{w}_h(x_{i+1/2}^+) \right)$. We follow \cite{chen2017entropy} and use the Gauss-Lobatto quadrature rule to approximate the  integrals. First, we present some notations and properties related to the Gauss-Lobatto quadrature.
\subsection{Gauss-Lobatto quadrature and summation-by-parts}
Consider a reference element  $I=[-1,1]$. Then $k$ Gauss-Lobatto quadrature points are taken as,
\begin{equation*}
-1=\xi_0<\xi_1<\xi_2...<\xi_k=1,
\end{equation*}
with corresponding quadrature weights $\omega_j,\, 0\leq j\leq k$.
Now, consider the following continuous and discrete inner products, \begin{align}
\langle u,\,v\rangle:=&\int_{I}uv d\xi,\\
\langle u,\,v\rangle _h:&=\sum_{j=0}^{k}\omega_j u(\xi_j)v(\xi_j).
\end{align} Then, using the Lagrange(nodal) basis, 
\begin{equation}\label{nodal_basis}
L_j(\xi)=\prod_{l=0,l\neq j}^{N}\dfrac{\xi-\xi_l}{\xi_j-\xi_l},
\end{equation} we define the matrices $D$, $M$, and $S$ as follows,
\begin{equation*}
\begin{cases}
D_{jl}=L'_l(\xi_j),\\
M_{jl}=\langle L_j,\,L_l\rangle _{h}=\omega_j\delta_{jl},\\
S{jl}=\langle L_j,\,L'_l\rangle _{h}=\langle L_j,\,L'_l\rangle.
\end{cases}
\end{equation*}
These matrices follow the following property:
\begin{theorem}[SBP property (\cite{carpenter2014entropy})]
	The following discrete analogue of integration by parts, known as summation-by-parts (SBP) holds:
	\begin{equation*}
	\begin{cases}
	S=MD,\\
	MD+D^TM=S+S^T=B,
	\end{cases}
	\end{equation*}
	where the matrix $B$ is known as the boundary matrix given by,\begin{equation*}
	B=diag\{\tau_0, \tau_1,...,\tau_k \},\; \tau_j:=\begin{cases}
	-1 & j=0\\
	0 & 1\leq j\leq k-1\\
	1 & j=k
	\end{cases}.
	\end{equation*}
\end{theorem} 
The SBP operators also have the following properties (see \cite{chen2017entropy}):
\begin{equation}\label{prop2}
\sum_{l=0}^{k}D_{jl}=\sum_{l=0}^{k}S_{jl}=0,\,\sum_{l=0}^{k}S_{lj}=\tau_j.
\end{equation}

\subsection*{}


To apply the  Gauss-Lobatto quadrature in \eqref{dgntform} we use the change of variable,
\begin{equation}\label{COV}
x_i(\xi)=\dfrac{1}{2}(x_{i-1/2}+x_{i+1/2})+\dfrac{\xi}{2}\Delta x_i,
\end{equation}
in \eqref{dgntform}to get,
\begin{equation}\label{dgntform_ref}
\begin{split}
\dfrac{\Delta x _i}{2}\int_{I} \dfrac{\partial \mathbf{w}_h^T}{\partial t} \mathbf{v}_h dx-&\int_{I} \mathbf{f_1}(\mathbf{w}_h)^T\dfrac{d \mathbf{v}_h}{d x}  dx\\ &+ \hat{\mathbf{f_1}}^T_{i+1/2} \mathbf{v}_h(x_{i+1/2}^-)-\hat{\mathbf{f_1}}^T_{i-1/2} \mathbf{v}_h(x_{i-1/2}^+)=0.
\end{split}
\end{equation}
To simplify the further discussion, let us present  the further scheme for  the scalar case, which can be easily extended to the system \eqref{CLaw_1d}. We express $w_h$ using nodal basis \eqref{nodal_basis} as, $w_h=\sum_{j=0}^{k}w^i_jL_j(\xi)$. We approximate $f_1(w_h)$ as, $f_1(w_h)\approx {f_1}_h(\xi):=\sum_{j=0}^{k}f_1(w^i_j)L_j(\xi)$.
Then applying the Gauss-Lobatto quadrature in \eqref{dgntform_ref} and choosing the test function $v_h=L_j$, we get,

\begin{equation}
\label{afterInt}
\frac{\Delta x_{i}}{2}\sum_{l=0}^{k}\frac{d w_l^i}{dt}\langle L_l,\, L_j\rangle_{h}- \sum_{l=0}^{k}{f_1}_l^i\langle L_l,\, L'_j\rangle+{\hat{f_1}_{i+1/2}} L_j(1)-{\hat{f_1}_{i-1/2}} L_j(-1)=0.
\end{equation}

This can be written in the simplified form,
\begin{equation}\label{weaknodal}
\frac{\Delta x_{i}}{2}M\frac{d w^i}{dt}- S^T{f_1}^i+B\hat{f_1}^i=0.
\end{equation}
where, the following notations are used,
\begin{align*}
w^i&=[w^i_0,\dots, w^i_k]^T\\
{f_1}^i&=[{f_1}^i_0,\dots, {f_1}^i_k]^T\\
\hat{f_1}^i&=[{f_1}_{i-1/2},0,\dots,0, {f_1}_{i+1/2}]^T.
\end{align*}
Using the SBP property, scheme \eqref{weaknodal} can be further simplified as,
\begin{equation}\label{strong}
\frac{\Delta x_{i}}{2}\frac{d w^i}{dt}- D{f_1}^i=M^{-1}B\left({f_1}^i-\hat{f_1}^i\right)=0.
\end{equation}
On a single element, the scheme \eqref{weaknodal} can be written as,
\begin{equation}\label{DG_single_nonen}
\dfrac{\Delta x}{2}\dfrac{d w_p}{dt}+ \sum_{l=0}^{k} D_{pl}{f_1}_l=\dfrac{\tau_p}{\omega_p}({f_{1}}_p-\hat{f_1}_p),\,\,p=0,1,\dots,k,
\end{equation}
where the indices $i$ is dropped. 
For the system case, one can proceed with the similar analysis to achieve the following scheme,
\begin{equation}\label{DG_single_nonen_sys}
\dfrac{\Delta x}{2}\dfrac{d \mathbf{w}_p}{dt}+ \sum_{l=0}^{k}  D_{pl}\mathbf{f_1}_l=\dfrac{\tau_p}{\omega_p}(\mathbf{f_{1}}_p-\hat{\mathbf{f_1}}_p),\,\,p=0,1,\dots,k.
\end{equation}
In general, the scheme \eqref{strong} is not entropy stable. To achieve entropy stability, we need to modify the scheme. We proceed as follows:
\begin{definition}
	A two point symmetric, consistent numerical flux is said to be entropy conservative flux for an entropy function $\mathcal{U}$ if \begin{equation}
	(\mathbf{v}_R-\mathbf{v}_L)^\text{T}\mathbf{f}^*(\mathbf{w}_R,\,\mathbf{w}_L)=\psi_R-\psi_L,
	\end{equation}
	where $\mathbf{v}=\mathcal{U}'(\mathbf{w})$ is known as entropy variable, and $\psi=\mathbf{v}^\text{T}\cdot\mathbf{f}-\mathcal{F}$ is the entropy potential. $L$ and $R$ in suffix denote the left and right state.
\end{definition}
\begin{definition}
	A two-point symmetric, consistent numerical flux is said to be entropy stable flux for the entropy function $\mathcal{U}$ if \begin{equation}
	(\mathbf{v}_R-\mathbf{v}_L)^\text{T}\mathbf{f}^*(\mathbf{w}_R,\,\mathbf{w}_L)\leq\psi_R-\psi_L.
	\end{equation}
\end{definition}

We modify the scheme \eqref{DG_single_nonen_sys} as,

\begin{equation}\label{DG_single_en}
\dfrac{\Delta x}{2}\dfrac{d \mathbf{w}_p}{dt}+ \sum_{l=0}^{k} 2 D_{pl}\mathbf{f_1}^{*}(\mathbf{w}_p,\mathbf{w}_l)=\dfrac{\tau_p}{\omega_p}(\mathbf{f_{1}}_p-\hat{\mathbf{f_1}}_p),\,\,p=0,1,\dots,k
\end{equation}

where the following notation is used,
\begin{equation*}
[\hat{\mathbf{f_1}}_0,\hat{\mathbf{f_1}}_1\dots,\hat{\mathbf{f_1}}_k]=[\hat{\mathbf{f_1}}_{i-1/2},0,\dots,0, \hat{\mathbf{f_1}}_{i+1/2}].
\end{equation*}
%
Here $\mathbf{f_1}^*$ is a two-point entropy conservative flux.

\begin{theorem}[\cite{chen2017entropy}]
	\label{thm:ent_con}
	If $\mathbf{f_1}^*(\mathbf{w}_L,\mathbf{w}_R)$ is consistent and symmetric, then \eqref{DG_single_en} is conservative and high order accurate. If we further assume that $\mathbf{f_1}^*(\mathbf{w}_L,\mathbf{w}_R)$ is entropy conservative, then \eqref{DG_single_en} is also locally entropy conservative within a single element.
\end{theorem}
\begin{proof}
Proof of accuracy and conservation follows from \cite{chen2017entropy}. Following  \cite{chen2017entropy}, we present the  proof of entropy conservation:
	\par {\setlength{\parindent}{0cm}\textbf{Entropy conservation:}}
	Entropy production within a single element is given by, 
	\begin{flalign}\label{en_exp}
	\dfrac{d}{dt}\left(\sum_{j=0}^{k}\dfrac{\Delta x}{2}\omega_j\mathcal{U}_j\right)&=\sum_{j=0}^{k}\dfrac{\Delta x}{2}\omega_j\mathbf{v}_j^T\dfrac{d\mathbf{w}_j}{dt}\nonumber\\
	&=\sum_{j=0}^{k}\tau_j\mathbf{v}_j^T(\mathbf{f_1}_j-\hat{\mathbf{f_1}}_j)-2{\sum_{j=0}^{k}\sum_{l=0}^{k}S_{jl}\mathbf{v}_j^T\mathbf{f_{1}}^*(\mathbf{w}_j,\mathbf{w}_l)}\\
	\end{flalign}
	Simplify the second term as,
	\begin{flalign}
	2\sum_{j=0}^{k}\sum_{l=0}^{k}S_{jl}\mathbf{v}_j^T\mathbf{f_{1}}^*(\mathbf{w}_j,\mathbf{w}_l)&=\sum_{j=0}^{k}\sum_{l=0}^{k}(B_{jl}+S_{jl}-S_{lj})\mathbf{v}_j^Tf^*(\mathbf{w}_j,\mathbf{w}_l), \;( S+S^T=B.)\nonumber\\
	&=\sum_{j=0}^{k}\tau_j\mathbf{v}_j^T\mathbf{f_1}_j+\sum_{j=0}^{k}\sum_{l=0}^{k}S_{jl}(\mathbf{v}_j-\mathbf{v}_l)^T\mathbf{f_{1}}^*(\mathbf{w}_j,\mathbf{w}_l)\nonumber\\
	&=\sum_{j=0}^{k}\tau_j\mathbf{v}_j^T\mathbf{f_1}_j+\sum_{j=0}^{k}\sum_{l=0}^{k}S_{jl}(\psi_j-\psi_l)\nonumber\\
	&=\sum_{j=0}^{k}\tau_j(\mathbf{v}_j^T\mathbf{f_1}_j-\psi_j). \;\text{(Using \eqref{prop2})}
	\end{flalign}
	Substituting in \eqref{en_exp}, we get,
	\begin{flalign*}
	\dfrac{d}{dt}\left(\sum_{j=0}^{k}\dfrac{\Delta x}{2}\omega_j\mathcal{U}_j\right)&=\sum_{j=0}^{k}\tau_j(\psi_j-\mathbf{v}_j^T\hat{\mathbf{f_1}}_j)\\
	&=(\psi_k-\mathbf{v}_k^T\hat{\mathbf{f_1}}_{i+1/2})-(\psi_0-\mathbf{v}_0^T\hat{\mathbf{f_1}}_{i-1/2})
	\end{flalign*}
	This shows that the scheme is entropy conservative within a single element.
\end{proof}

\par	For RHD equations we use the entropy conservative flux proposed in \cite{deepak2019entropy} which is given by,
\begin{align}
\mathbf{f_1}^{*(1)}&=\hat{\rho} \bar{\mu_x}\nonumber\\
\mathbf{f_1}^{*(2)}&=\dfrac{\bar{\mu_x}}{\bar{\Gamma}}\mathbf{f_1}^{*(4)}+\dfrac{\bar{\rho}}{\bar{\beta}}\nonumber\\
\mathbf{f_1}^{*(3)}&=\dfrac{\bar{\mu_y}}{\bar{\Gamma}}\mathbf{f_1}^{*(4)}\\
\mathbf{f_1}^{*(4)}&=-\dfrac{\bar{\Gamma}\left(k_1\hat{\rho}\bar{\mu_x}+\dfrac{\bar{\mu_x}\bar{\rho}}{\bar{\beta}}\right)}{\bar{\mu_x}^2+\bar{\mu_y}^2-\bar{\Gamma}^2}\nonumber
\end{align}
where $k_1=1+\dfrac{1}{\gamma-1}\dfrac{1}{\hat{\beta}}$, $\mu_x=\Gamma u_x$,  $\mu_y=\Gamma u_y$ and $\hat{a}:=\dfrac{a_R-a_L}{\ln a_R-\ln a_L}$.
Now the following statement follows immediately from the entropy conservation proof of the Theorem \ref{thm:ent_con}:
\begin{theorem}
	If the numerical flux $\hat{\mathbf{f_1}}$ at the element interface is entropy stable, then the scheme \eqref{DG_single_en} is entropy stable.
\end{theorem}
\begin{proof}
The entropy production rate at the interface is
\begin{flalign*}
(\psi_k^i-(\mathbf{v}_k^i)^T\hat{\mathbf{f_1}}(\mathbf{w}^{i}_k,\mathbf{w}^{i+1}_0))-(\psi_0^{i+1}-(\mathbf{v}_0^{i+1})^T\hat{\mathbf{f_1}}(\mathbf{w}^{i}_k,\mathbf{w}^{i+1}_0))\\
=(\mathbf{v}_0^{i+1}-\mathbf{v}_k^i)^T\hat{\mathbf{f_1}}(\mathbf{w}^{i}_k,\mathbf{w}^{i+1}_0)-(\psi_0^{i+1}-\psi_k^i).
\end{flalign*}
The expression clearly shows that the scheme \eqref{DG_single_en} is entropy stable if  the numerical flux $\hat{\mathbf{f_1}}$ at the element interface is entropy stable.
\end{proof}

\begin{remark}
	TVD/TVB limiters and bounds preserving limiters can be used for enhancing the performance of the scheme. We have utilized the bounds preserving limiter given in \cite{qin2016bound} to keep the solution in physical space. We recall from \cite{chen2017entropy} that the use of bounds preserving limiter will not increase the entropy.
\end{remark}
\section{Entropy stable DG schemes: Two dimensional schemes}
\label{sec:2dscheme}
We now present two dimensional schemes. We use the rectangular mesh $I_{i,j}=\left[x_{i-\frac{1}{2}},\,x_{i+\frac{1}{2}}\right]\times\left[y_{j-\frac{1}{2}},\,y_{j+\frac{1}{2}}\right]\, (1\leq i\leq N_x), (1\leq j\leq N_y)\,$ with mesh size $\Delta x_i$ and $\Delta y_j$ in $x$ and $y$ direction respectively. For simplicity we consider the same number of Gauss-Lobatto points ($k+1$) in both the direction. On applying the change of variables
 \begin{align*}
x_i(\xi)&=\dfrac{1}{2}(x_{i-1/2}+x_{i+1/2})+\dfrac{\xi}{2}\Delta x_i,\\
y_j(\xi)&=\dfrac{1}{2}(y_{j-1/2}+y_{j+1/2})+\dfrac{\xi}{2}\Delta y_j,
\end{align*}
the nodal values are denoted as $\mathbf{w}_{p,q}=\mathbf{w}_h(x_i(\xi_p),y_j(\xi_q))$.
Then for a single element $I_{i,j}$, the scheme is given by,
\begin{eqnarray}\label{2dscheme}
\dfrac{d \mathbf{w}_{p,q}}{dt}&=&-\dfrac{2}{\Delta x} \left(\sum_{l=0}^{k} 2 D_{pl}\mathbf{f_1}^{*}(\mathbf{w}_p,\mathbf{w}_l)-\dfrac{\tau_p}{\omega_p}(\mathbf{f_{1}}_{p,q}-\hat{\mathbf{f_1}}_{p,q})\right)\\
&&\qquad-\dfrac{2}{\Delta y}\left(\sum_{l=0}^{k} 2 D_{ql}\mathbf{f_2}^{*}(\mathbf{w}_q,\mathbf{w}_l)-\dfrac{\tau_q}{\omega_q}(\mathbf{f_{2}}_{p,q}-\hat{\mathbf{f_2}}_{p,q})\right),\,\,\,(p,q=0,1,\dots,k)\nonumber
\end{eqnarray}
where we have used the following notations by dropping the indices $i$ and $j$,
\begin{align*}
\mathbf{f_{d}}_{p,q}&=\mathbf{f_d}(\mathbf{w}_{p,q}),\,\,d=1,2\\
[\hat{\mathbf{f_1}}_{0,q},\hat{\mathbf{f_1}}_{1,q}\dots,\hat{\mathbf{f_1}}_{k,q}]&=[\hat{\mathbf{f_1}}_{i-1/2,q},0,\dots,0, \hat{\mathbf{f_1}}_{i+1/2,q}],\\
[\hat{\mathbf{f_2}}_{p,0},\hat{\mathbf{f_2}}_{p,1}\dots,\hat{\mathbf{f_2}}_{p,k}]&=[\hat{\mathbf{f_2}}_{p,j-1/2},0,\dots,0, \hat{\mathbf{f_2}}_{p,j+1/2}],\\
\hat{\mathbf{f_1}}_{i+1/2,q}&=\hat{\mathbf{f_1}}\left(\mathbf{w}_h(x_{i+1/2}^-,y_j(\xi_q)), \mathbf{w}_h(x_{i+1/2}^+,y_j(\xi_q)) \right),\\
\hat{\mathbf{f_2}}_{p,j+1/2}&=\hat{\mathbf{f_2}}\left(\mathbf{w}_h(x_i(\xi_p),y_{j+1/2}^-), \mathbf{w}_h(x_i(\xi_p),y_{j+1/2}^+) \right).
\end{align*}

\section{Numerical results}
\label{sec:num}
We will now present the numerical results for $k=1$, which is a second-order scheme (denoted by ESDG-O2) and $k=2$, which results in a third-order numerical scheme (denoted by ESDG-O3). We use Lax-Friedrich flux at the cell interface. For the time integration, we use strong stability preserving (SSP) Runge-Kutta method \cite{gottlieb2001strong}.  In all test cases, we use CFL number 0.1 for consistency. We also set gas constant $\gamma=\dfrac{5}{3}$, unless stated in the particular test case. Furthermore, the TVDM limiter is used with $M=10$ at every stage of Runge-Kutta time update to prevent the oscillations in solution.

\subsection{One dimensional numerical tests}
We first proceed with one dimensional test cases. First we test accuracy of the schemes. We then test the proposed schemes on various Riemann problems.
\begin{example}[Accuracy test:]
	\label{test1}
	To check the accuracy of the proposed schemes, we first consider a test case with smooth exact solution. We consider $[0,1]$ as computational domain with periodic boundary conditions. The initial condition are given as follows:
	\begin{equation*}
	\left(\rho,\,u,\,p\right)=\left(2 + \sin(2 \pi x),0.5,1\right).
	\end{equation*}
	The exact smooth solution of the problem is advection of rest-mass density, i.e. $\rho=2 + \sin(2 \pi (x-0.5t))$. Other variables, remain unchanged. We compute the solution till final time $t=2$. 
	
	The numerical errors ($L^1$ and $L^\infty$) for the $\rho$ variable are presented in  Table \ref{tab:accrk1k2}. We note that both schemes have reached the desired order of accuracy, even at the coarse mesh of 256 cells. Furthermore, the schemes are highly more accurate when compared with finite-difference entropy stable schemes in \cite{deepak2019entropy}.
	\begin{table}[htb!]
		\centering
		\begin{tabular}{@{\extracolsep{4pt}}crrrrrrrr@{}}
			\toprule
			Number of &\multicolumn{4}{c}{ESDG-O2}&\multicolumn{4}{c}{ESDG-O3} \\
			\cline{2-5} \cline{6-9}
			cells & $L^1$ error  &  Order &  $L^\infty$ error & Order & $L^1$ error  &  Order &  $L^\infty$ error & Order\\
			\midrule
			32 & 4.97e-02 &  ...  & 4.59e-02 & ... &4.03e-04 &  ...  & 4.31e-04 & ... \\
			64 & 1.24e-02 &  2.00 & 1.09e-02 & 2.08 & 5.24e-05 &  2.94 & 5.52e-05 & 2.96 \\ 
			128 & 3.04e-03 &  2.03 & 2.61e-03 & 2.06 &6.66e-06 &  2.98 & 6.96e-06 & 2.99 \\  
			256 & 7.30e-04 &  2.06 & 6.31e-04 & 2.05 & 8.38e-07 &  2.99 & 8.71e-07 & 3.00 \\  
			512 & 1.71e-04 &  2.10 & 1.51e-04 & 2.07 &1.05e-07 &  3.00 & 1.09e-07 & 3.00 \\ 
			1024 & 3.81e-05 &  2.16 & 3.50e-05 & 2.11 &1.32e-08 &  3.00 & 1.36e-08 & 3.00 \\  
			\bottomrule
		\end{tabular} 
		\caption{Test problem 1 (Accuracy test): $L^1$ and $L^\infty$ errors for $\rho$ at various resolutions using ESDG-O2 and ESDG-O3 schemes}
		\label{tab:accrk1k2}
	\end{table}
\end{example}

\begin{example}[Isentropic Smooth Flows:]
	\label{test2}
	We consider another  smooth test case from \cite{zhang2006ram,marti2003numerical}. The initial profile consist of a pulse of width $L=0.3$ and amplitude $\alpha=1$ inside the domain $[-0.35,1]$ over the reference state $(\rho_{ref}, u_{ref}, p_{ref})=(1,0,100)$. The initial density profile is given by     
	\begin{equation*}
	\rho=\rho_{ref} \left(1+\alpha f(x) \right),
	\end{equation*}
	where 
	\begin{equation*}
	f(x)=\begin{cases}
	\left(\left(\dfrac{x}{L}\right)^2-1\right)^4 & \text{if } |x|<L\\
	0 & \text{otherwise}.
	\end{cases}
	\end{equation*}
	Initial pressure is taken as $p=K\rho^\gamma$, where $K$ is constant. The velocity inside the pulse is set such that the Riemann invariant, 
	\begin{equation*}
	J_{-}=\dfrac{1}{2}\ln\left(\dfrac{1+u}{1-u}\right)-\dfrac{1}{\sqrt{\gamma-1}}\ln\left(\dfrac{\sqrt{\gamma-1}+c}{\sqrt{\gamma-1}-c}\right)
	\end{equation*}
	remains constant throughout the region, where $c$ is the sound speed. Exact solution is calculated using the standard characteristic analysis.     
	\begin{figure}
		\centering
		\begin{subfigure}[b]{0.45\textwidth}
			\includegraphics[width=\textwidth]{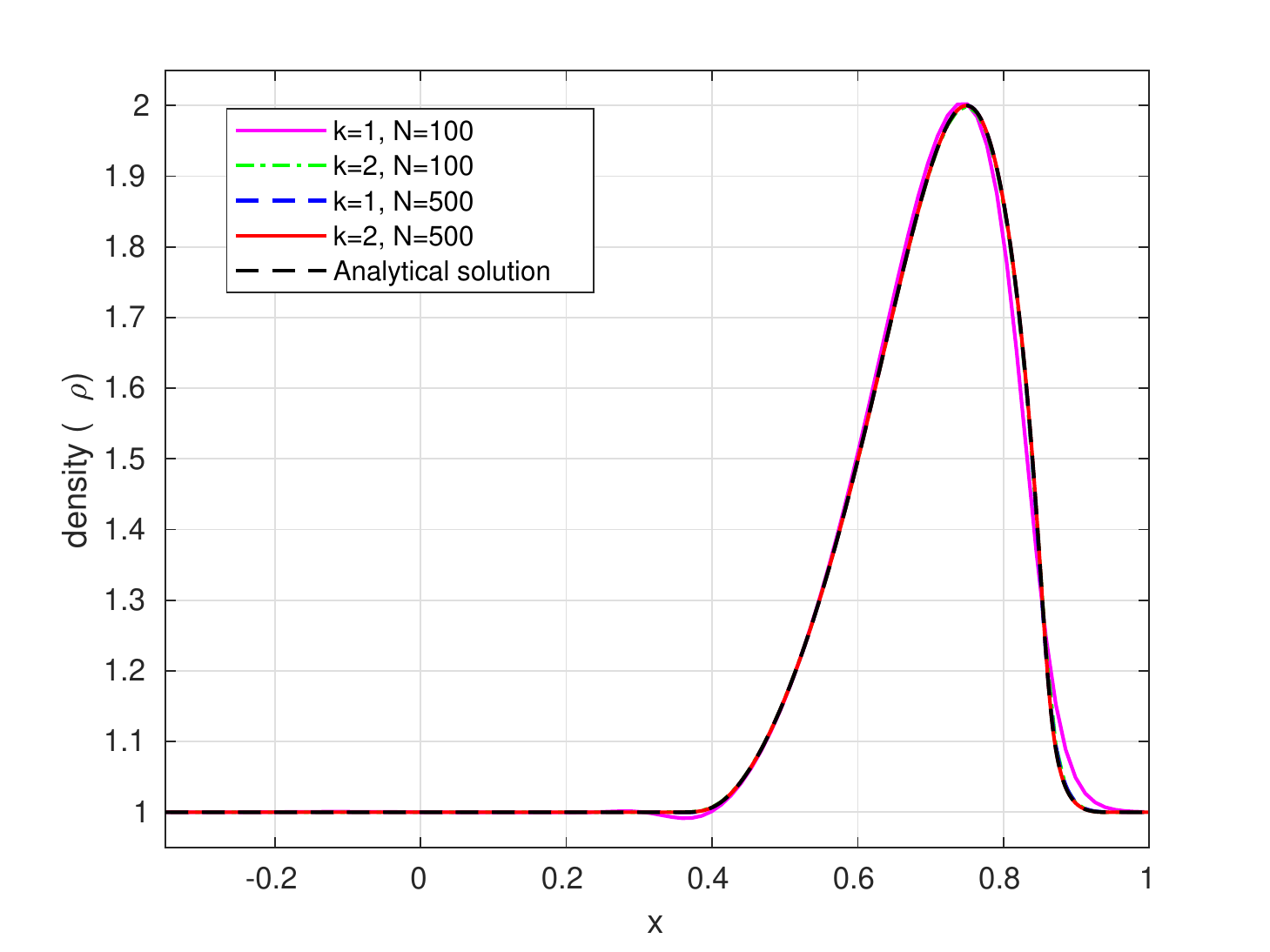}
			\caption{Density}
		\end{subfigure}
		\begin{subfigure}[b]{0.45\textwidth}
			\includegraphics[width=\textwidth]{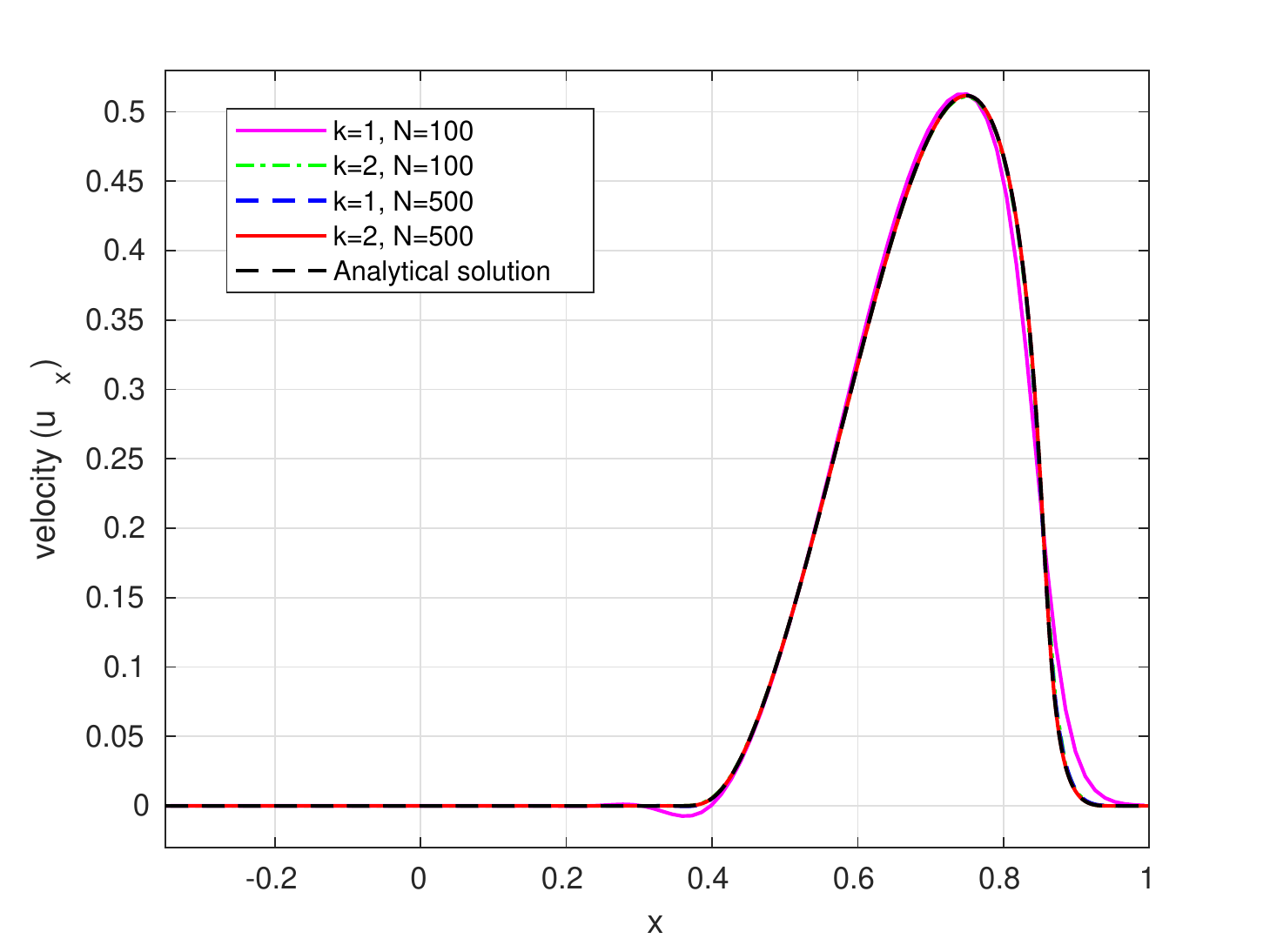}
			\caption{Velocity}
		\end{subfigure}
		\begin{subfigure}[b]{0.45\textwidth}
			\includegraphics[width=\textwidth]{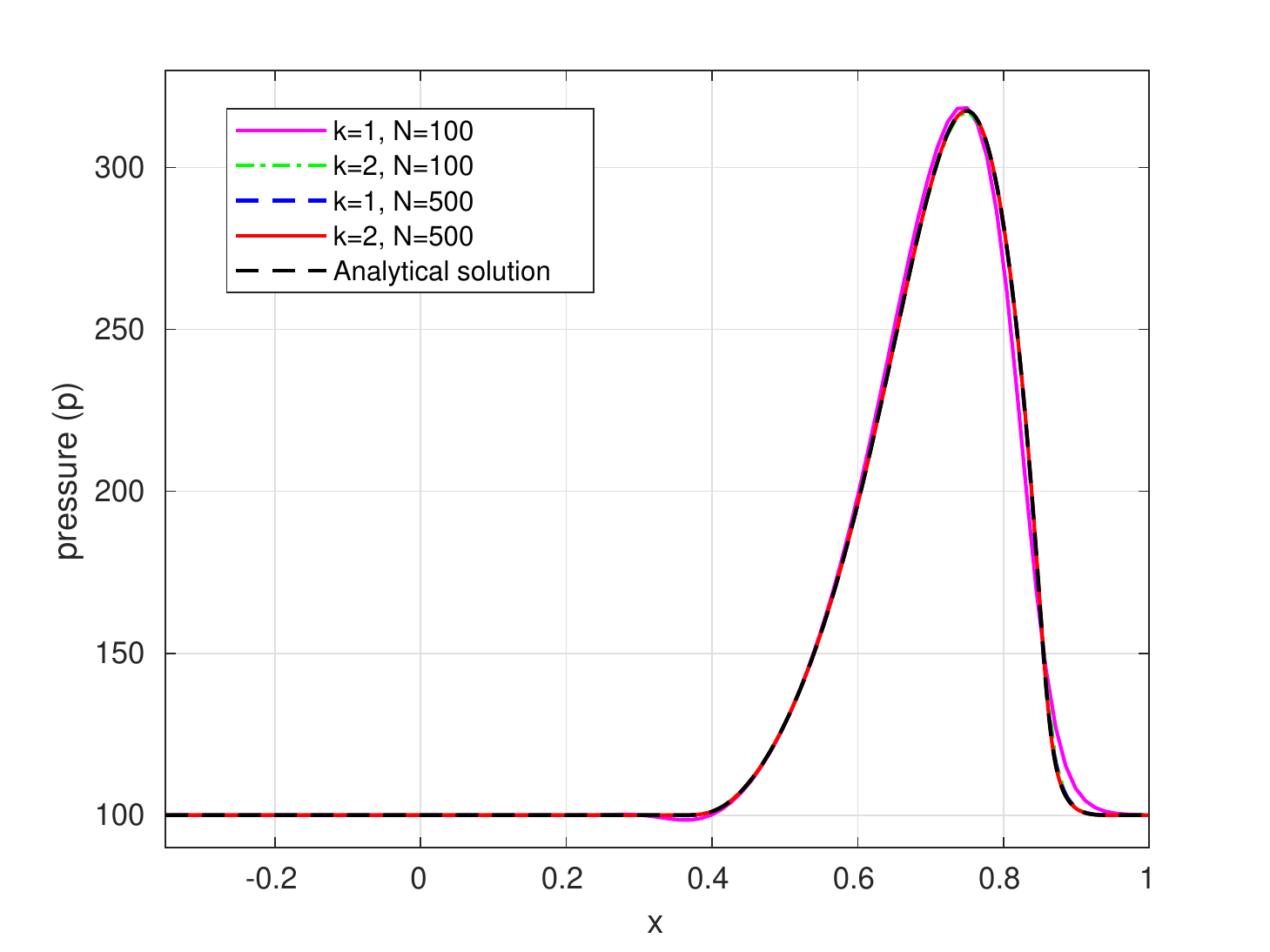}
			\caption{Pressure}
		\end{subfigure}
		\begin{subfigure}[b]{0.45\textwidth}
			\includegraphics[width=\textwidth]{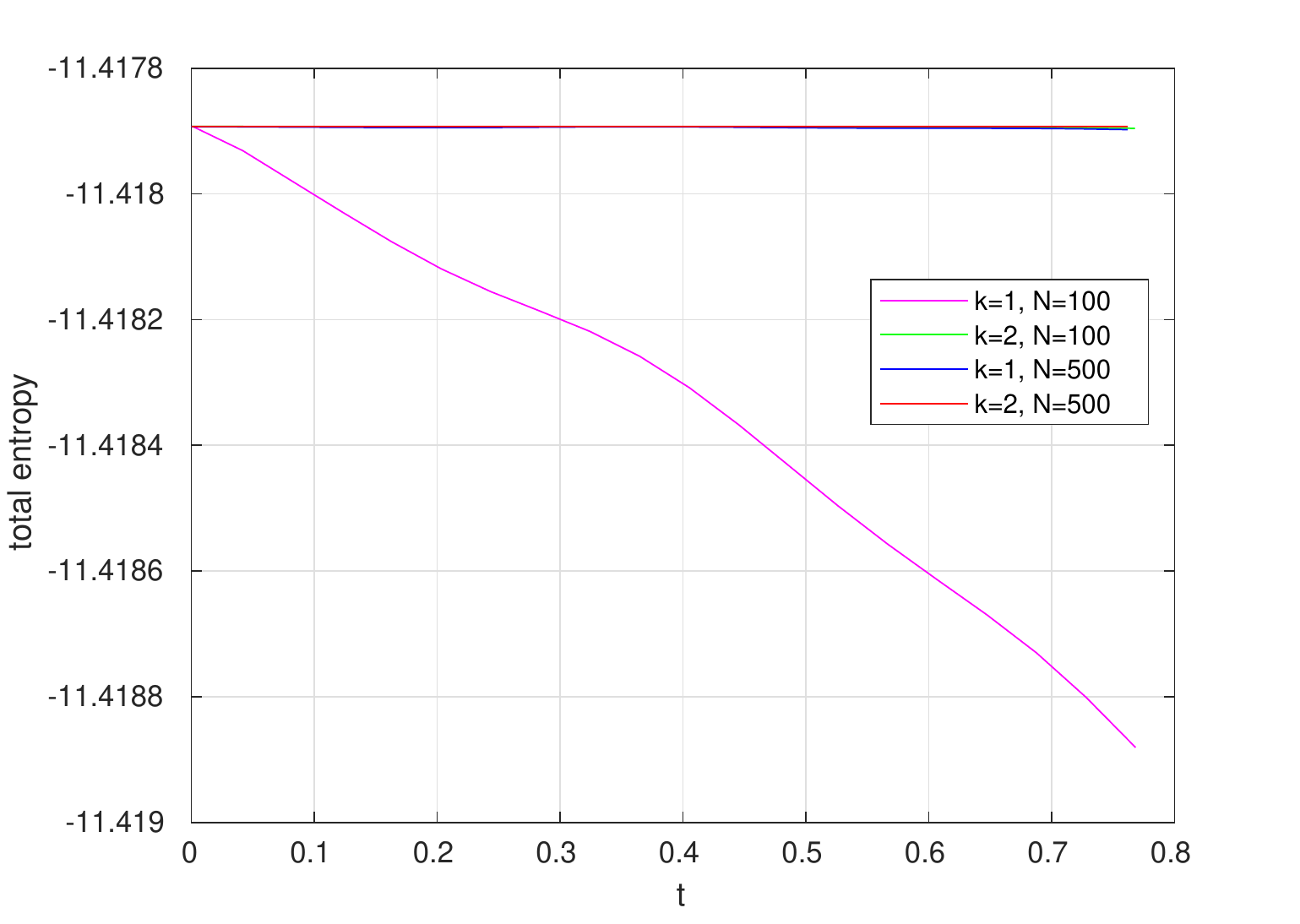}
			\caption{Evolution of total entropy}
		\end{subfigure}
		
		\caption{Test Problem \ref{test2} (Isentropic Smooth Flows): Plots of density, velocity, pressure and total entropy evolution for ESDG-O2(k=1) and ESDG-O3(k=2) using 100 and 500 cells.}
		\label{fig:pb19}
	\end{figure}
	\begin{table}[htb!]
		\centering
		\begin{tabular}{ccccccccc}
			\toprule
			Number of &\multicolumn{4}{c}{ESDG-O2}&\multicolumn{4}{c}{ESDG-O3} \\
			\cmidrule(r){2-5} \cmidrule(r){6-9}
			cells & $L^1$ error  &  Order &  $L^\infty$ error & Order & $L^1$ error  &  Order &  $L^\infty$ error & Order\\
			\midrule
			32 & 7.15e-02 &  ...  & 1.18e-01 & ... & 4.95e-03 &  ...  & 2.52e-02 & ... \\  
			64 & 2.25e-02 &  1.67 & 7.04e-02 & 0.75 & 8.52e-04 &  2.54 & 8.47e-03 & 1.57  \\  
			128 & 6.83e-03 &  1.72 & 3.58e-02 & 0.98 & 1.24e-04 &  2.78 & 1.46e-03 & 2.53 \\ 
			256 & 1.79e-03 &  1.93 & 1.66e-02 & 1.11 & 1.32e-05 &  3.24 & 2.26e-04 & 2.70 \\ 
			512 & 4.17e-04 &  2.10 & 5.98e-03 & 1.48 & 1.48e-06 &  3.16 & 2.39e-05 & 3.24 \\ 
			1024 & 9.16e-05 &  2.19 & 1.72e-03 & 1.80 & 1.92e-07 &  2.94 & 3.66e-06 & 2.70 \\  
			\bottomrule
		\end{tabular} 
		\caption{Test problem 2 (isentropic smooth flows): $L^1$ and $L^\infty$ errors for $\rho$ at various resolutions using ESDG-O2 and ESDG-O3 schemes}
		\label{tab:acc_isen_k1k2}
	\end{table}
	
	In Figure \ref{fig:pb19}, we present the computational results of ESDG-O2 and ESDG-O3 at time $t=0.8$. We note that the ESDG-O2 at 100 cells is slightly more diffusive, but when using 500 cells, it matches the exact solution. On the other hand, ESDG-O3 matches the exact solution even at 100 cells. This can also be seen from the entropy decay plots, where we note that both ESDG-O2 and ESDG-O3 produce no unnecessary dissipation at both resolutions.
	
	To test the accuracy, In Table \ref{tab:acc_isen_k1k2}, we have presented  $L^1$ and $L^\infty$ errors for $\rho$ at various resolutions. We note that both schemes have the desired order of accuracy in $L^1$ norm. However, in the $L^\infty$ norm ESDG-O2 scheme reaching accuracy only at higher resolutions, whereas ESDG-O3 is of desired order of accuracy. This performance is in complete contrast of finite-difference entropy stable schemes (see \cite{deepak2019entropy}) and other finite volume based schemes (see \cite{zhang2006ram}). There we note that the third and fourth-order schemes do not reach more than second-order accuracy in $L^1$ norm, even at very high resolutions. It demonstrates the very high accuracy of the proposed schemes when compared to standard finite difference and finite volume schemes.
\end{example}

\begin{example}[Riemann Problem 1:]
	\label{test3}
In this test case, we consider a Riemann problem from \cite{Mignone2005HLLC}. We use computational domain of $[0,1]$ with initial discontinuity at $x=0.5$. The states are given by,
	\begin{equation*}
	\left(\rho,\,u,\,p\right)=\begin{cases}
	\left(1,-0.6,10\right) & \text{if $x<0.5$}\\
	\left(10,0.5,20\right) &  \text{if $x>0.5$}
	\end{cases}.
	\end{equation*}
The exact solution contains two rarefaction waves moving in opposite direction separated by a contact wave. We use outflow boundary conditions.
	\begin{figure}
	\centering
	\begin{subfigure}[b]{0.45\textwidth}
		\includegraphics[width=\textwidth]{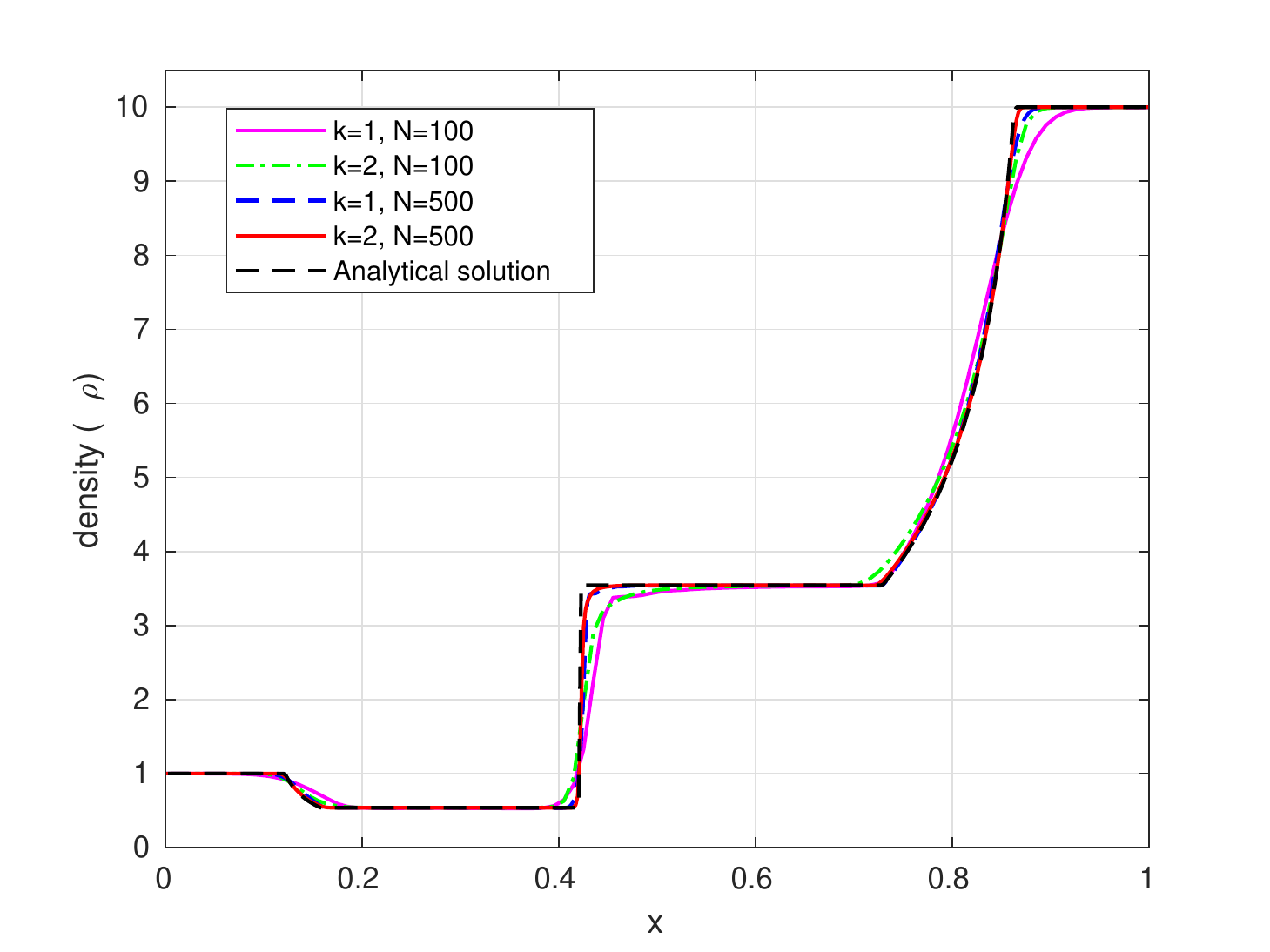}
		\caption{Density}
	\end{subfigure}
	\begin{subfigure}[b]{0.45\textwidth}
	\includegraphics[width=\textwidth]{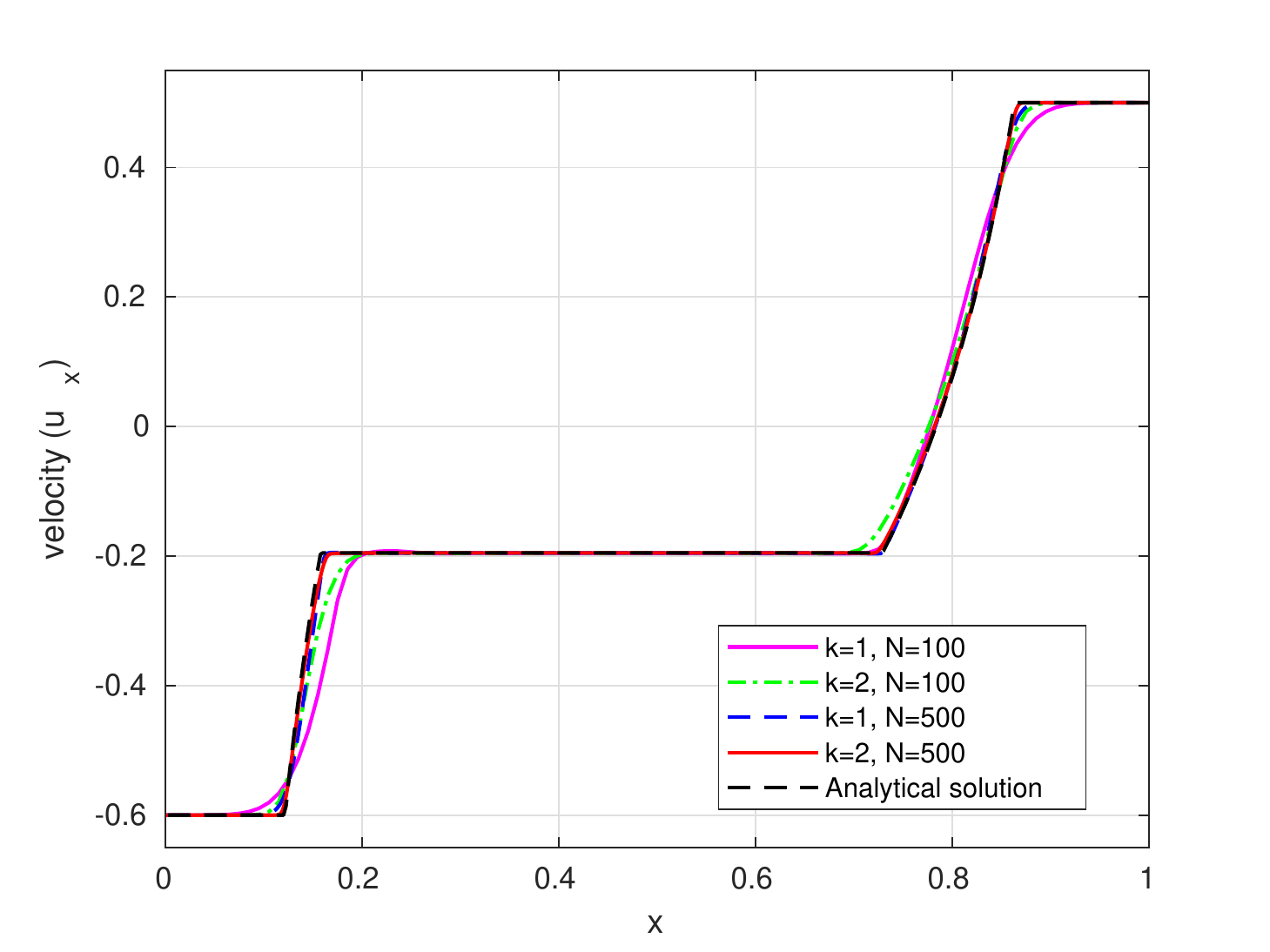}
	\caption{Velocity}
\end{subfigure}
	\begin{subfigure}[b]{0.45\textwidth}
		\includegraphics[width=\textwidth]{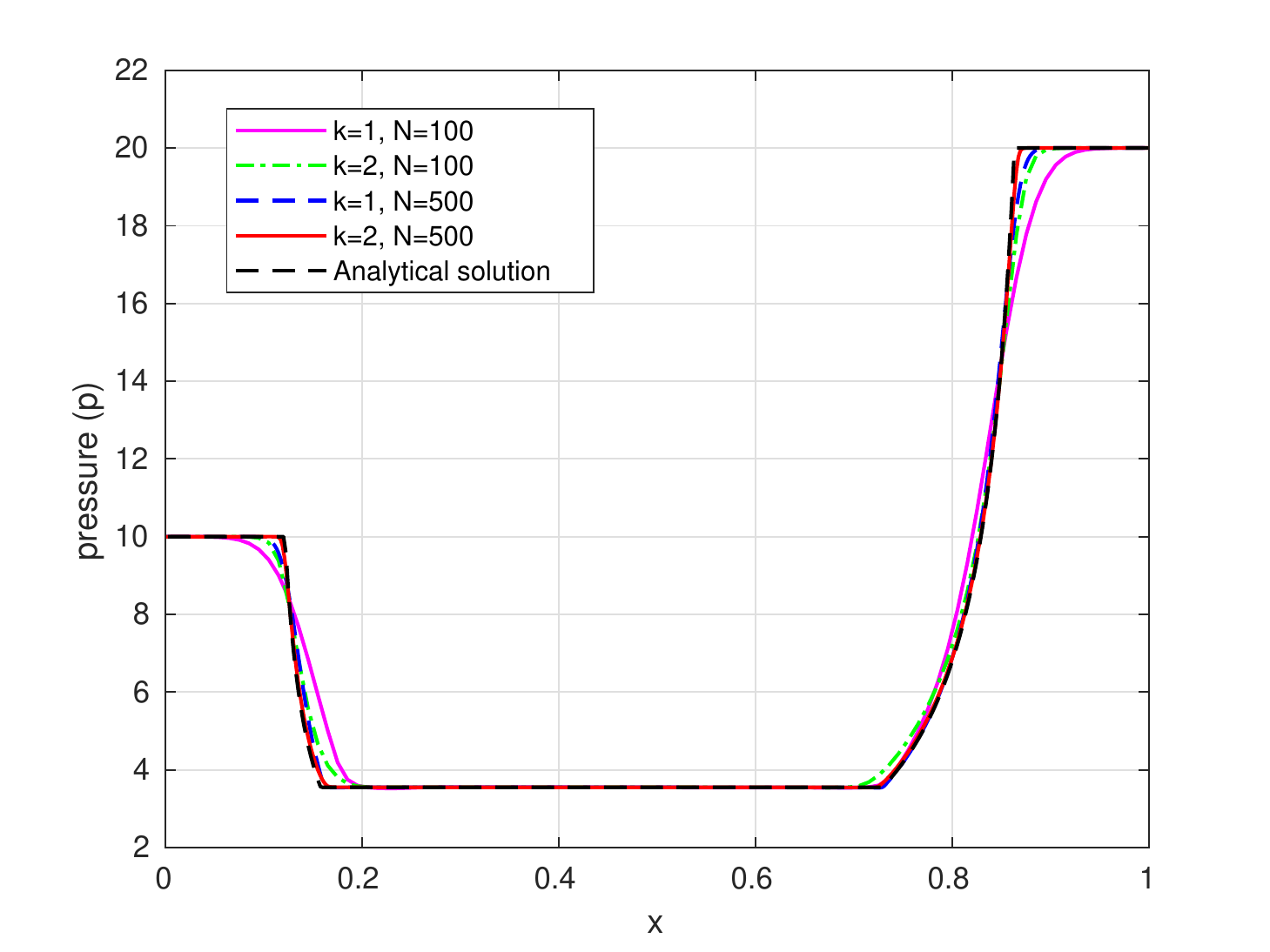}
		\caption{Pressure}
	\end{subfigure}
	\begin{subfigure}[b]{0.45\textwidth}
		\includegraphics[width=\textwidth]{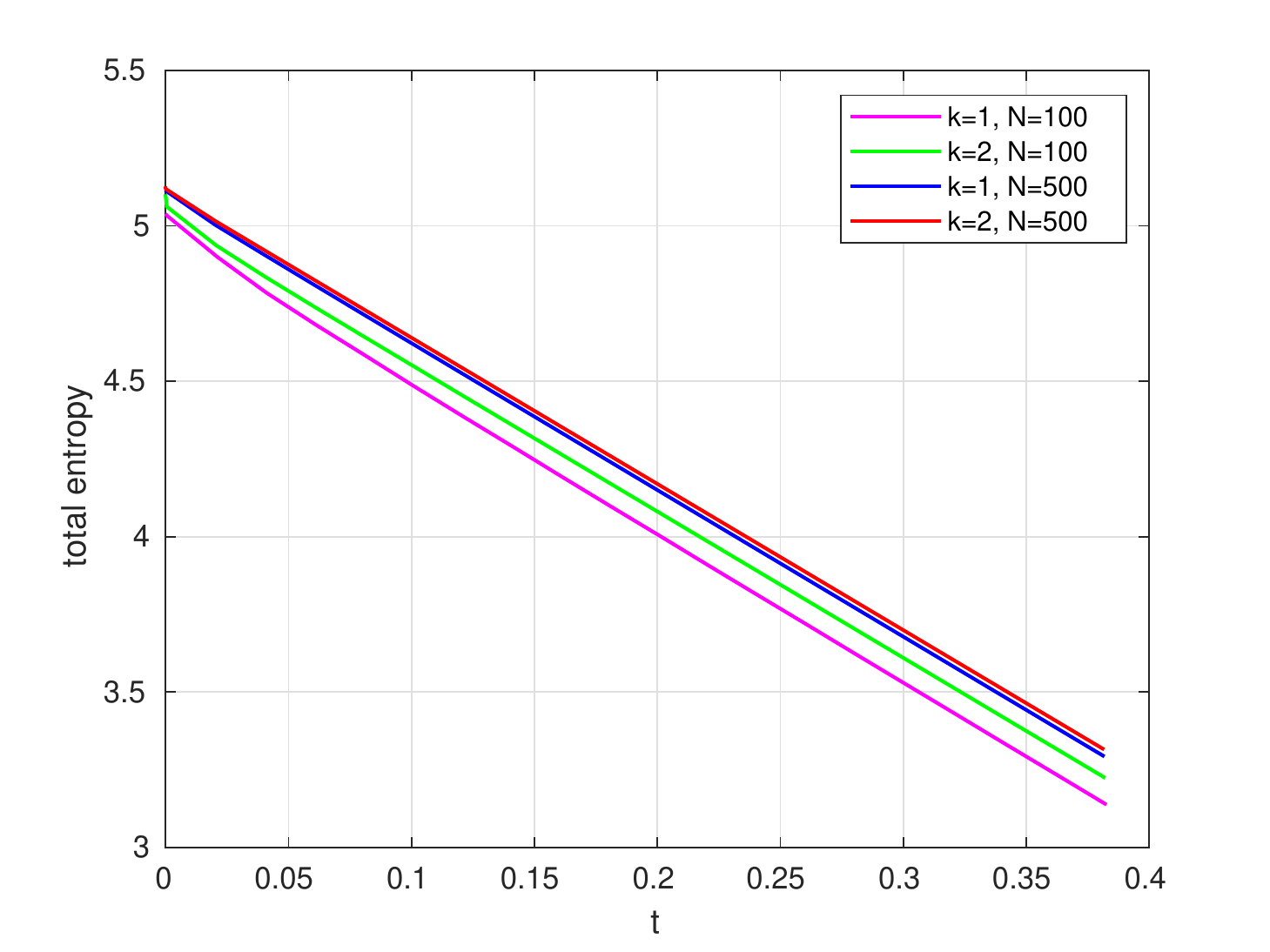}
		\caption{Evolution of total entropy}
	\end{subfigure}
	\caption{Test Problem \ref{test3} (Riemann Problem 1): Plots of density, velocity, pressure and total entropy evolution for ESDG-O2(k=1) and ESDG-O3(k=2) using 100 and 500 cells.}
	\label{fig:pb3}
\end{figure}
The numerical solutions are presented in Figure \ref{fig:pb3} at the final time of $0.4$ for both ESDG-O2 and ESDG-O3 numerical schemes. At the resolution of $100$ cells, we note that the ESDG-O3 scheme is more accurate than the ESDG-O2 scheme. At the resolution of $500$ cells, both schemes have matched the exact solution. Furthermore, both schemes can capture rarefaction and contact waves. We have also plotted total entropy decay. At a coarse mesh of $100$ cells, ESDG-O2 decays more entropy compared to the ESDG-O3 scheme, whereas at the finer mesh of $500$ cells, both schemes have similar entropy decay performance.
\end{example}

\begin{example}[Riemann Problem 2:]
	\label{test4}
We consider a shock tube Riemann problem  from \cite{marti2003numerical}. The computational domain is $[0,1]$ with outflow boundary conditions. The initial conditions are given by,
	\begin{equation}\label{rp2}
	\left(\rho,\,u,\,p\right)=\begin{cases}
	\left(1,0,10^3\right) & \text{if $x<0.5$}\\
	\left(1,0,10^{-2}\right) &  \text{if $x>0.5$}
	\end{cases}.
	\end{equation}
The exact solutions contain all kinds of waves; hence, it is a suitable example to test the proposed schemes' wave-capturing ability. 

In Figure \ref{fig:pb6}, we have plotted density, velocity and pressure using $100$ and $500$ cells for both schemes at time $t=0.4$. At a coarse mesh of $100$ cells, both ESDG-O2 and ESDG-O3 has similar performance with ESDG-O3 being slightly more accurate. At the fine mesh of $500$ cells, both schemes are very close to the exact solution, with ESDG-O3 again slightly more accurate than the ESDG-O2 scheme. Both the schemes are able to capture all the waves at both resolutions. This can also be seen from the total entropy evolution plot. We observe both schemes perform similarly at both resolutions with ESDG-O2 being more entropy diffusive.
	
	\begin{figure}
		\centering
		\begin{subfigure}[b]{0.45\textwidth}
			\includegraphics[width=\textwidth]{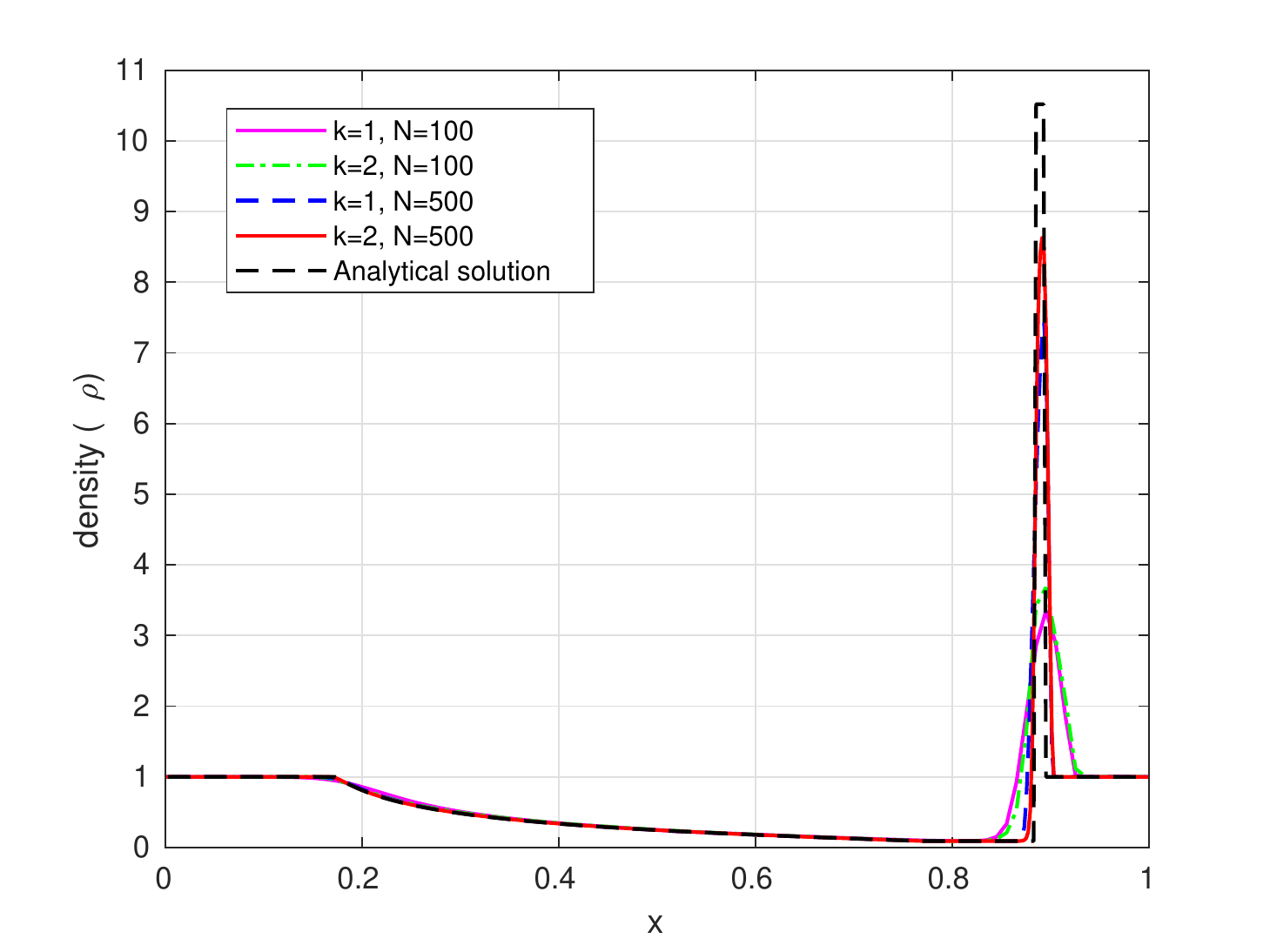}
			\caption{Density}
		\end{subfigure}
		\begin{subfigure}[b]{0.45\textwidth}
		\includegraphics[width=\textwidth]{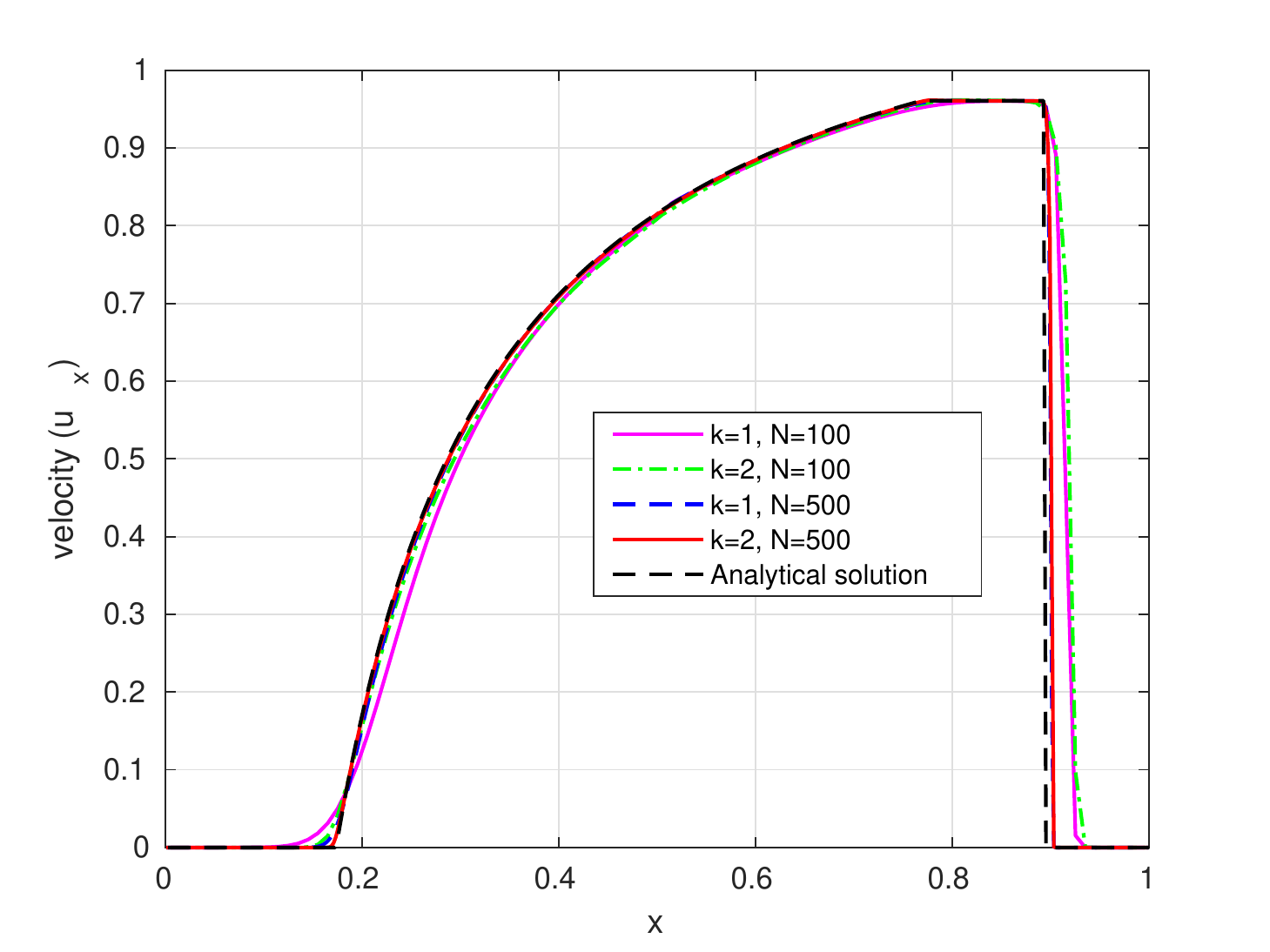}
		\caption{Velocity}
	\end{subfigure}
		\begin{subfigure}[b]{0.45\textwidth}
			\includegraphics[width=\textwidth]{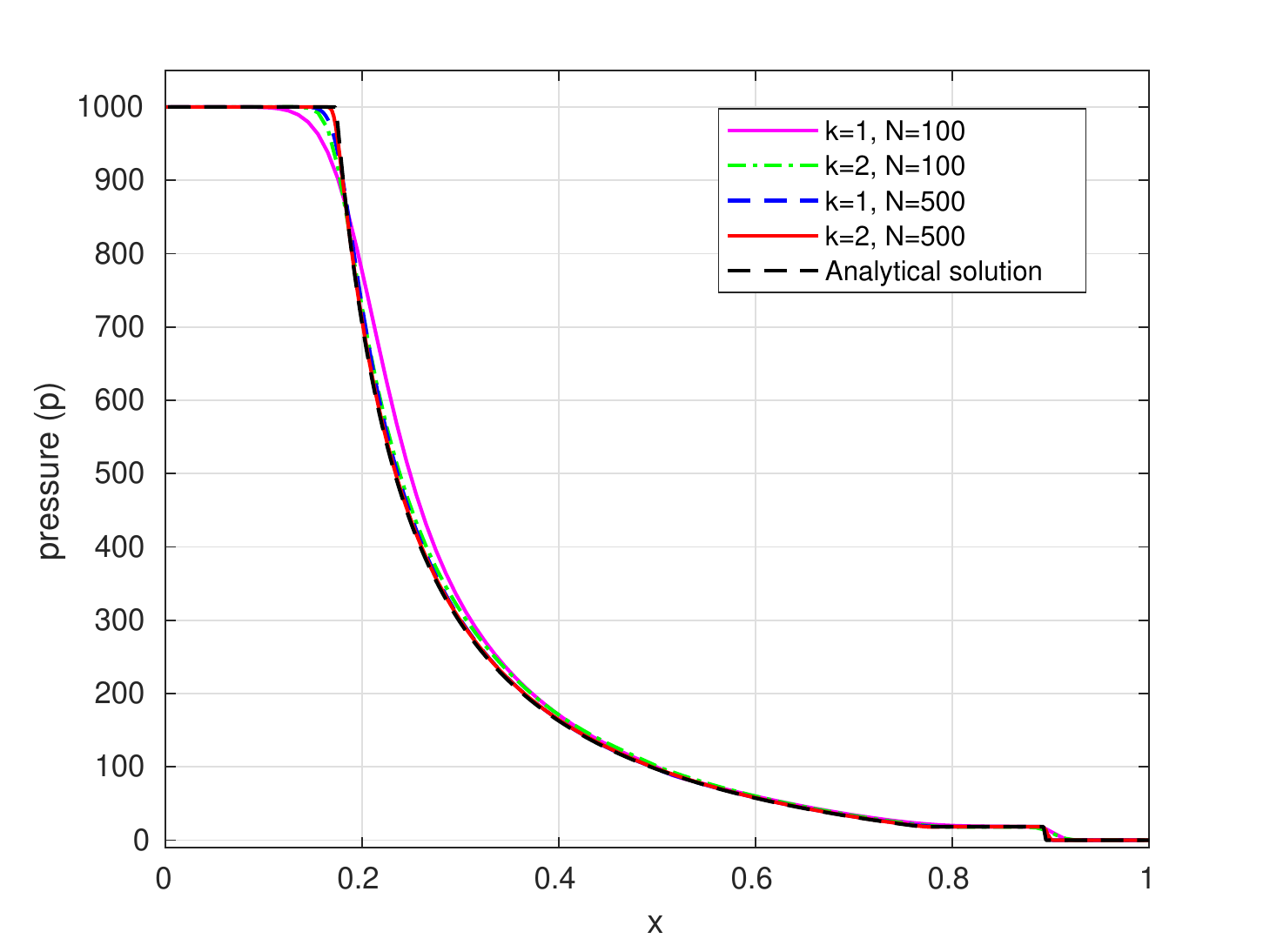}
			\caption{Pressure}
		\end{subfigure}
		\begin{subfigure}[b]{0.45\textwidth}
			\includegraphics[width=\textwidth]{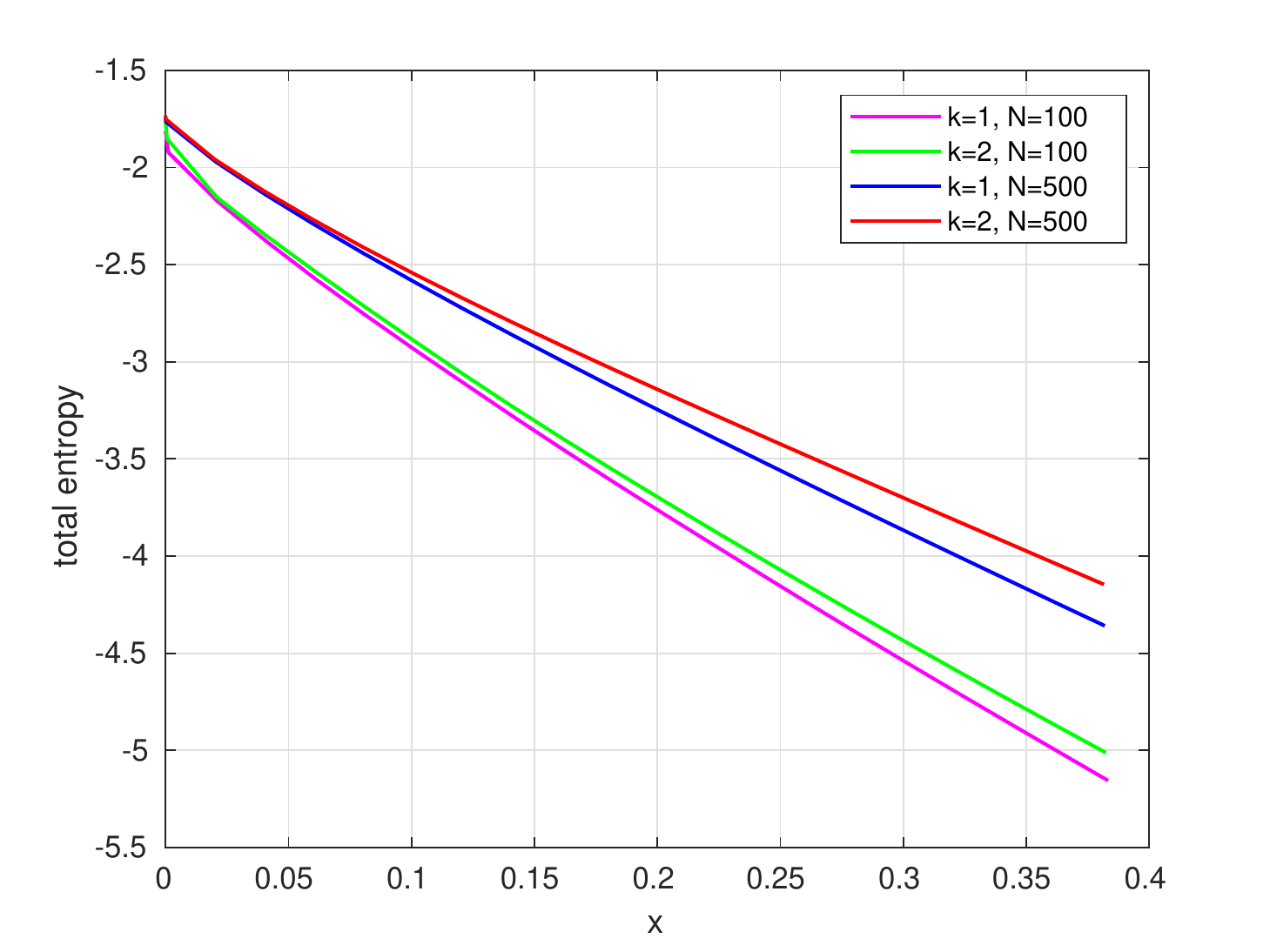}
			\caption{Evolution of total entropy}
		\end{subfigure}
		\caption{Test Problem \ref{test4} (Riemann Problem 2): Plots of density, velocity, pressure and total entropy evolution for ESDG-O2(k=1) and ESDG-O3(k=2) using 100 and 500 cells.}
		\label{fig:pb6}
	\end{figure}
\end{example}

\begin{example}[Riemann Problem 3:]
\label{test5}
We consider another Riemann problem form\cite{Mignone2005HLLC}.  On domain $[0,1]$, the states are given by,
	\begin{equation}\label{rp3}
	\left(\rho,\,u,\,p\right)=\begin{cases}
	\left(10,0,\dfrac{40}{3}\right) & \text{if $x<0.5$}\\
	\left(1,0,\dfrac{2}{3}\times10^{-6}\right) &  \text{if $x>0.5$}
	\end{cases}.
	\end{equation}
	We use outflow boundary conditions. The problem contains a very low pressure area. So, this Riemann problem will test the robustness of the numerical schemes. The computation results presented in Figure \ref{fig:pb7} at time $t=0.4$  using $100$ and $500$ cells. 
	
	\begin{figure}[htb!]
		\centering
		\begin{subfigure}[b]{0.45\textwidth}
			\includegraphics[width=\textwidth]{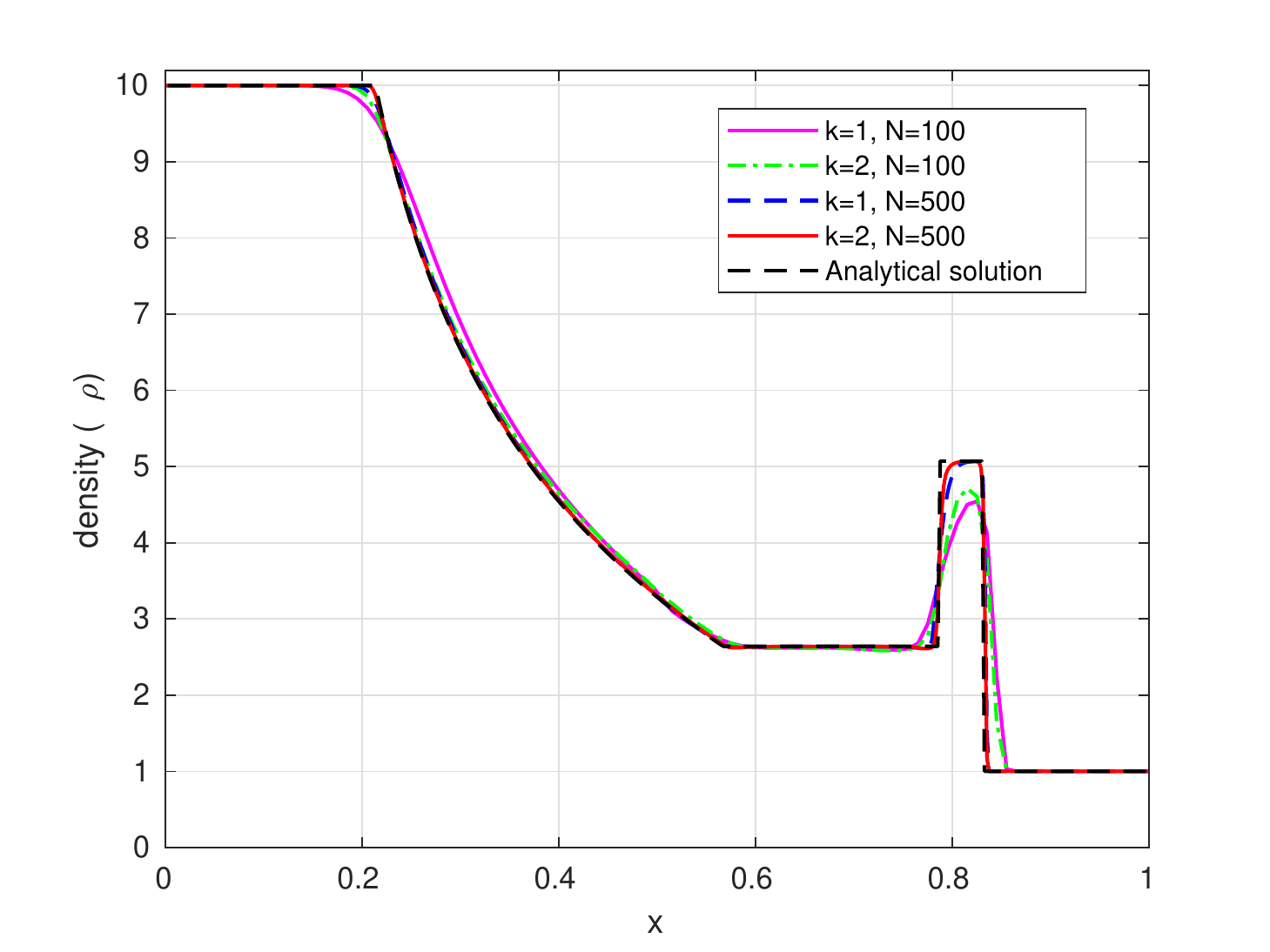}
			\caption{Density}
		\end{subfigure}
	\begin{subfigure}[b]{0.45\textwidth}
		\includegraphics[width=\textwidth]{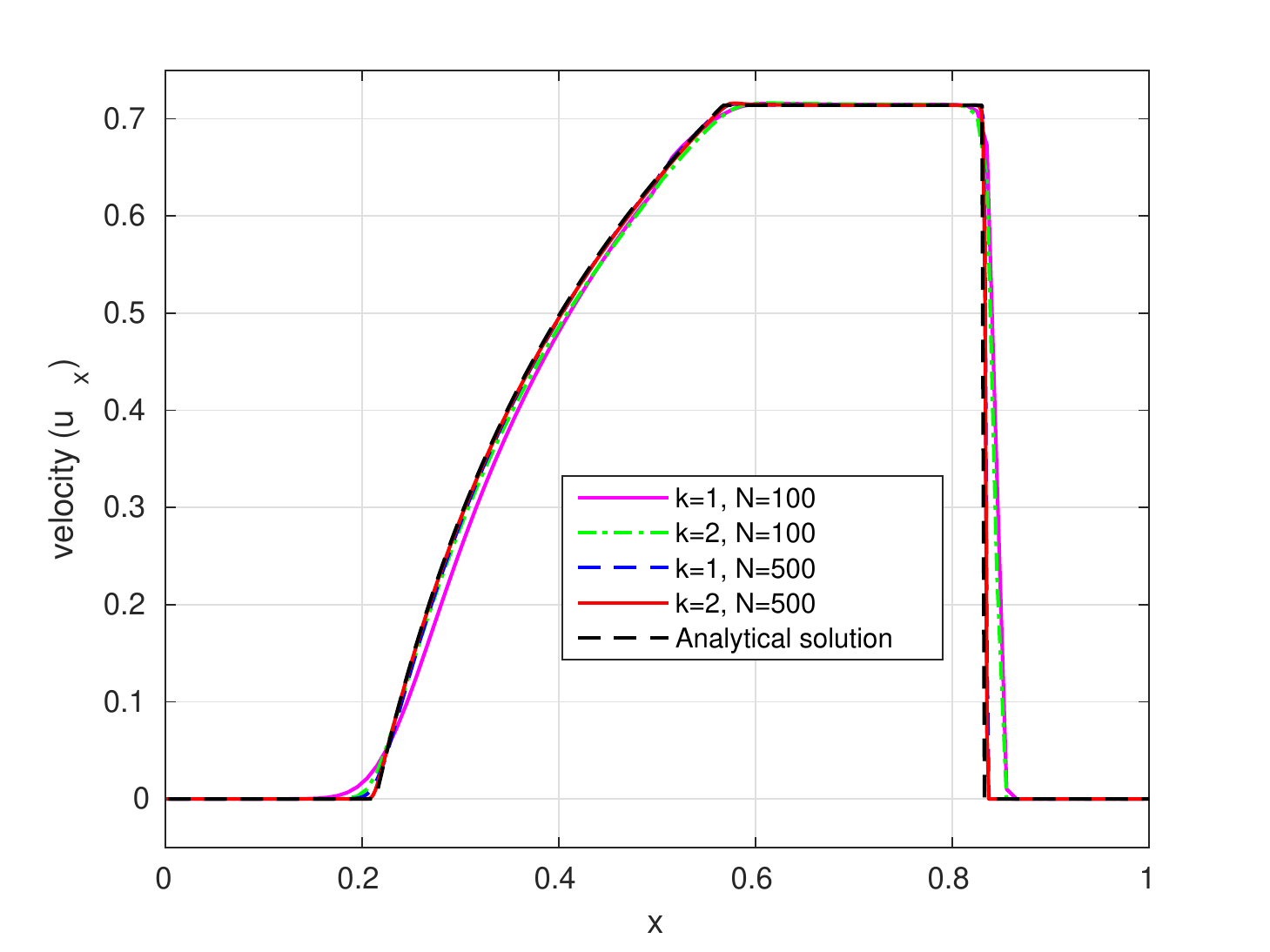}
		\caption{Velocity}
	\end{subfigure}
		\begin{subfigure}[b]{0.45\textwidth}
			\includegraphics[width=\textwidth]{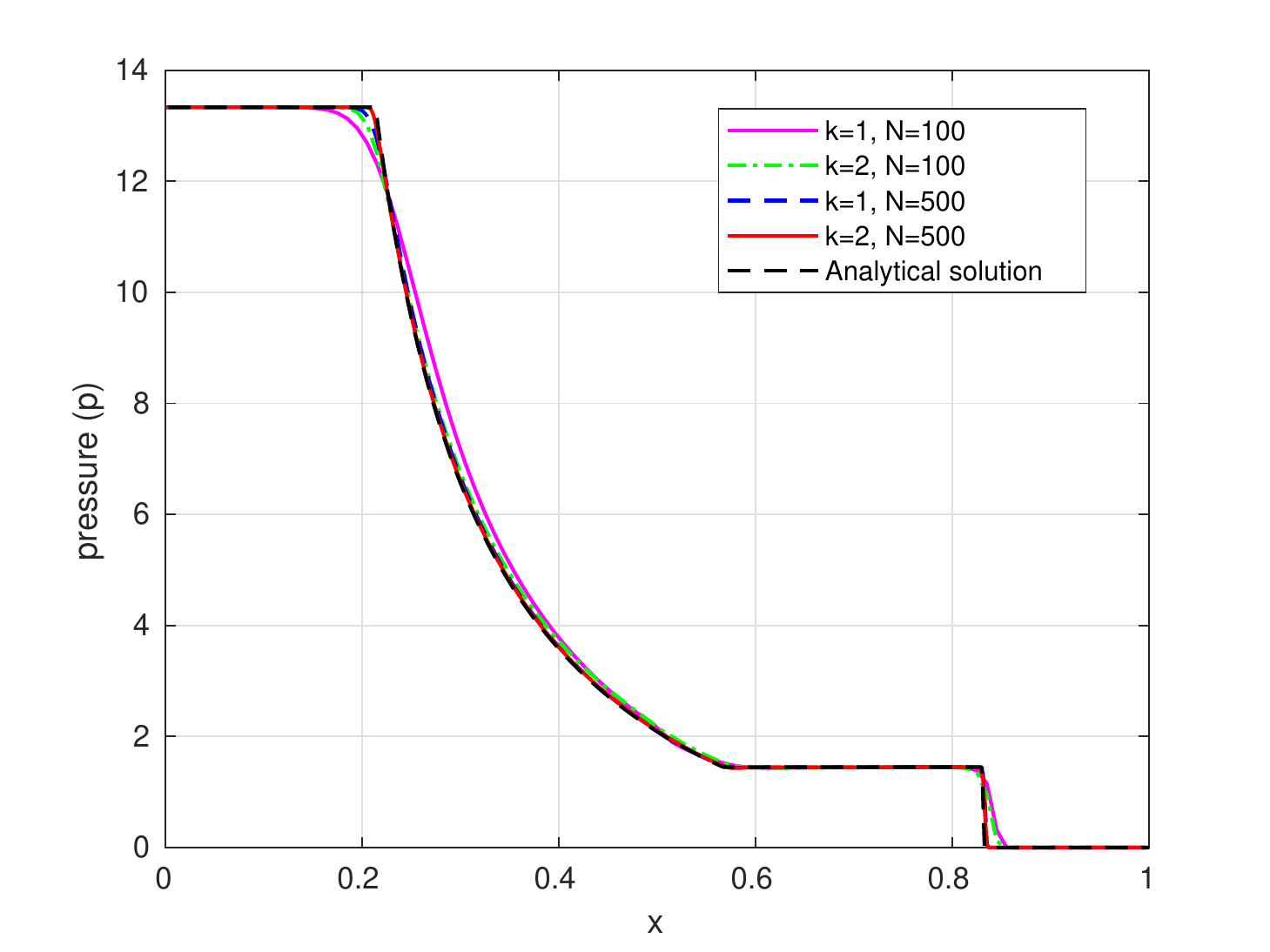}
			\caption{Pressure}
		\end{subfigure}
		\begin{subfigure}[b]{0.45\textwidth}
			\includegraphics[width=\textwidth]{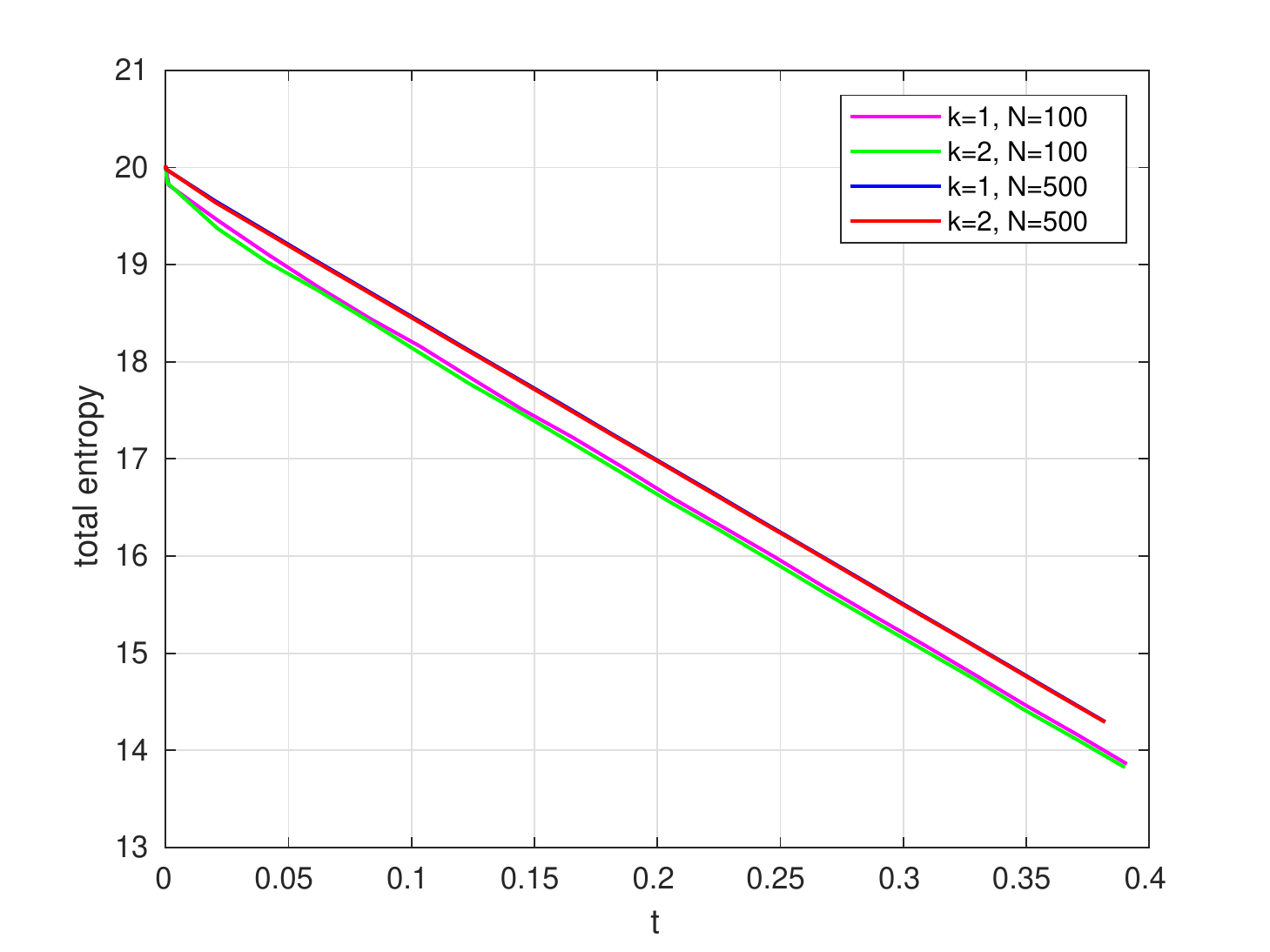}
			\caption{Evolution of total entropy}
		\end{subfigure}
		\caption{Test Problem \ref{test5} (Riemann Problem 3): Plots of density, velocity, pressure and total entropy evolution for ESDG-O2(k=1) and ESDG-O3(k=2) using 100 and 500 cells.}
		\label{fig:pb7}
	\end{figure}
We observe that both schemes are stable and have a similar performance at $100$ cells. At $500$ cells, both schemes are very close to the exact solution and have similar wave capturing ability. We also observe that both schemes have similar entropy decay rates at both resolutions.

\end{example}

\begin{example}[Riemann Problem 4:]
\label{test6}
In this test case, we again consider another Riemann problem from \cite{Mignone2005HLLC}. The computational domain is set to be $[0,1]$, with initial condition, 
	\begin{equation}\label{}
	\left(\rho,\,u,\,p\right)=\begin{cases}
	\left(1,0.9,1\right) & \text{if $x<0.5$}\\
	\left(1,0,10\right) &  \text{if $x>0.5$}
	\end{cases}.
	\end{equation}
The exact solution contains two shock waves and a contact discontinuity. Numerical solutions are plotted in Figure \ref{fig:pb9}  at time $t=0.4$ using outflow boundary conditions.  
	\begin{figure}[htb!]
		\centering
		\begin{subfigure}[b]{0.45\textwidth}
			\includegraphics[width=\textwidth]{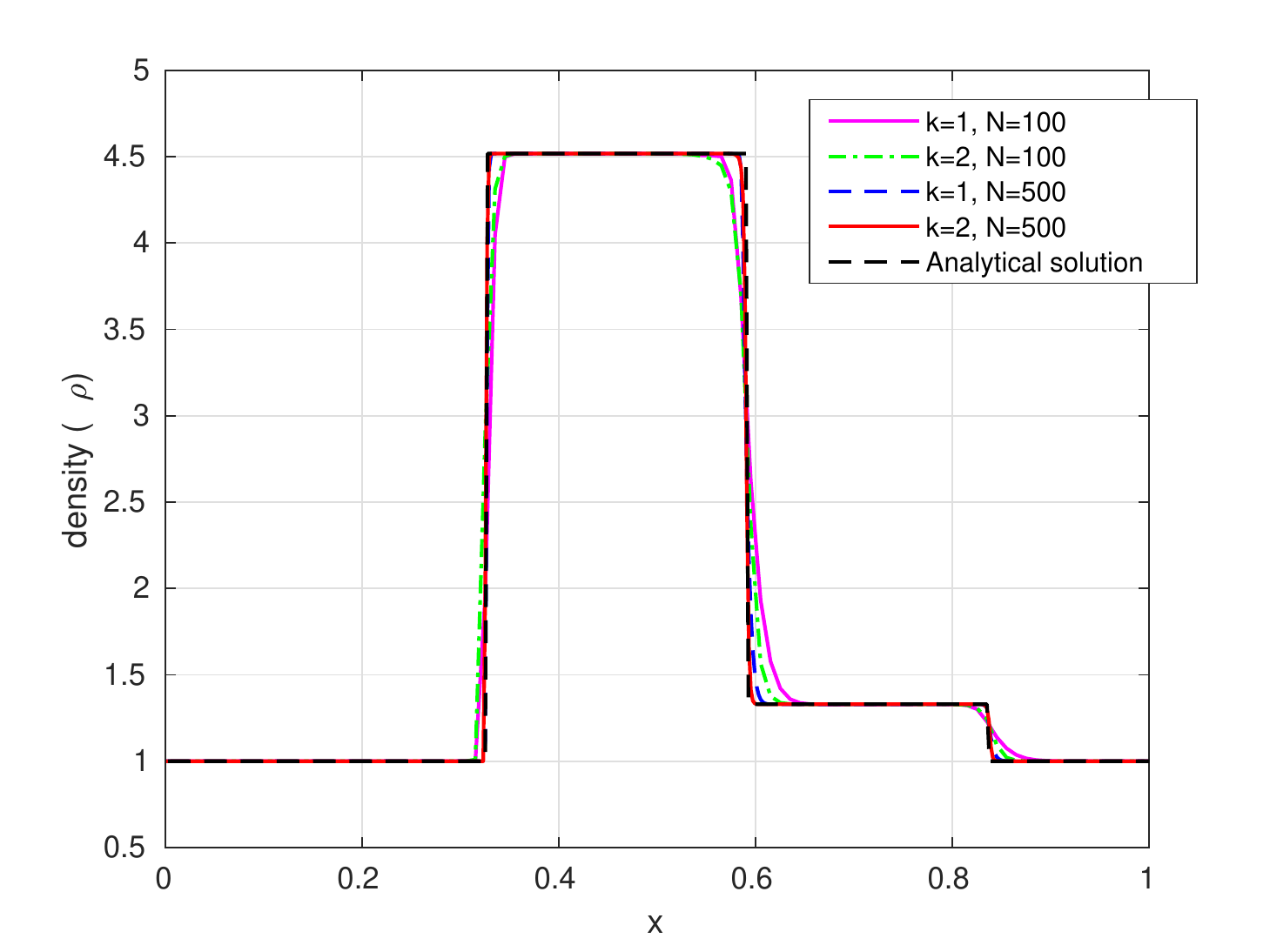}
			\caption{Density}
		\end{subfigure}
		\begin{subfigure}[b]{0.45\textwidth}
		\includegraphics[width=\textwidth]{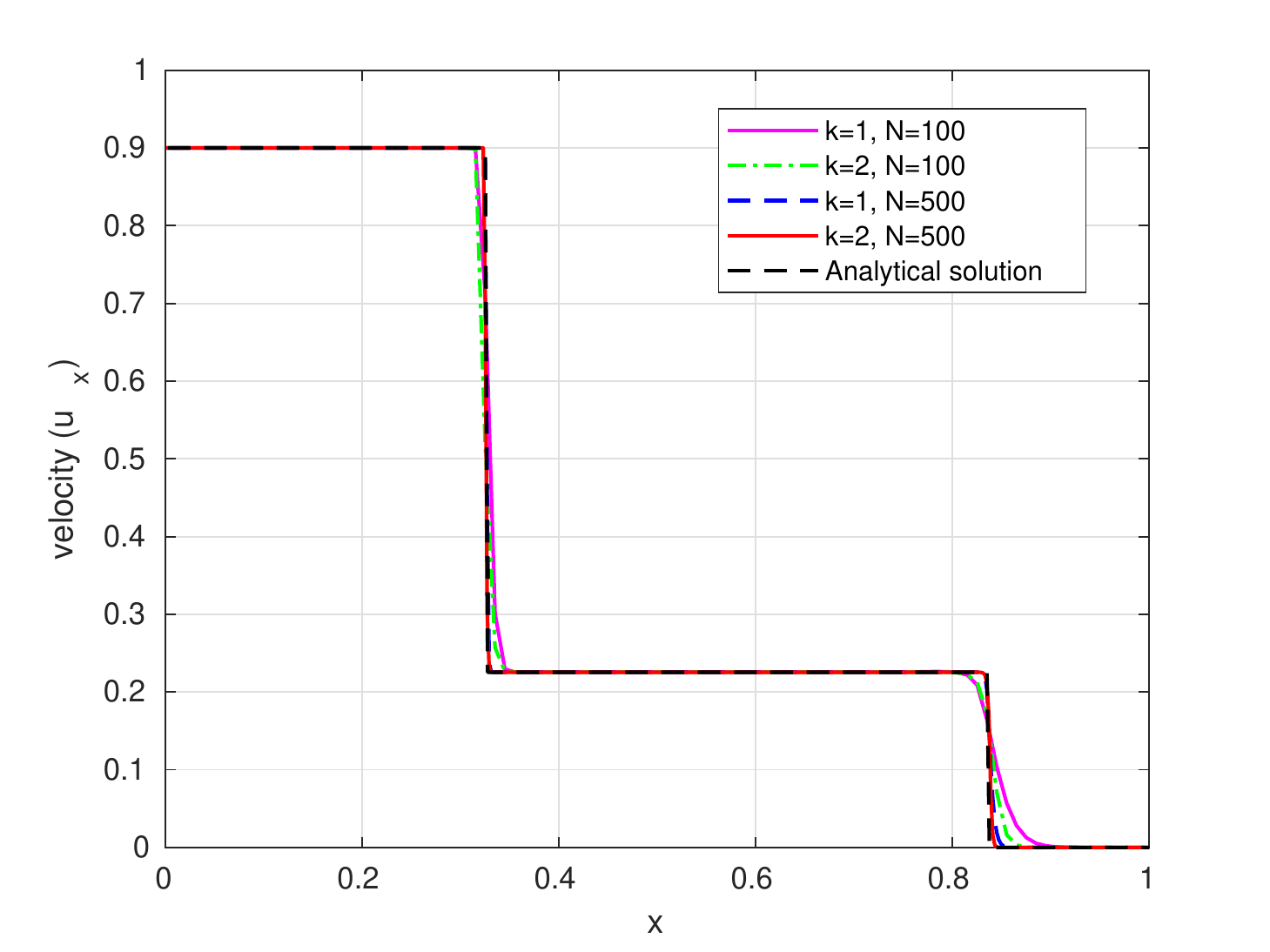}
		\caption{Velocity}
	\end{subfigure}
		\begin{subfigure}[b]{0.45\textwidth}
			\includegraphics[width=\textwidth]{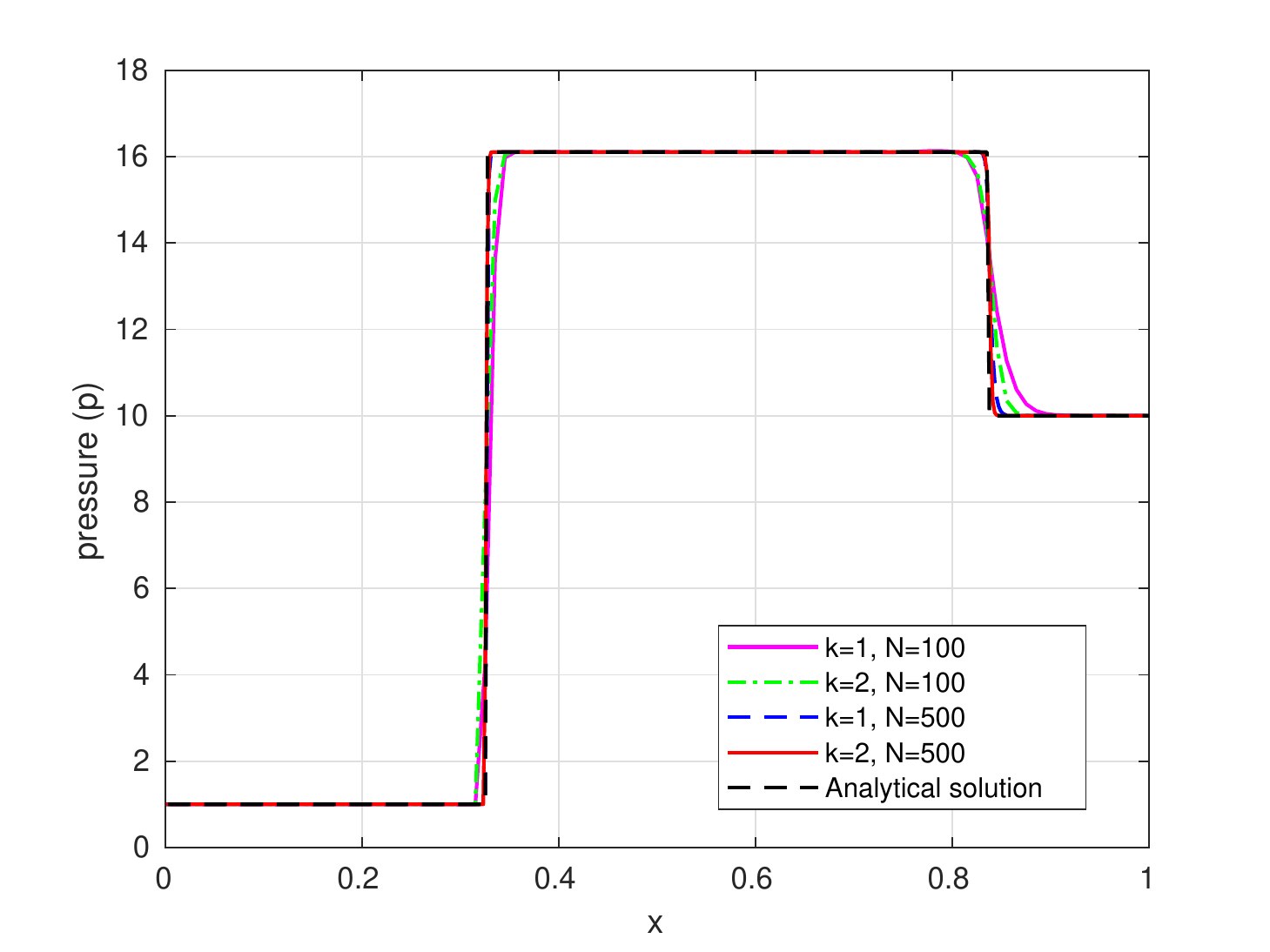}
			\caption{Pressure}
		\end{subfigure}
		\begin{subfigure}[b]{0.45\textwidth}
			\includegraphics[width=\textwidth]{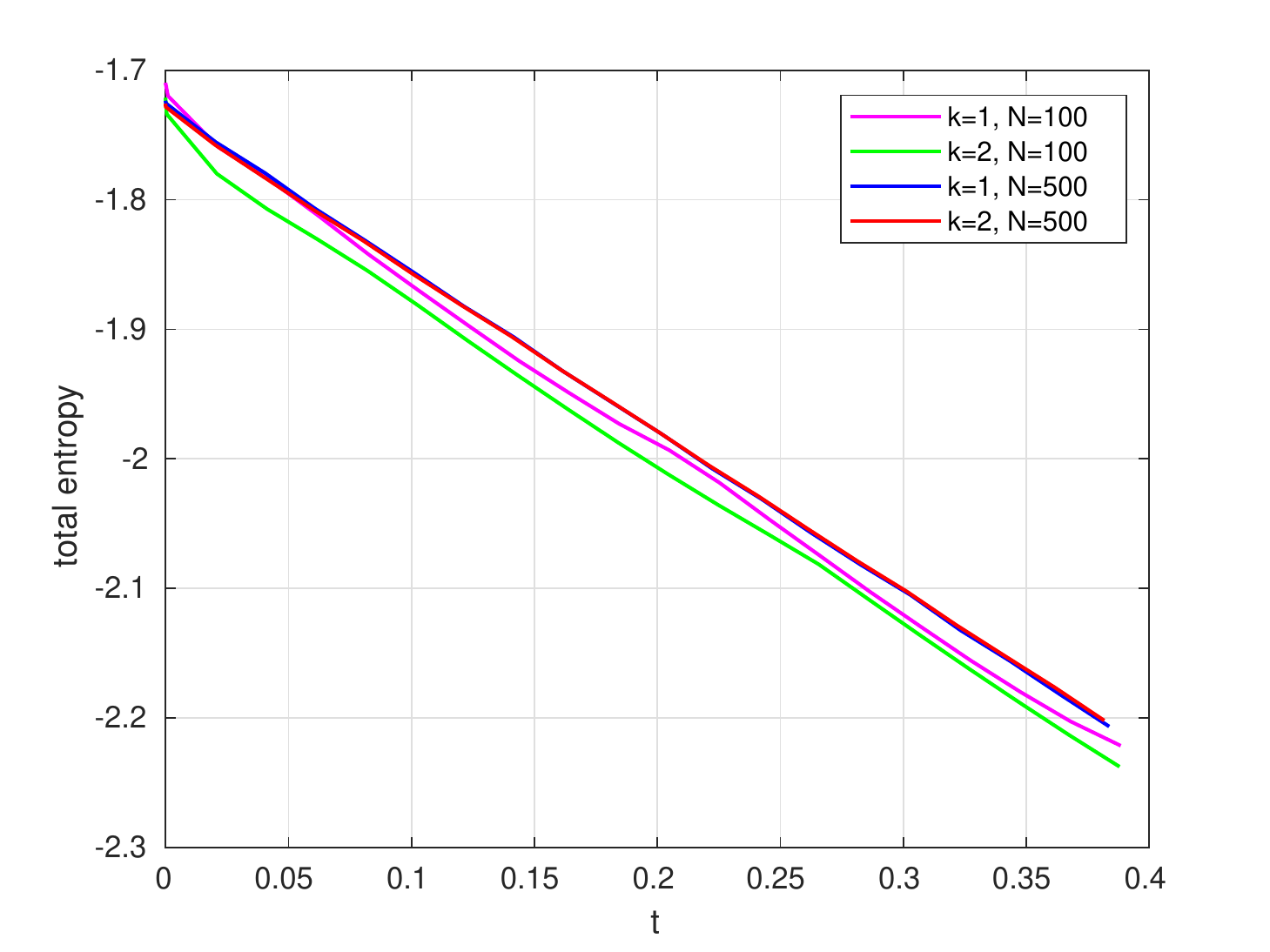}
			\caption{Evolution of total entropy}
		\end{subfigure}
		\caption{Test Problem \ref{test6} (Riemann Problem 4): Plots of density, velocity, pressure and total entropy evolution for ESDG-O2(k=1) and ESDG-O3(k=2) using 100 and 500 cells.}
		\label{fig:pb9}
	\end{figure}
We observe that both schemes are highly accurate in capturing the waves at both resolutions. We again observe only a small difference in the performance of the schemes at both resolutions.
\end{example}

\begin{example}[Density perturbation test case:]
	\label{test7}
	In this test case, we consider a problem from  \cite{del2002efficient}. The computational domain is $[0,1]$ with outflow boundary conditions. The initial conditions are,
	\begin{equation*}
	\left(\rho,\,u,\,p\right)=\begin{cases}
	\left(5,0,50\right) & \text{if $x<0.5$}\\
	\left(2 + 0.3\sin(50x),0,5\right) &  \text{if $x>0.5$}
	\end{cases}
	\end{equation*}
    The solution contains fluctuating smooth density waves, which are challenging to capture by any scheme. Numerical results are presented in Figure \ref{fig:pb8} at time $t=0.35$. At both resolutions, ESDG-O3 is more accurate than ESDG-O2. Furthermore, the solutions are highly accurate at the $500$ cells. We can also observe this from the total entropy evolution plot. At 100 cells, ESDG-O3 decays much less entropy than the ESDG-O2 scheme. Similar, at 500 cell ESDG-O3 is less entropy diffusive than the ESDG-O2 scheme. 

	\begin{figure}[htb!]
		\centering
		\begin{subfigure}[b]{0.45\textwidth}
			\includegraphics[width=\textwidth]{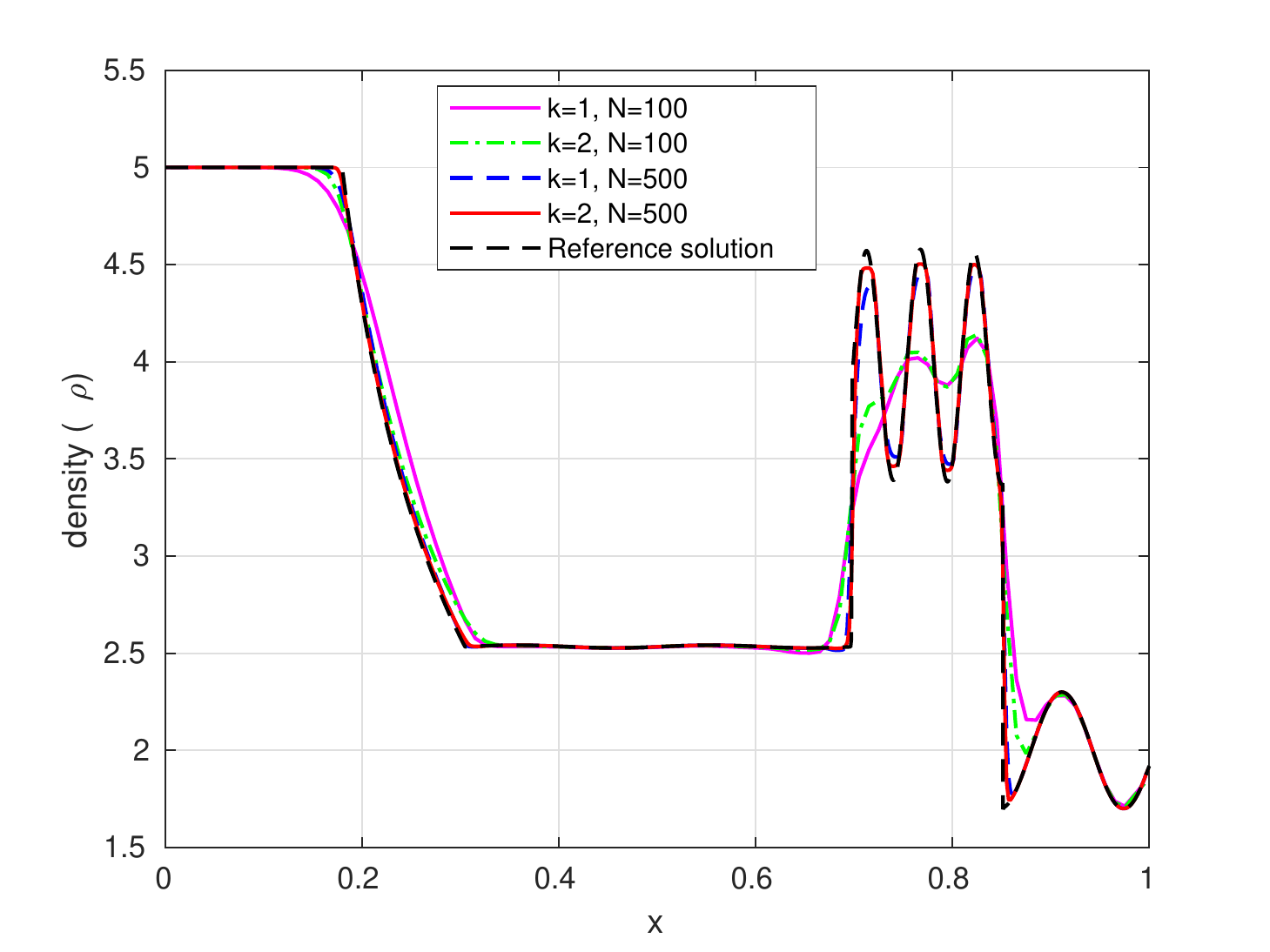}
			\caption{Density}
		\end{subfigure}
	\begin{subfigure}[b]{0.45\textwidth}
		\includegraphics[width=\textwidth]{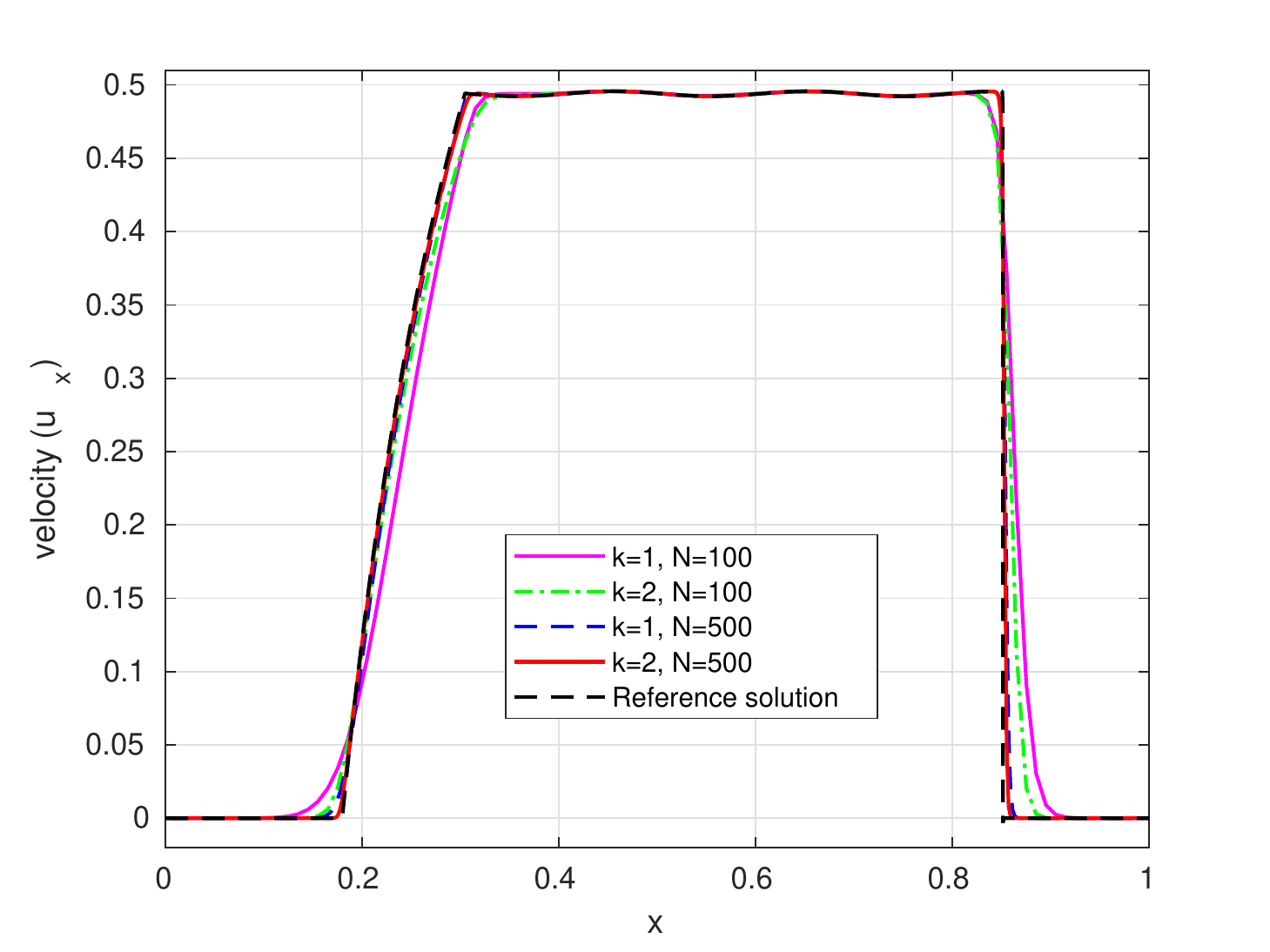}
		\caption{Velocity}
	\end{subfigure}
		\begin{subfigure}[b]{0.45\textwidth}
			\includegraphics[width=\textwidth]{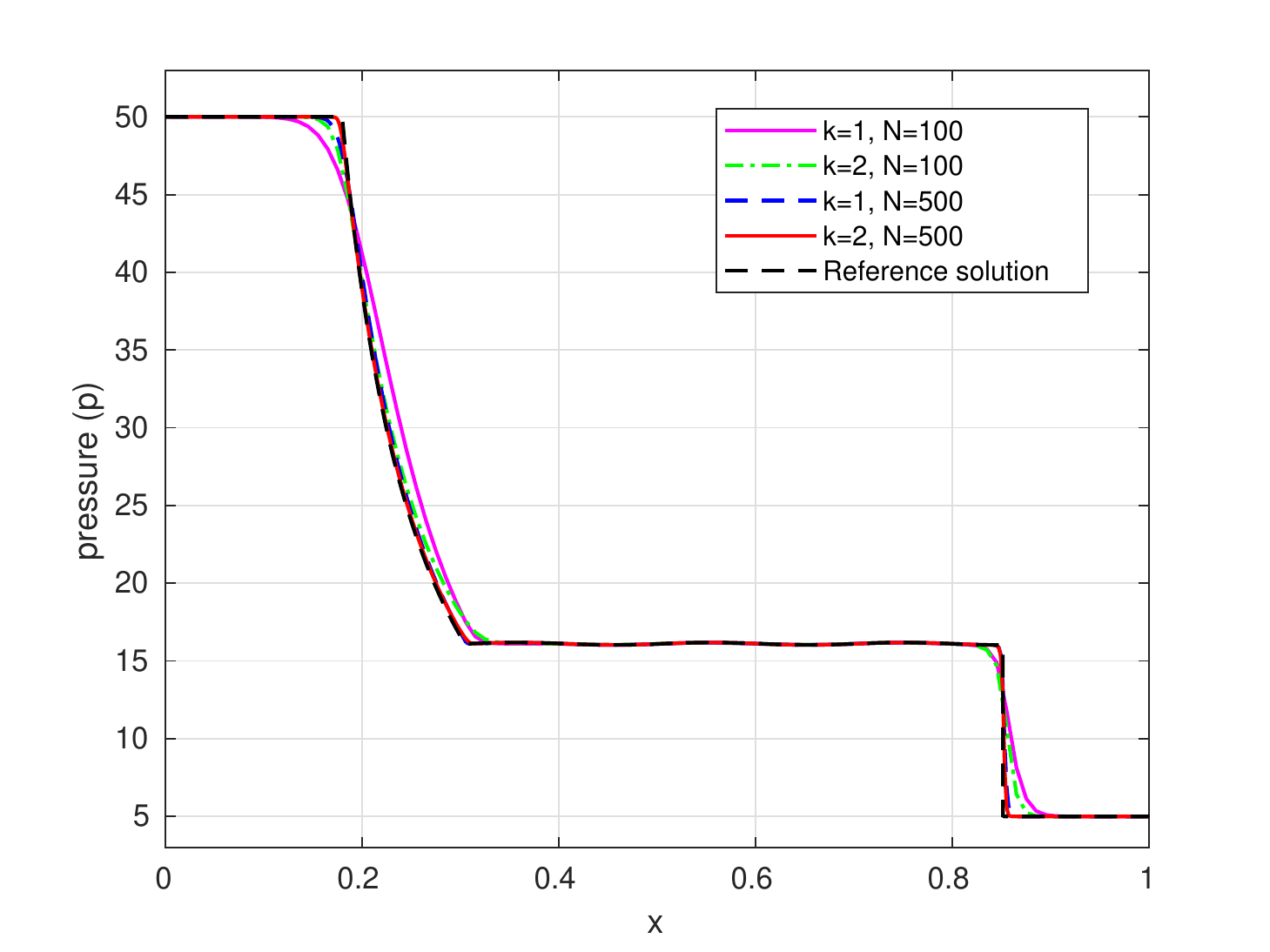}
			\caption{Pressure}
		\end{subfigure}
		\begin{subfigure}[b]{0.45\textwidth}
			\includegraphics[width=\textwidth]{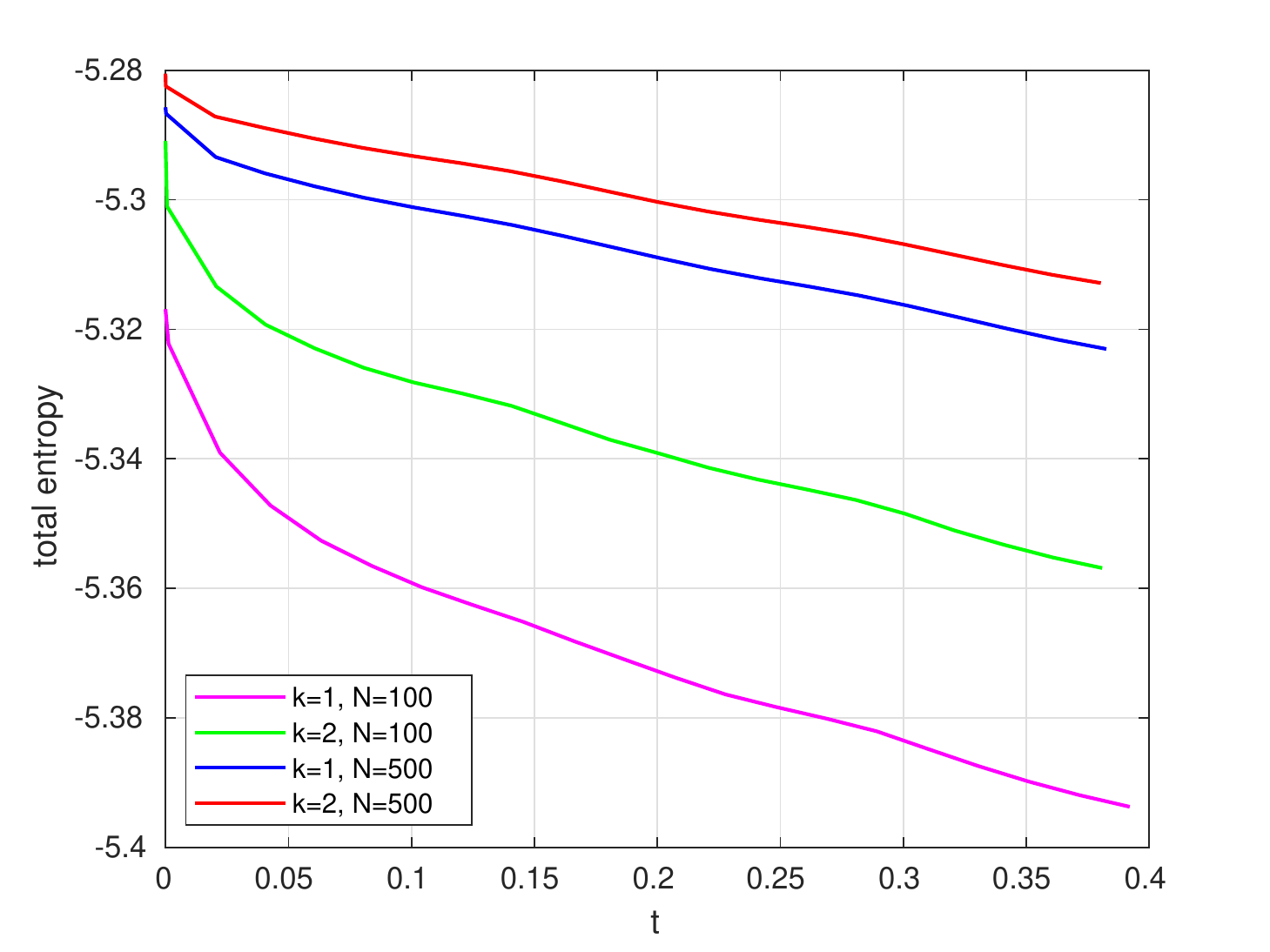}
			\caption{Evolution of total entropy}
		\end{subfigure}
		\caption{Test Problem \ref{test7} (Density perturbation test case): Plots of density, velocity, pressure and total entropy evolution for ESDG-O2(k=1) and ESDG-O3(k=2) using 100 and 500 cells.}
		\label{fig:pb8}
	\end{figure}
\end{example}

\begin{example}[Blast waves test case:]
	\label{test8}
For this test case,	we consider the blast wave interaction problem from \cite{marti1996extension}. We take computational domain of $[0,1]$ with outflow boundary conditions. The initial conditions are given by,
	\begin{equation*}
	\left(\rho,\,u,\,p\right)=\begin{cases}
	\left(1,0,1000\right) & \text{if $x<0.1$}\\
	\left(1,0,0.01\right) & \text{if $x>0.1$ and $x<0.9$}\\
	\left(1,0,100\right) &  \text{if $x>0.9$}
	\end{cases}
	\end{equation*}
    We take gas constant $\gamma=1.4$. As the solution features are concentrated in a very narrow zone; hence we need very high resolution to capture the solution. Hence we take two meshes of  $2000$ and $4000$ cells. Numerical results are plotted in Figure \ref{fig:pb18}  at $T=0.43$ for both schemes. We have zoomed in on the solution at $x=0.52$ to show the schemes' performance more clearly. At both resolutions, we observe that ESDG-O3 is more accurate than the ESDG-O2 scheme. However, both schemes can capture all the features of the solution, especially at $4000$ cells. 
	
	\begin{figure}[htb!]
		\centering
		\begin{subfigure}[b]{0.45\textwidth}
			\includegraphics[width=\textwidth]{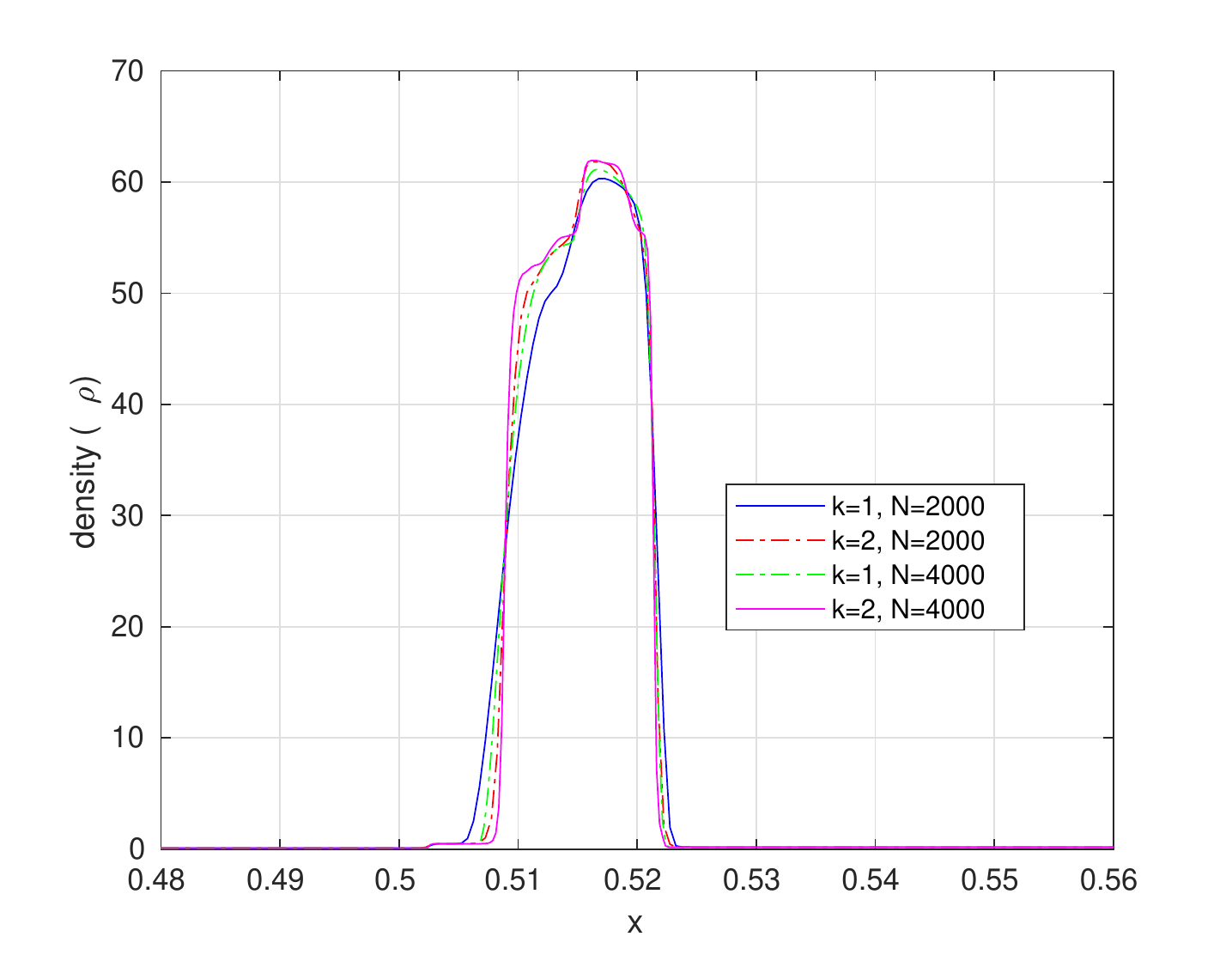}
			\caption{Density}
		\end{subfigure}
			\begin{subfigure}[b]{0.45\textwidth}
		\includegraphics[width=\textwidth]{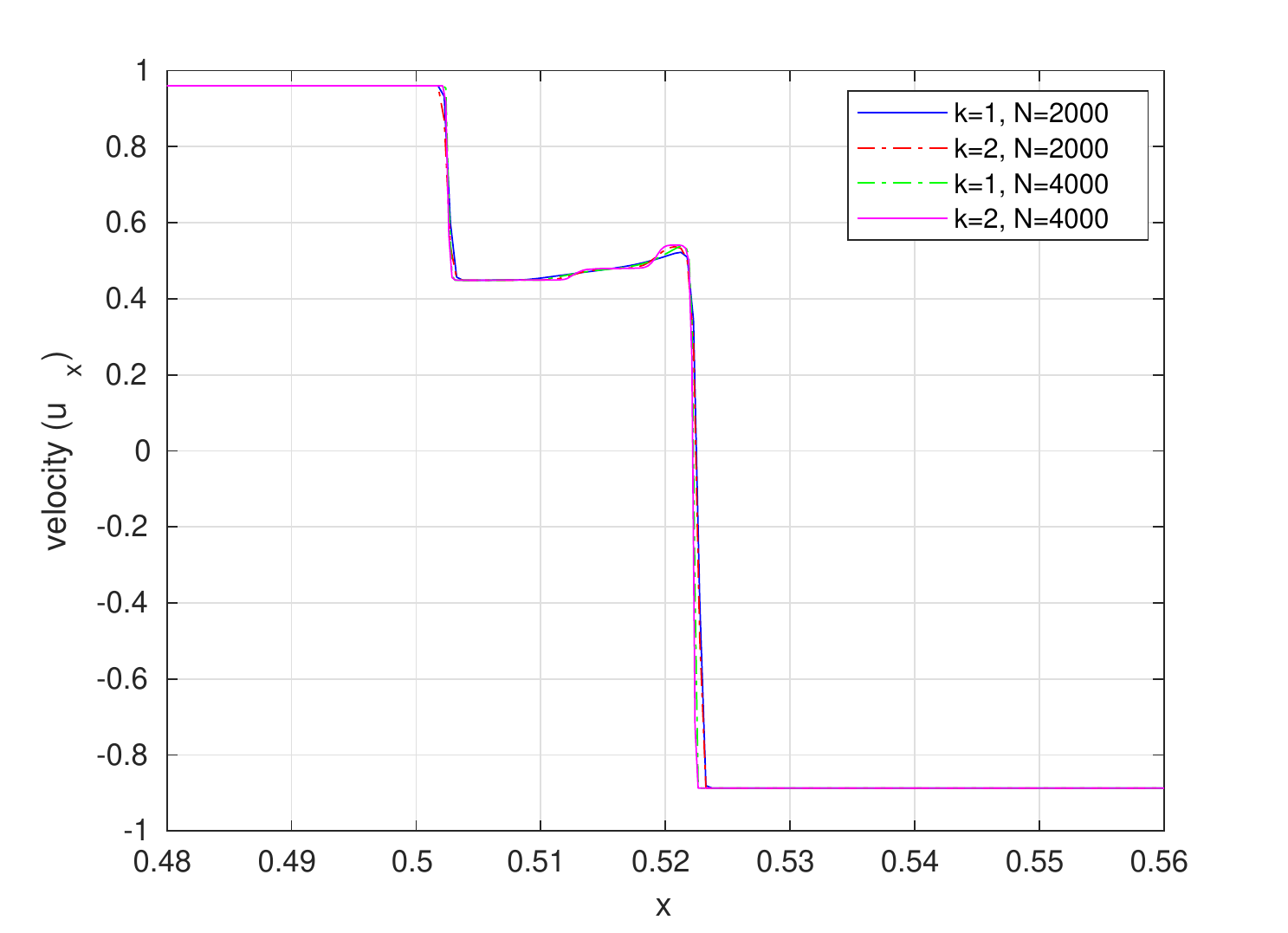}
		\caption{Velocity}
	\end{subfigure}
		\begin{subfigure}[b]{0.45\textwidth}
			\includegraphics[width=\textwidth]{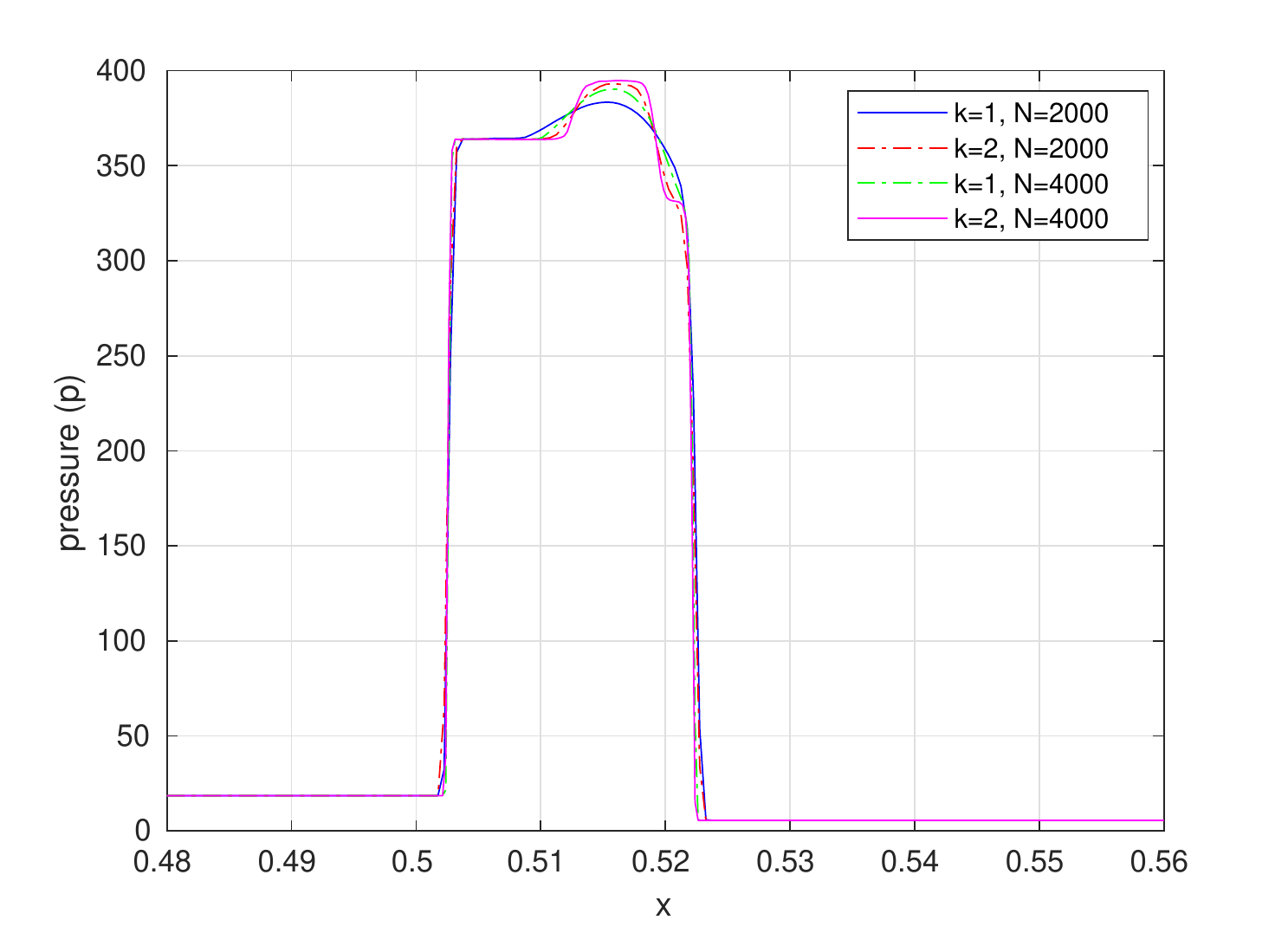}
			\caption{Pressure}
		\end{subfigure}
		\caption{Test Problem \ref{test8} (Blast waves test case): Plots of density, velocity and pressure zoomed in at $x=0.52$  for ESDG-O2(k=1) and ESDG-O3(k=2) using 2000 and 4000 cells.}
		\label{fig:pb18}
	\end{figure}
\end{example}
\section{Two dimensional numerical tests}
We now present two dimensional test cases. They are a set of two dimensional Riemann problems. 

\begin{example}[Two-dimensional Riemann problem 1:]
\label{test9}
We consider a two-dimensional Riemann problem from \cite{nunez2016xtroem}. The computational domain is $[0,1]\times[0,1]$ with outflow boundary conditions. The Riemann data is given as follows:
	\begin{equation*}
	\left(\rho,\,u_x,u_y,\,p\right)=\begin{cases}
	\left(0.5,\,0.5,\,-0.5,\,5\right) & \text{if $x>0.5$ and $y>0.5$}\\
	\left(1,\,0.5,\,0.5,\,5\right) & \text{if $x<0.5$ and $y>0.5$}\\
	\left(3,\,-0.5,\,0.5,\,5\right) & \text{if $x<0.5$ and $y<0.5$}\\
	\left(1.5,\,-0.5,\,-0.5,\,5\right) & \text{if $x>0.5$ and $y<0.5$}
	\end{cases}
	\end{equation*}
The solution to the problem has four vortex sheets that interact with a low density in the center.  Computational results are plotted in Figure \ref{fig:pb1} with $100\times 100$ cells at time $t=0.4$. We have plotted $\ln(\rho)$ and $\ln(p)$ using 25 contours. We observe that ESDG-O2 is much more diffusive than the ESDG-O3 scheme at this resolution. Furthermore, ESDG-O3 is producing very accurate results even at $100\times100$ mesh.
	
	\begin{figure}[h]
		\centering
		\begin{subfigure}[b]{0.45\textwidth}
			\includegraphics[width=\textwidth]{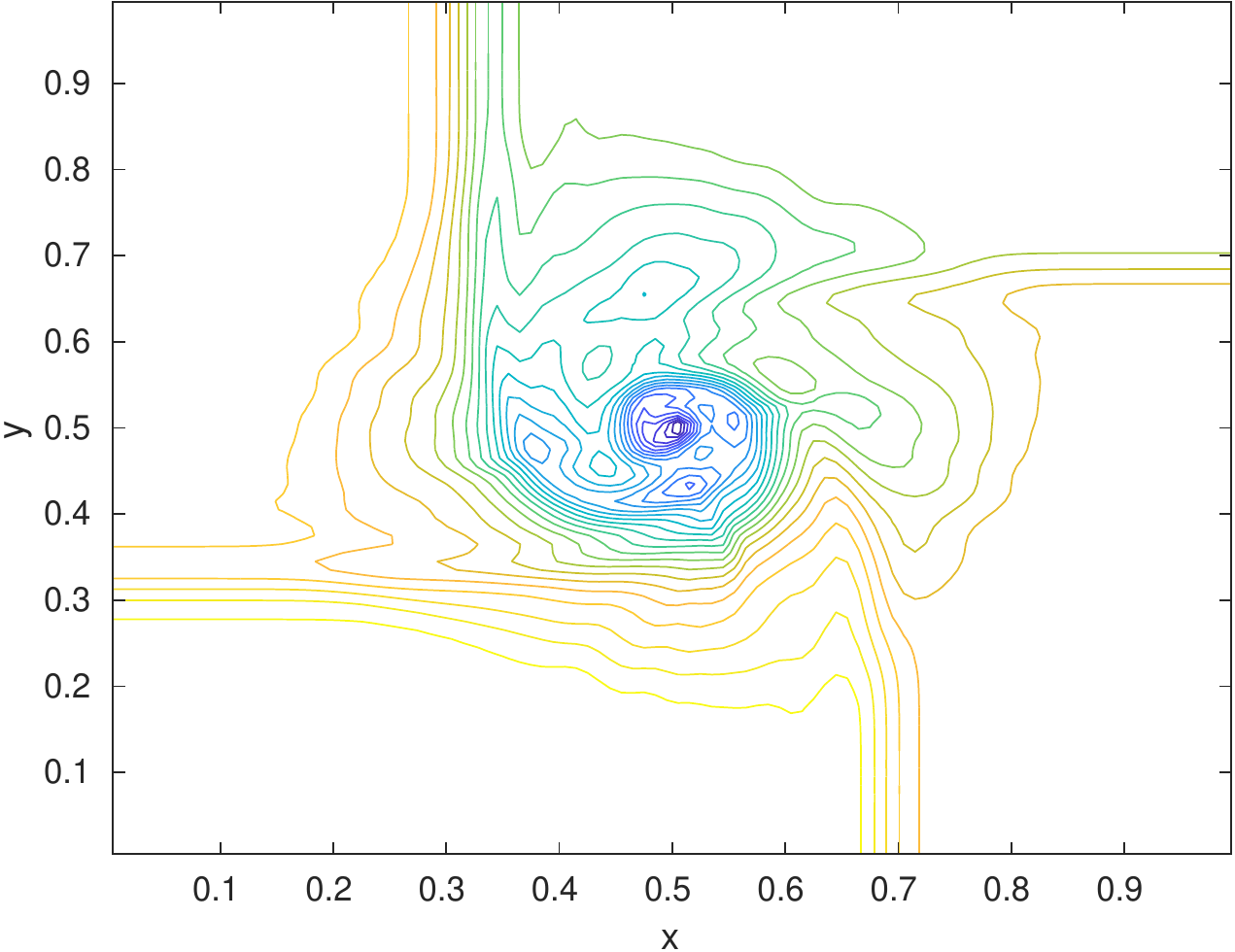}
			\caption{$\ln(\rho)$ using ESDG-O2}
		\end{subfigure}	
		\begin{subfigure}[b]{0.45\textwidth}
			\includegraphics[width=\textwidth]{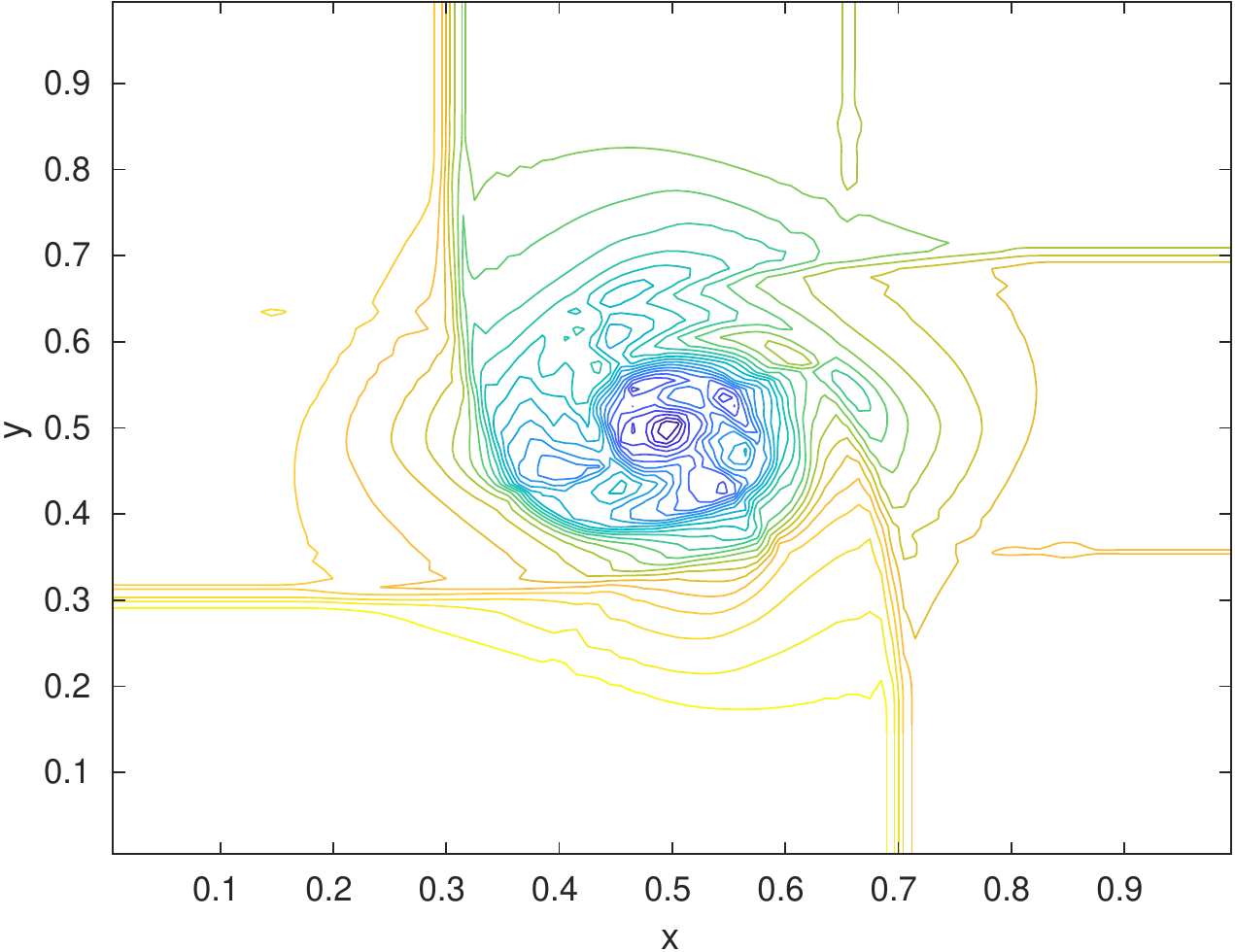}
			\caption{$\ln(\rho)$ using ESDG-O3}
		\end{subfigure}	
		\begin{subfigure}[b]{0.45\textwidth}
			\includegraphics[width=\textwidth]{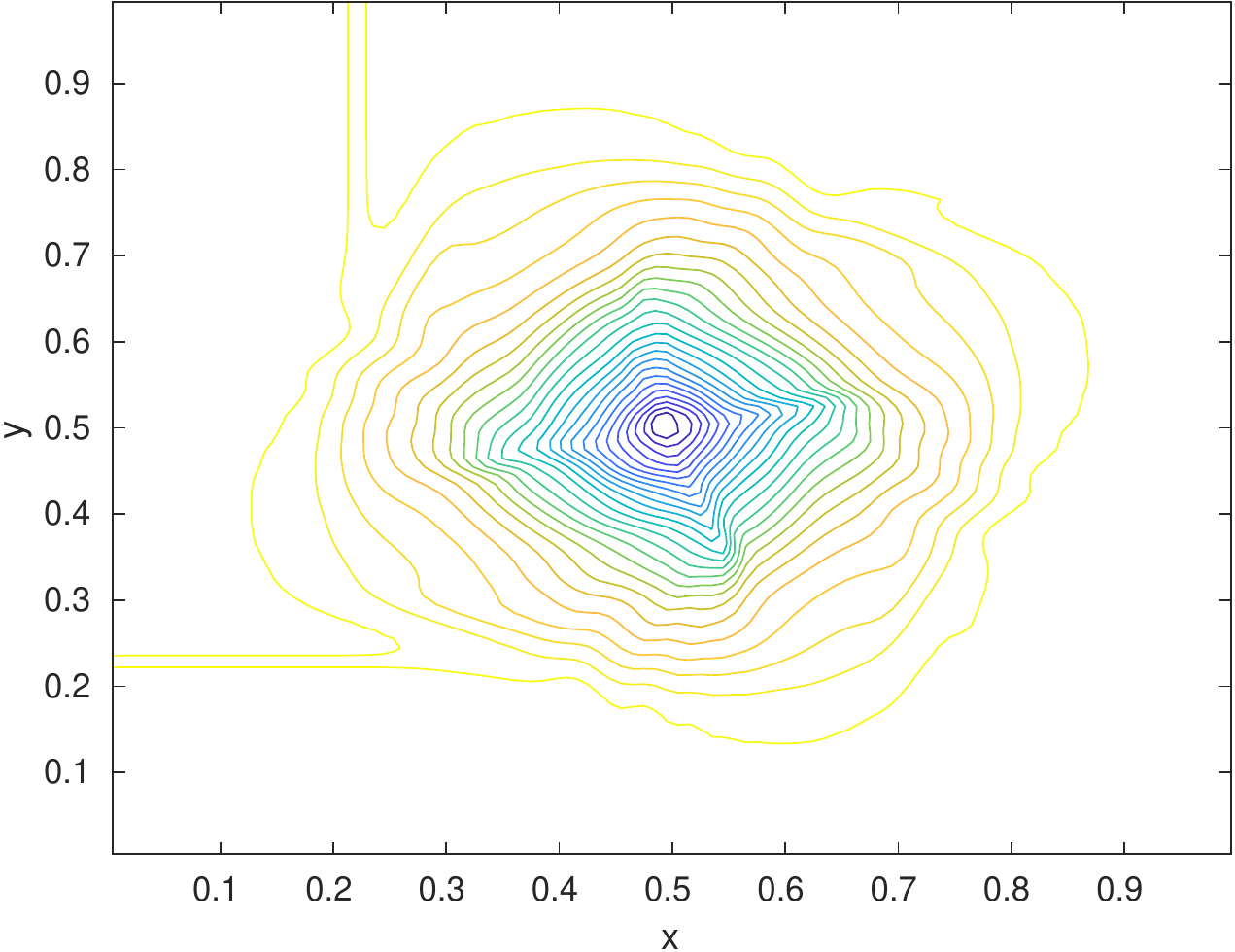}
			\caption{$\ln(p)$ using ESDG-O2}
		\end{subfigure}
		\begin{subfigure}[b]{0.45\textwidth}
			\includegraphics[width=\textwidth]{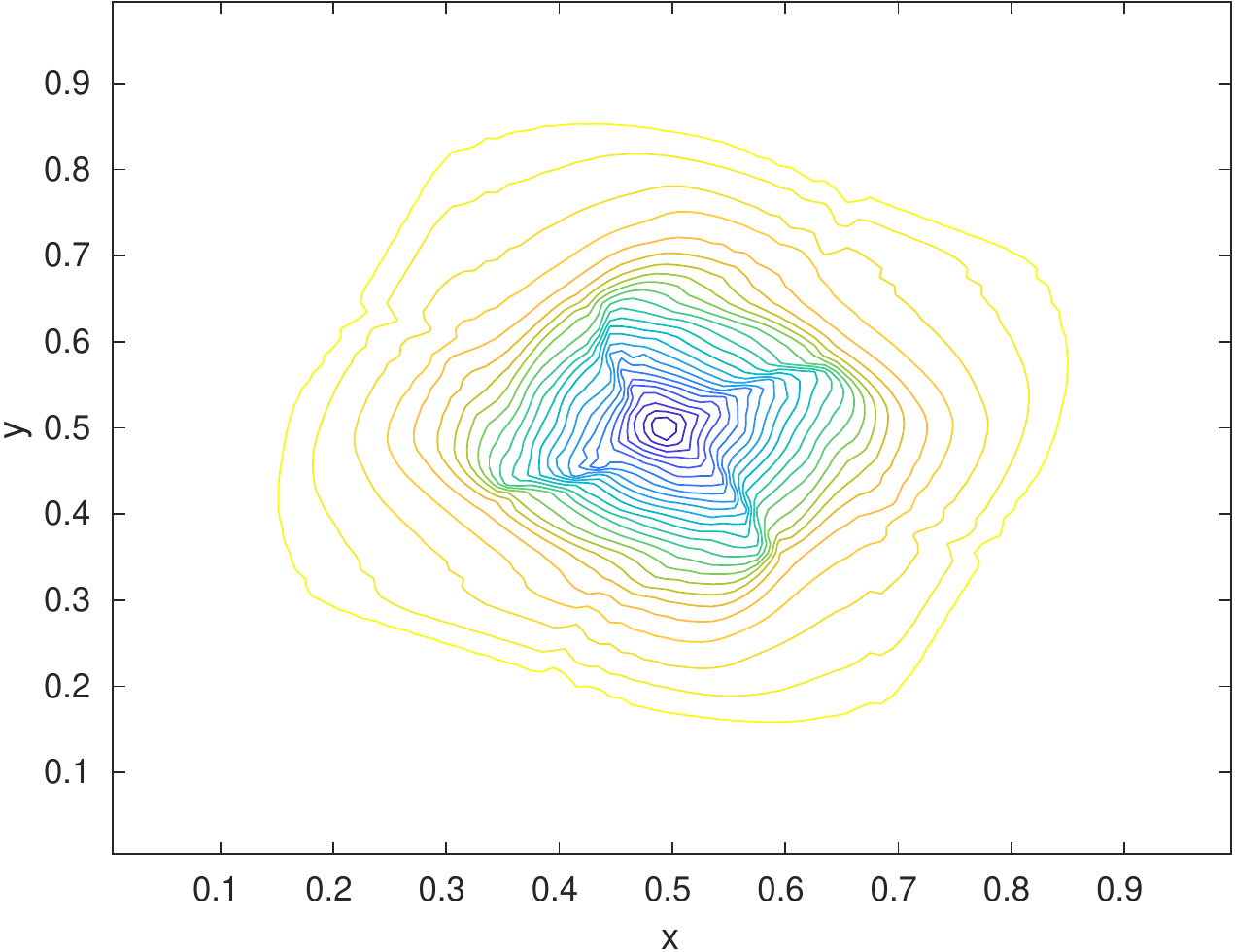}
			\caption{$\ln(p)$ using ESDG-O3}
		\end{subfigure}
		\caption{Test Problem \ref{test9} (Two-dimensional Riemann problem 1): Plots of $\ln(\rho)$ and $\ln(p)$ at time $t=0.4$ using $100\times100$ mesh.}
		\label{fig:pb1}
	\end{figure}
\end{example}
\begin{example}[Two-dimensional Riemann problem 2:]
	\label{test10}
	In this test case, we consider a Riemann problem  similar to the one in \cite{del2002efficient}. The computational domain $[0,1]\times[0,1]$ is filled with initial conditions given by,
	\begin{equation*}
	\left(\rho,\,u_x,u_y,\,p\right)=\begin{cases}
	\left(0.1,\,0,\,0,\,0.01\right) & \text{if $x>0.5$ and $y>0.5$}\\
	\left(0.1,\,0.9,\,0,\,1\right) & \text{if $x<0.5$ and $y>0.5$}\\
	\left(0.5,\,0,\,0,\,1\right) & \text{if $x<0.5$ and $y<0.5$}\\
	\left(0.1,\,0,\,0.9,\,1\right) & \text{if $x>0.5$ and $y<0.5$}
	\end{cases}
	\end{equation*}
	The computational results are presented in Figure \ref{fig:pb14} using $100\times100$ cells using outflow boundary conditions. We observe that both schemes are able to capture the solution features, with ESDG-O3 more accurate than the ESDG-O2 scheme.
	\begin{figure}[h]
		\centering
		\begin{subfigure}[b]{0.45\textwidth}
			\includegraphics[width=\textwidth]{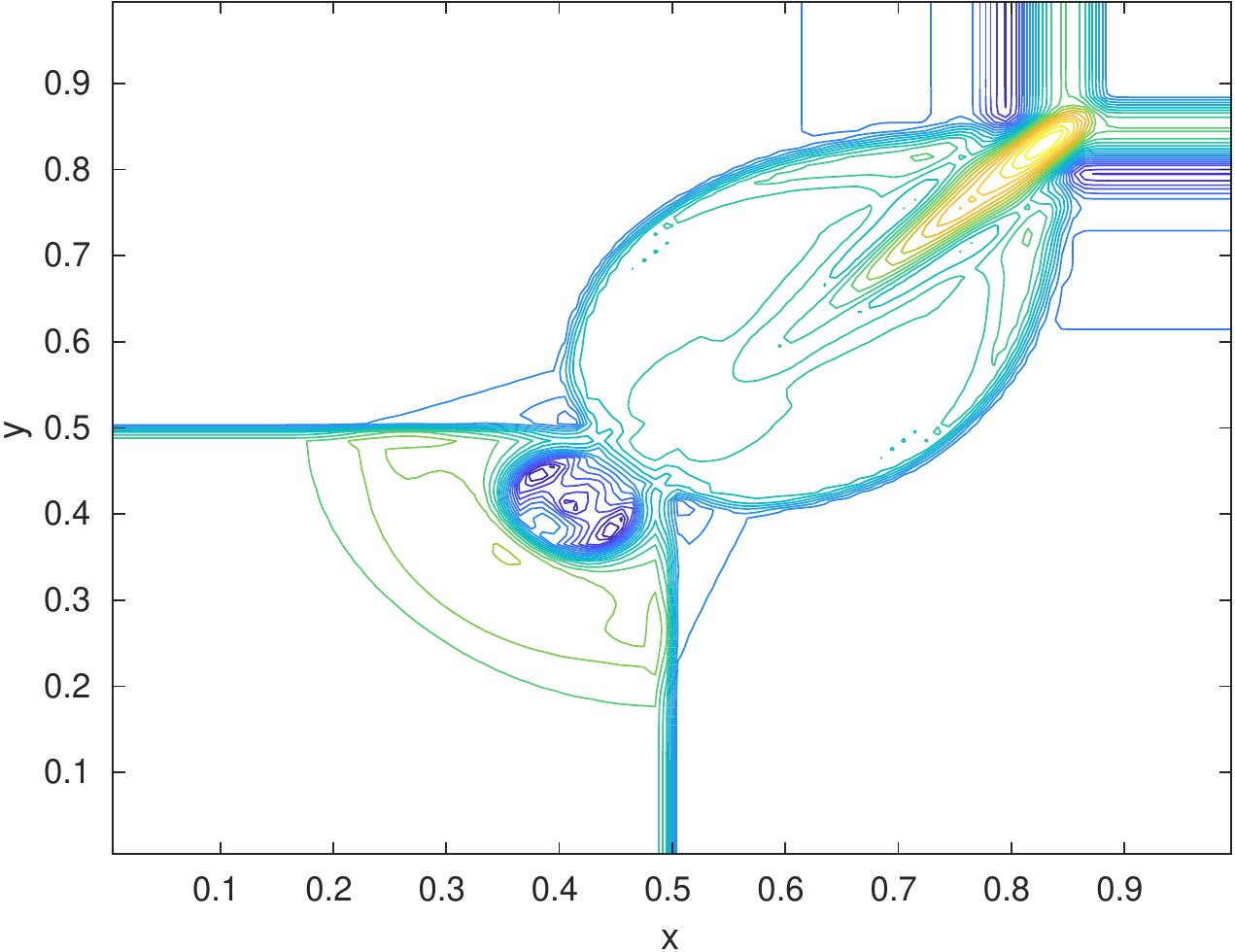}
			\caption{$\ln(\rho)$ using ESDG-O2}
		\end{subfigure}	
		\begin{subfigure}[b]{0.45\textwidth}
			\includegraphics[width=\textwidth]{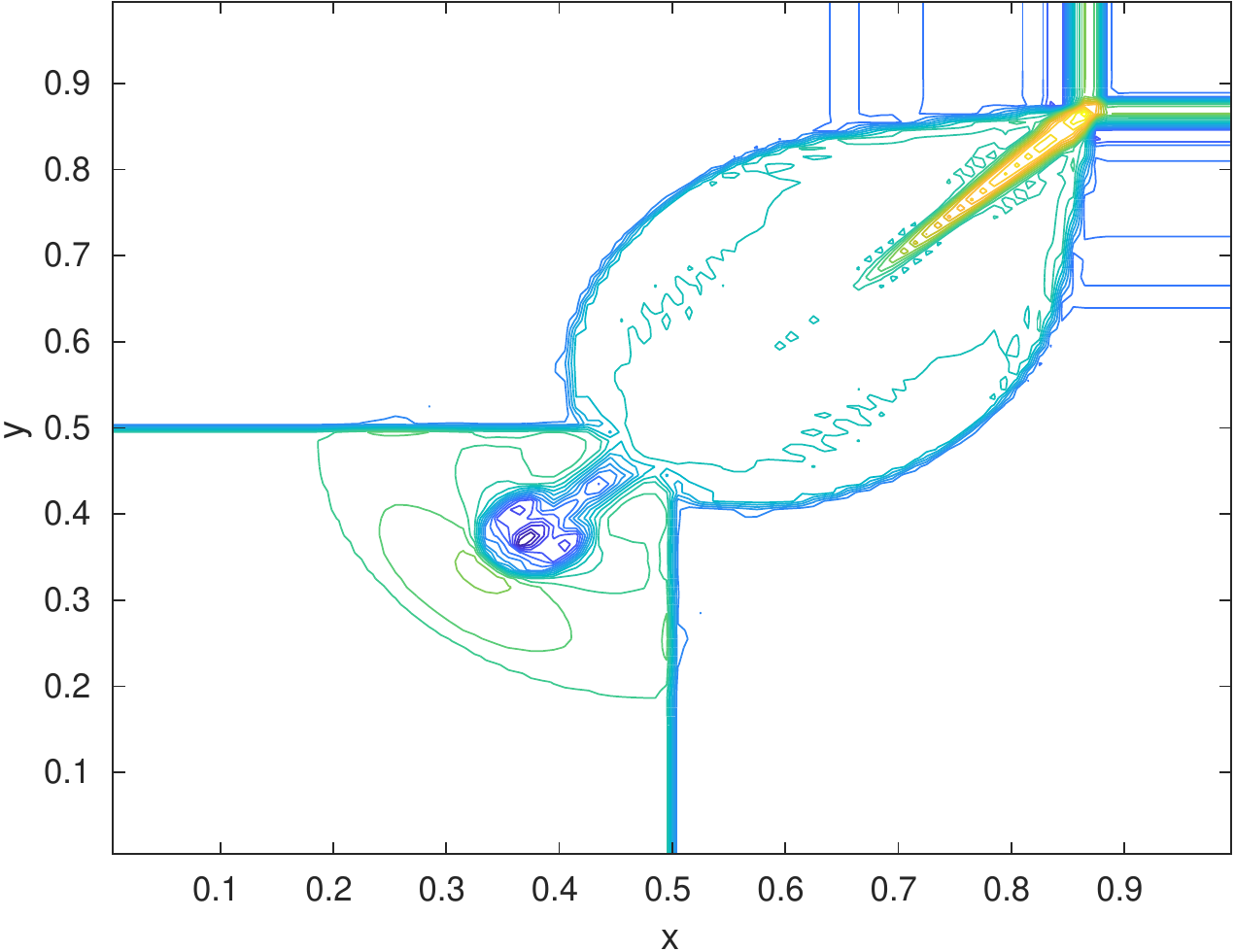}
			\caption{$\ln(\rho)$ using ESDG-O3}
		\end{subfigure}	
		\begin{subfigure}[b]{0.45\textwidth}
			\includegraphics[width=\textwidth]{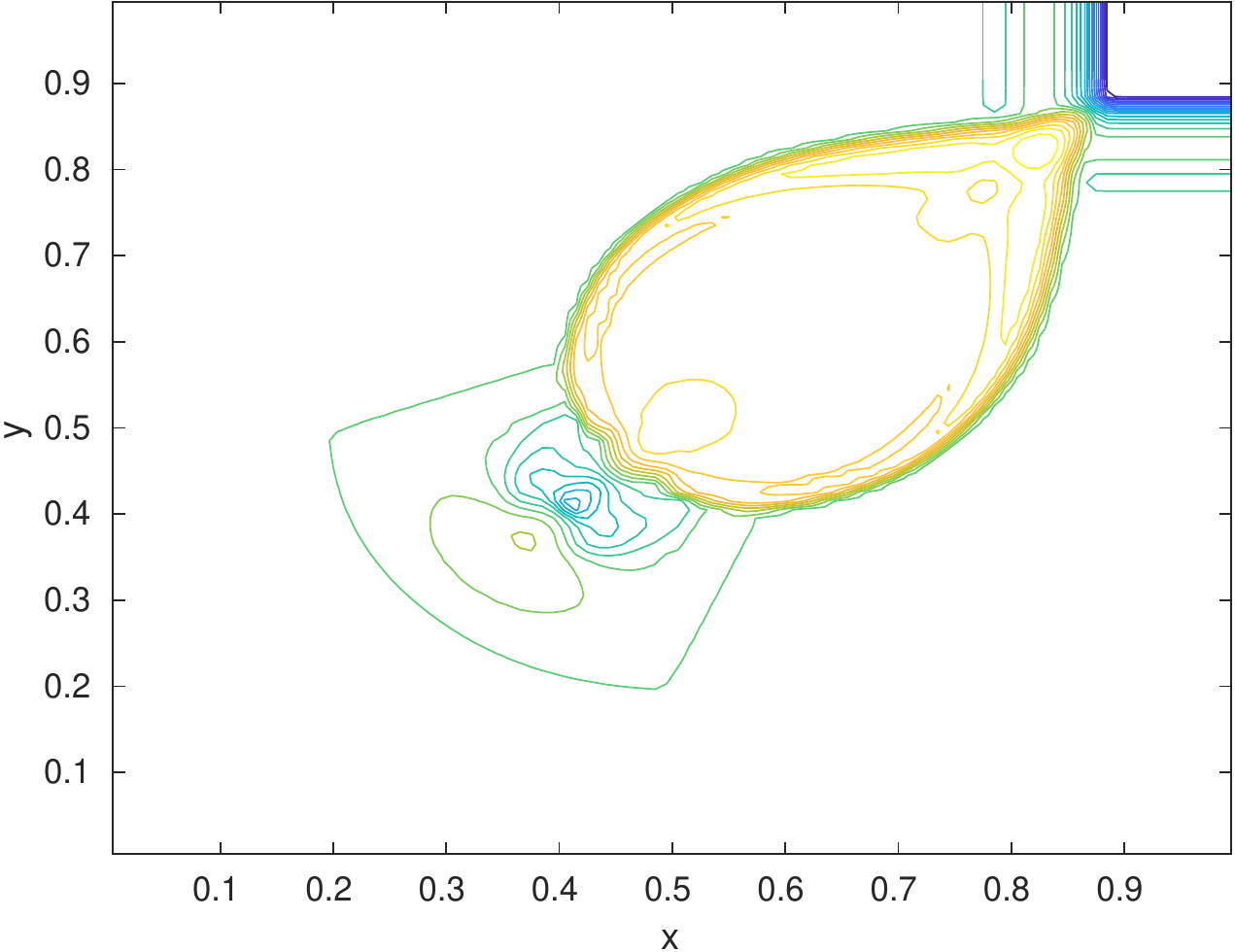}
			\caption{$\ln(p)$ using ESDG-O2}
		\end{subfigure}
		\begin{subfigure}[b]{0.45\textwidth}
			\includegraphics[width=\textwidth]{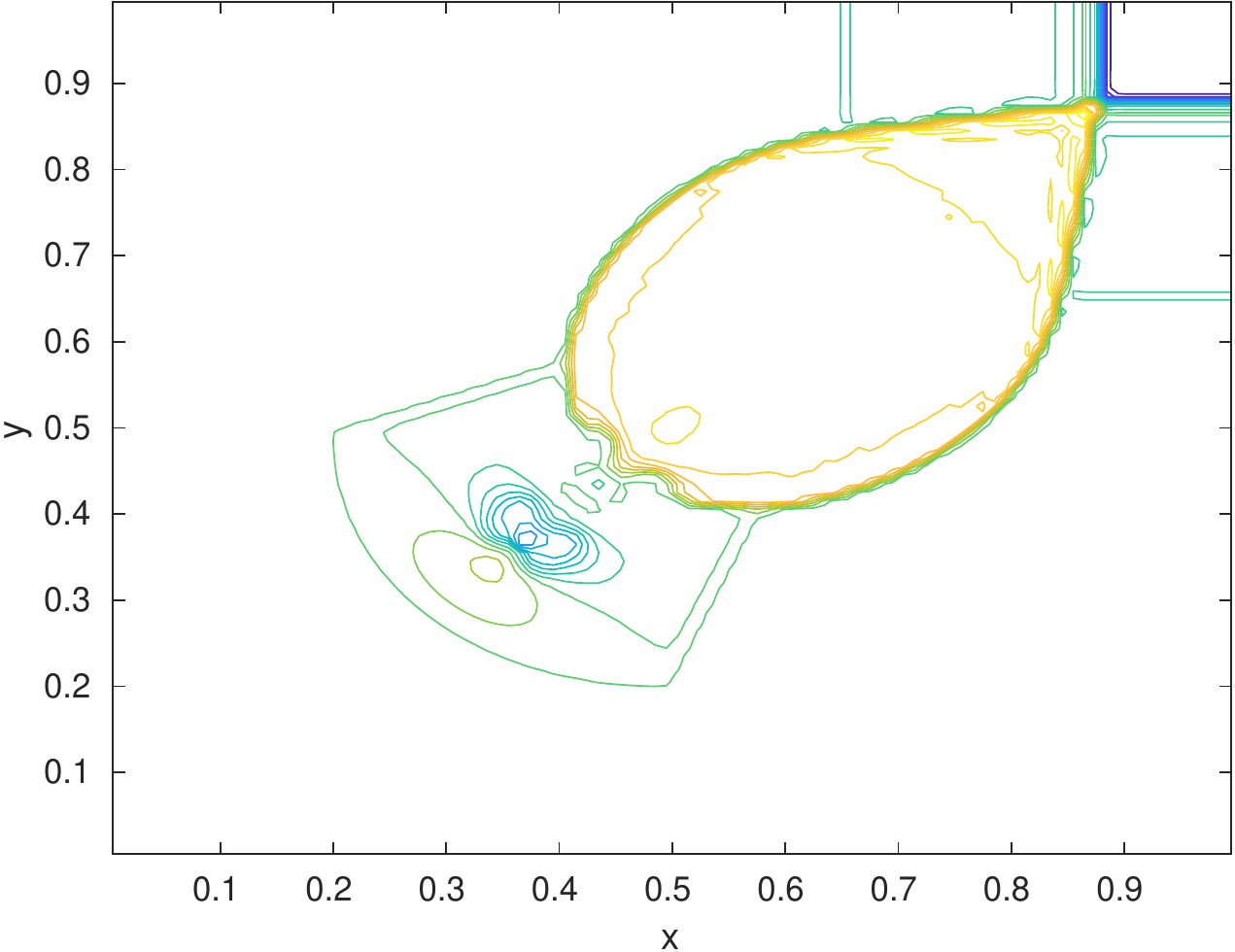}
			\caption{$\ln(p)$ using ESDG-O3}
		\end{subfigure}
		\caption{Test Problem \ref{test10} (Two-dimensional Riemann problem 2): Plots of $\ln(\rho)$ and $\ln(p)$ at time $t=0.4$ using $100\times100$ mesh.}
		\label{fig:pb14}
	\end{figure}
\end{example}
\begin{example}[Two-dimensional Riemann problem 3:]
	\label{test11}
In this test case we consider another Riemann problem from \cite{nunez2016xtroem}, where domain $[0,1]\times[0,1]$ is filled with with states given by,
	\begin{equation*}
	\left(\rho,\,u_x,u_y,\,p\right)=\begin{cases}
	\left(1,\,0,\,0,\,1\right) & \text{if $x>0.5$ and $y>0.5$}\\
	\left(0.5771,\,-0.3529,\,0,\,0.4\right) & \text{if $x<0.5$ and $y>0.5$}\\
	\left(1,\,-0.3529,\,-0.3529,\,1\right) & \text{if $x<0.5$ and $y<0.5$}\\
	\left(0.5771,\,0,\,-0.3529,\,0.4\right) & \text{if $x>0.5$ and $y<0.5$}
	\end{cases}
	\end{equation*}
    The solution involves the interaction of two rarefaction waves, which results in two symmetric shocks. The numerical results are presented in  Figure \ref{fig:pb15}  using $100\times 100$ cells for ESDG-O3 and ESDG-O2 schemes at time $t=0.4$. We have plotted $\ln(\rho)$ and $\ln(p)$ using 25 contours. We again observe the outperformance of ESDG-O3 compared to the ESDG-O2 scheme.     
	\begin{figure}[h]
		\centering
		\begin{subfigure}[b]{0.45\textwidth}
			\includegraphics[width=\textwidth]{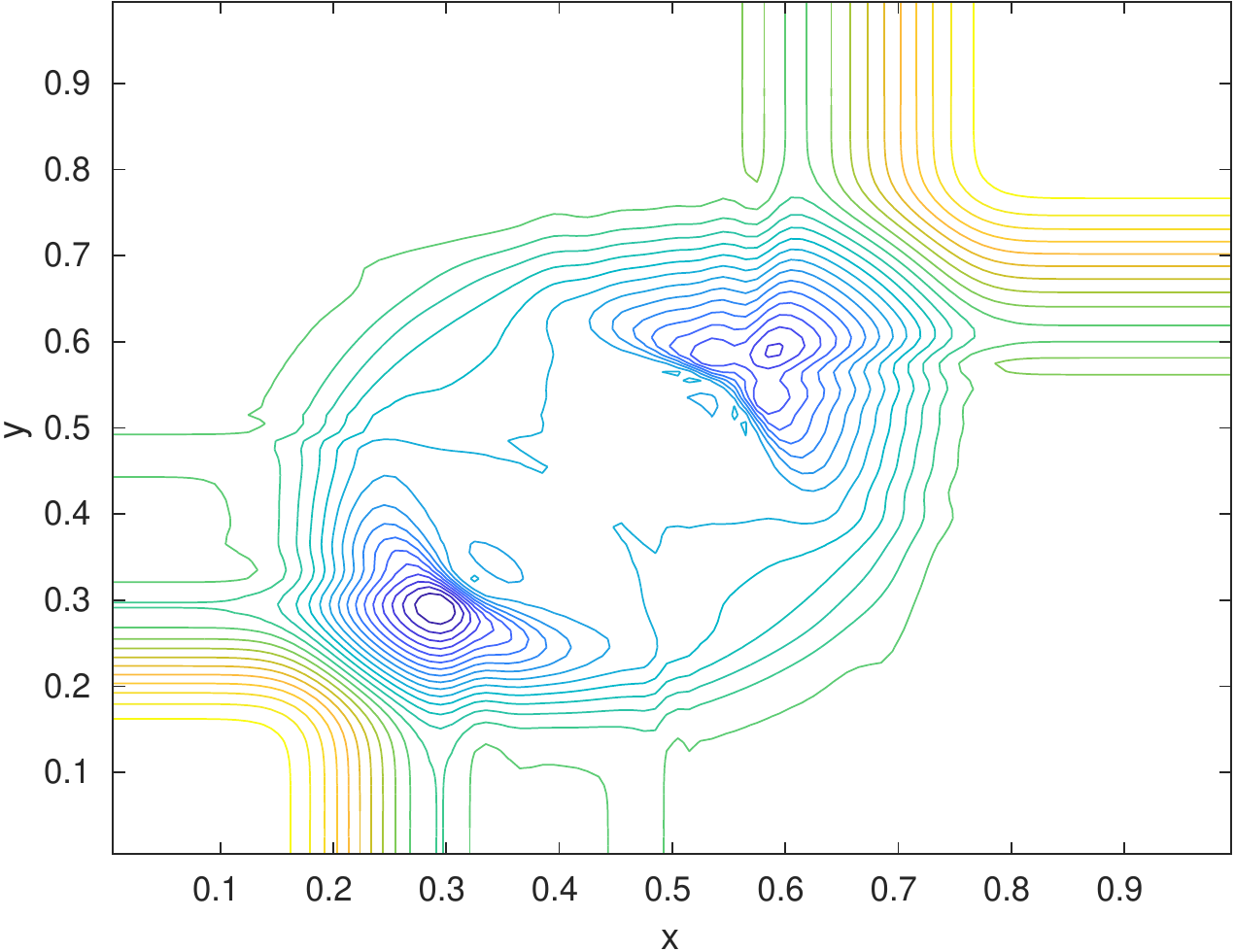}
			\caption{$\ln(\rho)$ using ESDG-O2}
		\end{subfigure}	
		\begin{subfigure}[b]{0.45\textwidth}
			\includegraphics[width=\textwidth]{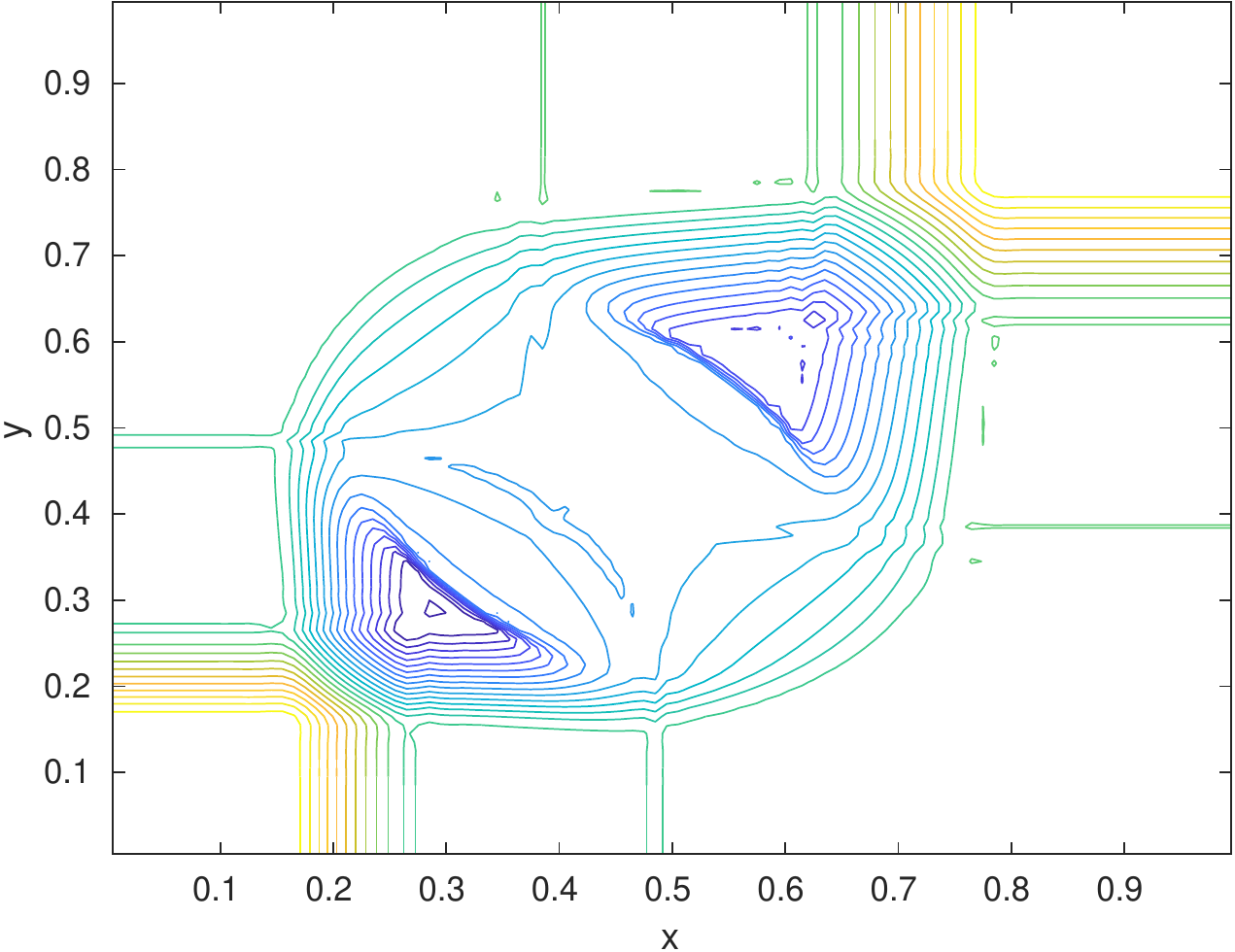}
			\caption{$\ln(\rho)$ using ESDG-O3}
		\end{subfigure}	
		\begin{subfigure}[b]{0.45\textwidth}
			\includegraphics[width=\textwidth]{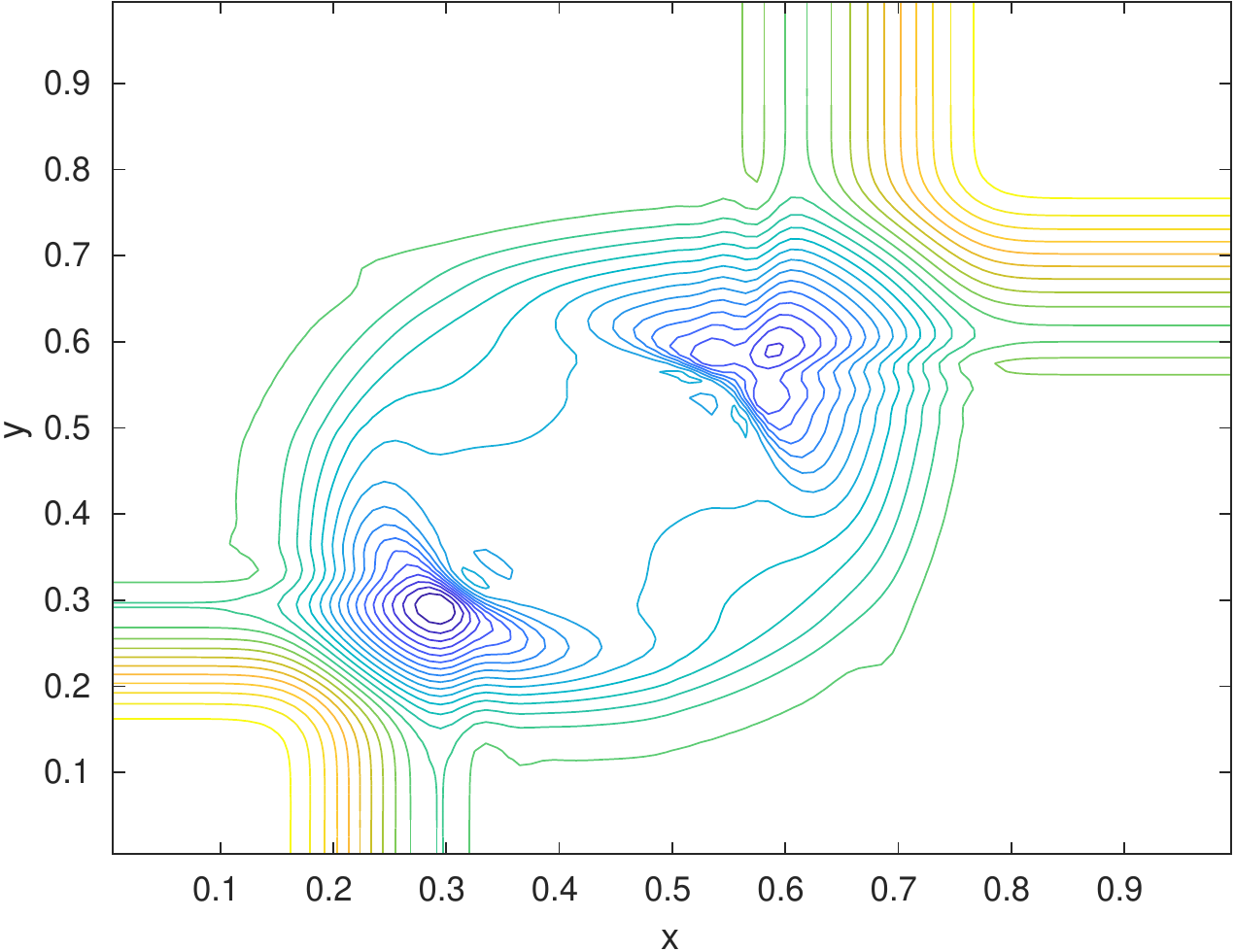}
			\caption{$\ln(p)$ using ESDG-O2}
		\end{subfigure}
		\begin{subfigure}[b]{0.45\textwidth}
			\includegraphics[width=\textwidth]{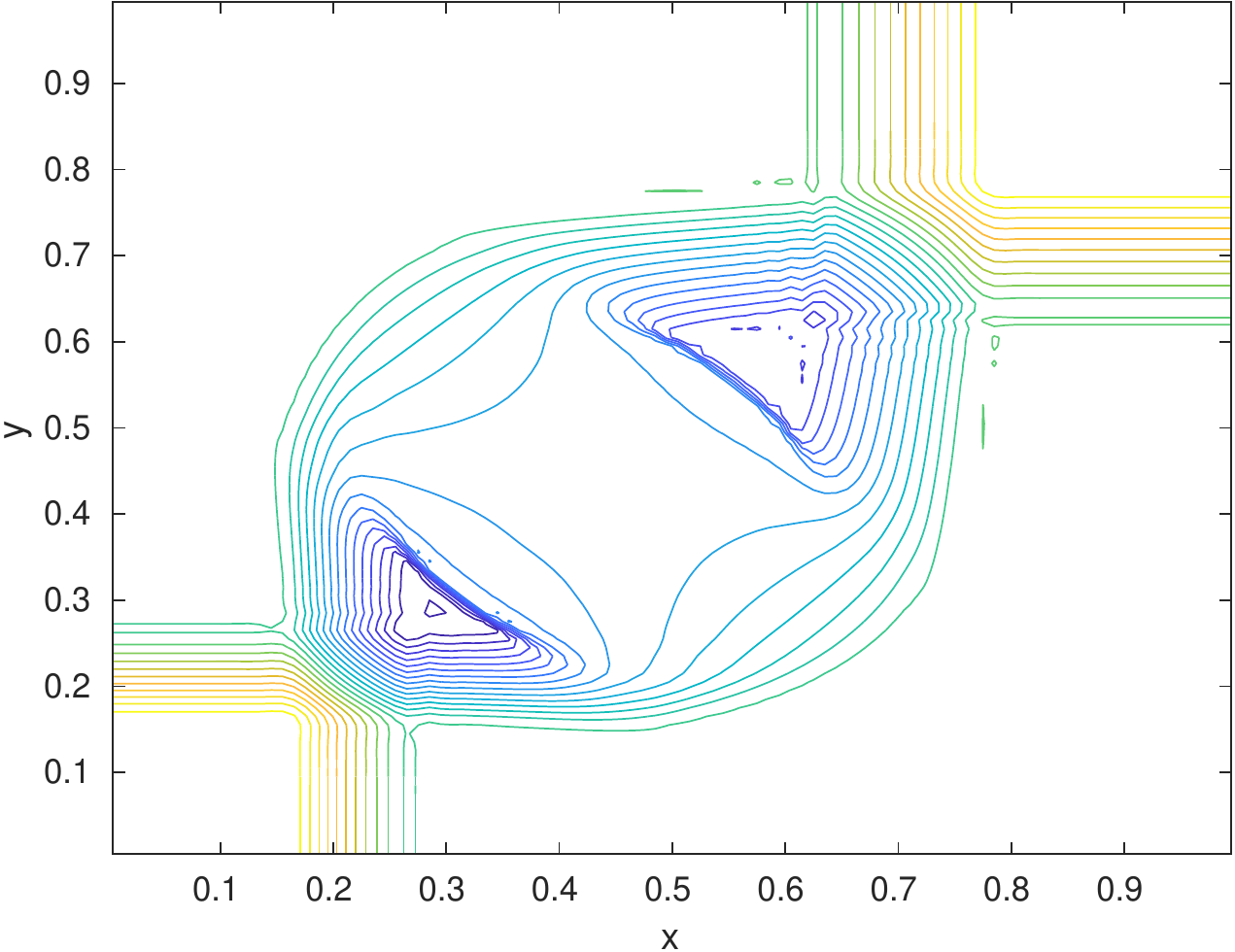}
			\caption{$\ln(p)$ using ESDG-O3}
		\end{subfigure}
		\caption{Test Problem \ref{test11} (Two-dimensional Riemann problem 3): Plots of $\ln(\rho)$ and $\ln(p)$ at time $t=0.4$ using $100\times100$ mesh.}
		\label{fig:pb15}
	\end{figure}
\end{example}

\begin{example}[Two-dimensional Riemann problem 4:]
	\label{test12}	
In this test case from \cite{nunez2016xtroem} , we again consider computational domain $[0,1]\times[0,1]$ with outflow boundary conditions. The  initial conditions are given by,
	\begin{equation*}
	\left(\rho,\,u_x,u_y,\,p\right)=\begin{cases}
	(0.035145216124503,\,0,&\\
	\,0,\,0.162931056509027) & \text{if $x>0.5$ and $y>0.5$}\\
	\left(0.1,\,0.7,\,0,\,1\right) & \text{if $x<0.5$ and $y>0.5$}\\
	\left(0.5,\,0,\,0,\,1\right) & \text{if $x<0.5$ and $y<0.5$}\\
	\left(0.1,\,0,\,0.7,\,1\right) & \text{if $x>0.5$ and $y<0.5$}
	\end{cases}
	\end{equation*}
    We again use the $100\times 100$ mesh for both schemes. we have plotted $\ln(\rho)$ and $\ln(p)$ using 25 contours. We observe that schemes can capture curved shocks very well. Furthermore, the ESDG-O3 scheme is more accurate than the ESDG-O2 scheme
		\begin{figure}[h]
		\centering
		\begin{subfigure}[b]{0.45\textwidth}
			\includegraphics[width=\textwidth]{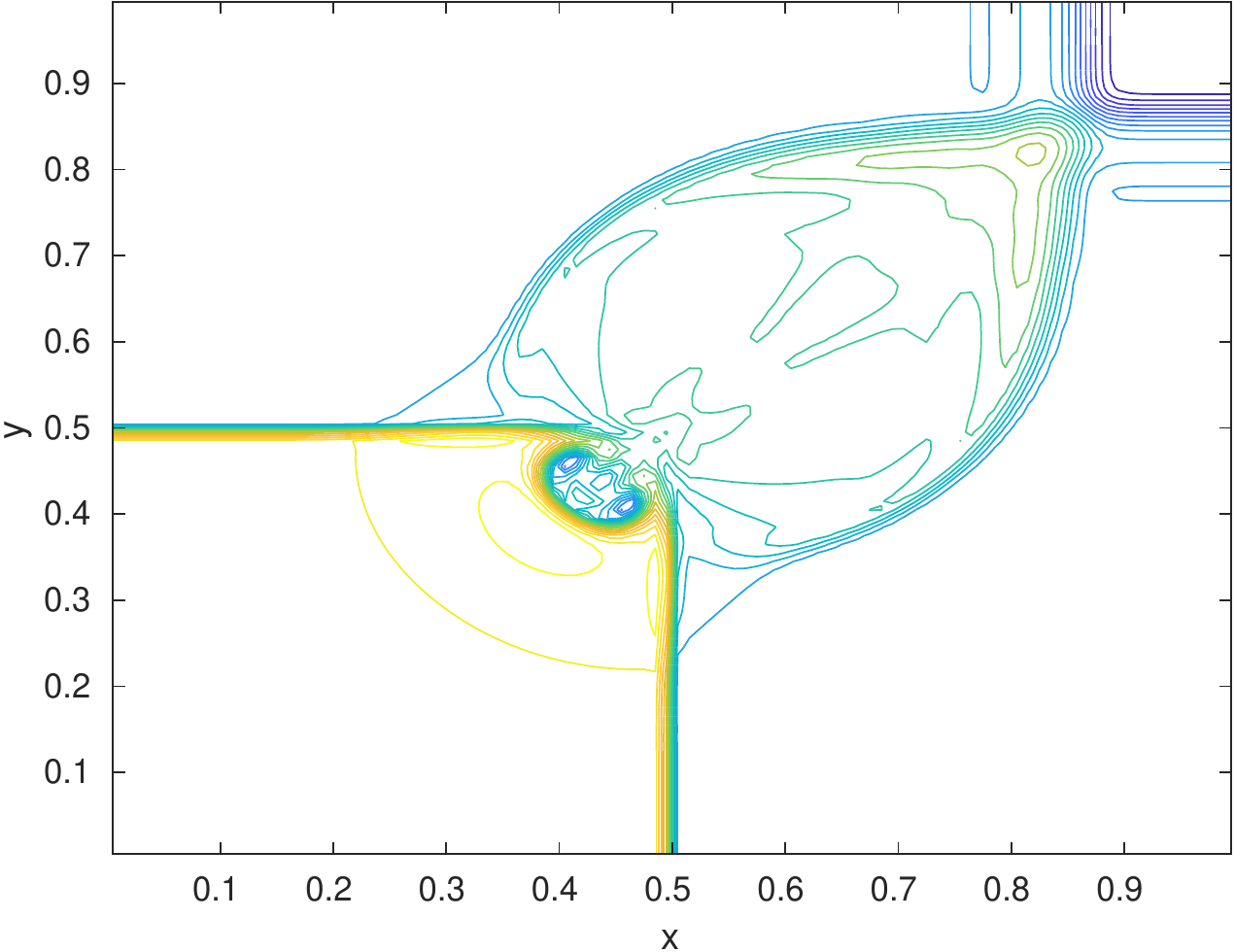}
			\caption{$\ln(\rho)$ using ESDG-O2}
		\end{subfigure}	
		\begin{subfigure}[b]{0.45\textwidth}
			\includegraphics[width=\textwidth]{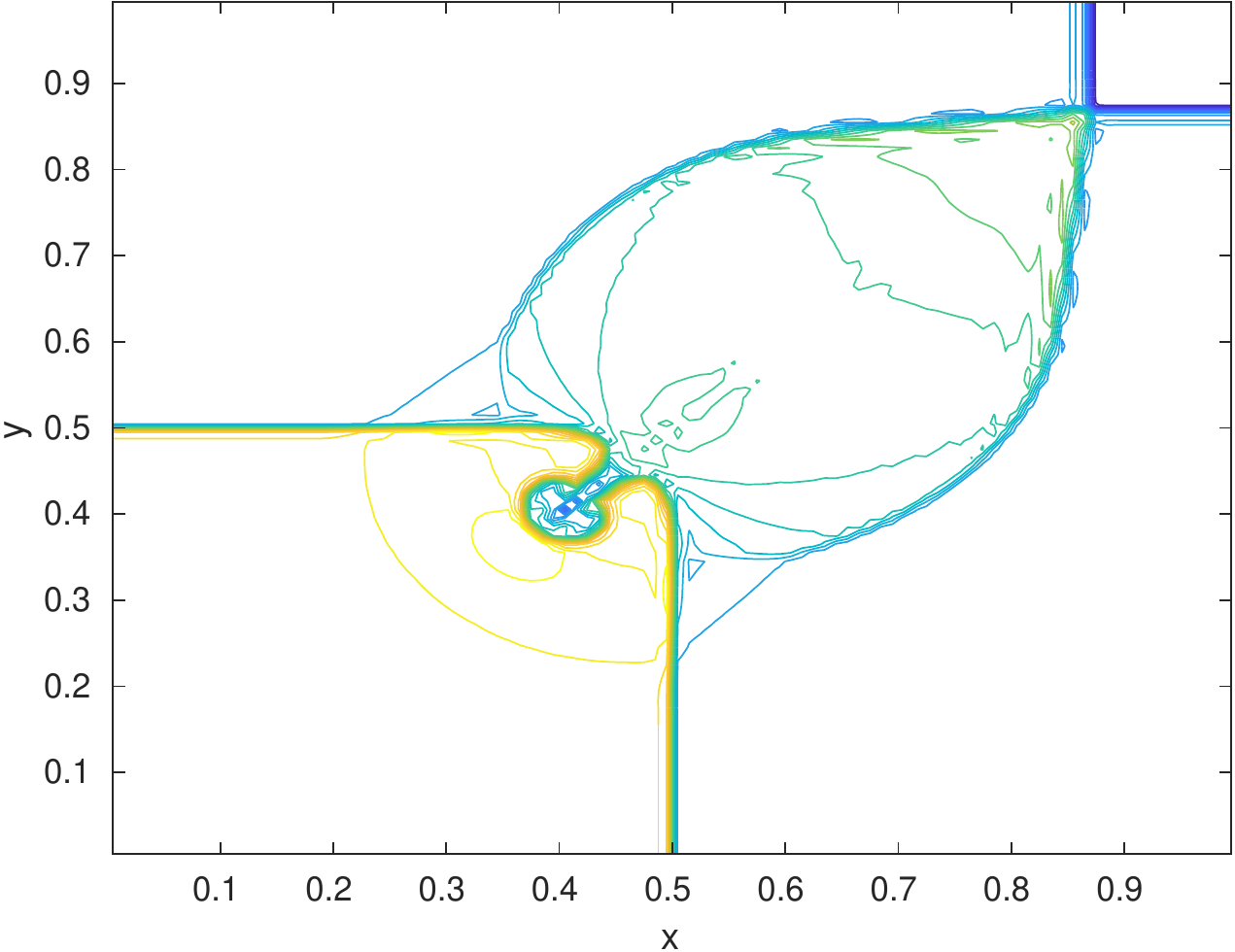}
			\caption{$\ln(\rho)$ using ESDG-O3}
		\end{subfigure}	
		\begin{subfigure}[b]{0.45\textwidth}
			\includegraphics[width=\textwidth]{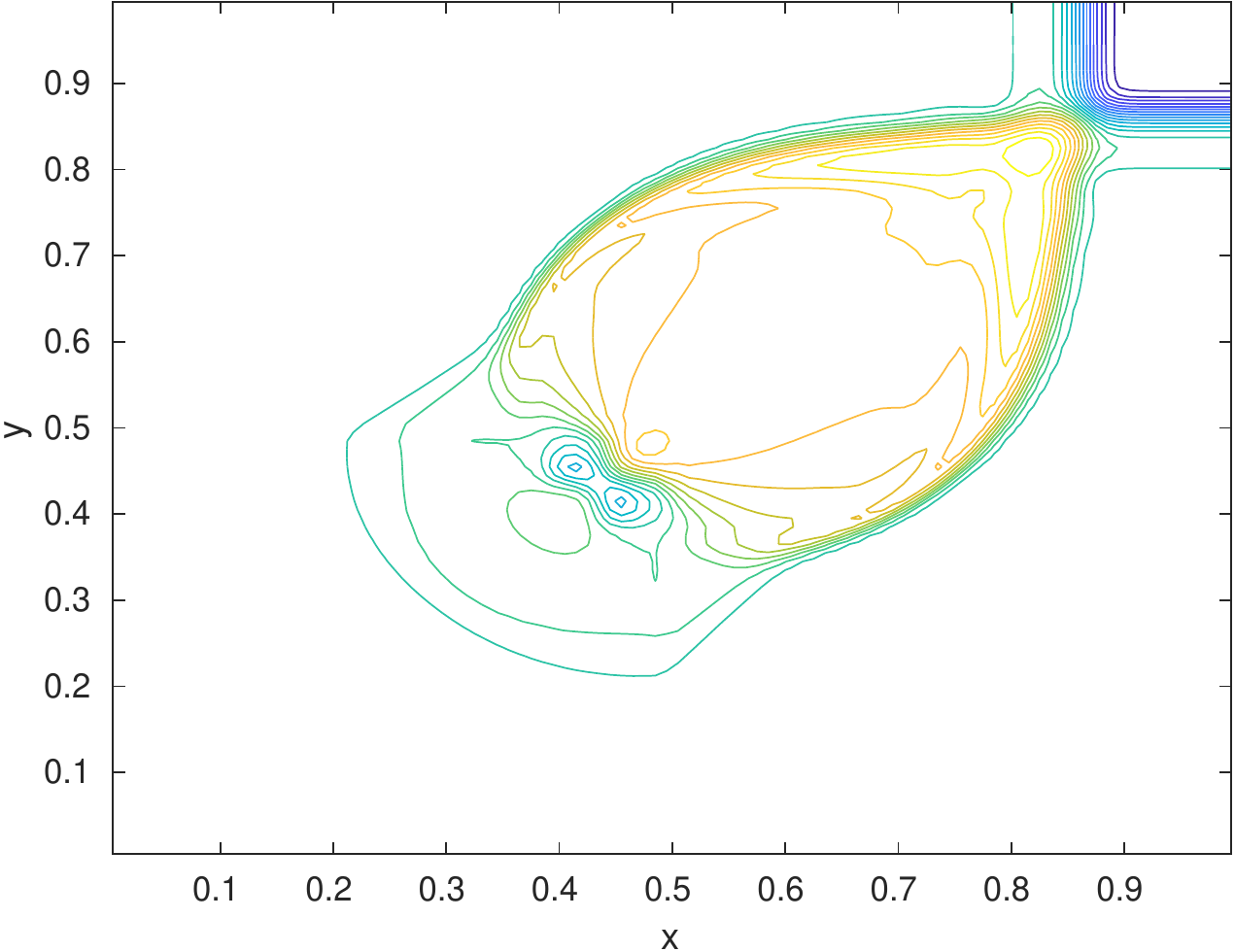}
			\caption{$\ln(p)$ using ESDG-O2}
		\end{subfigure}
		\begin{subfigure}[b]{0.45\textwidth}
			\includegraphics[width=\textwidth]{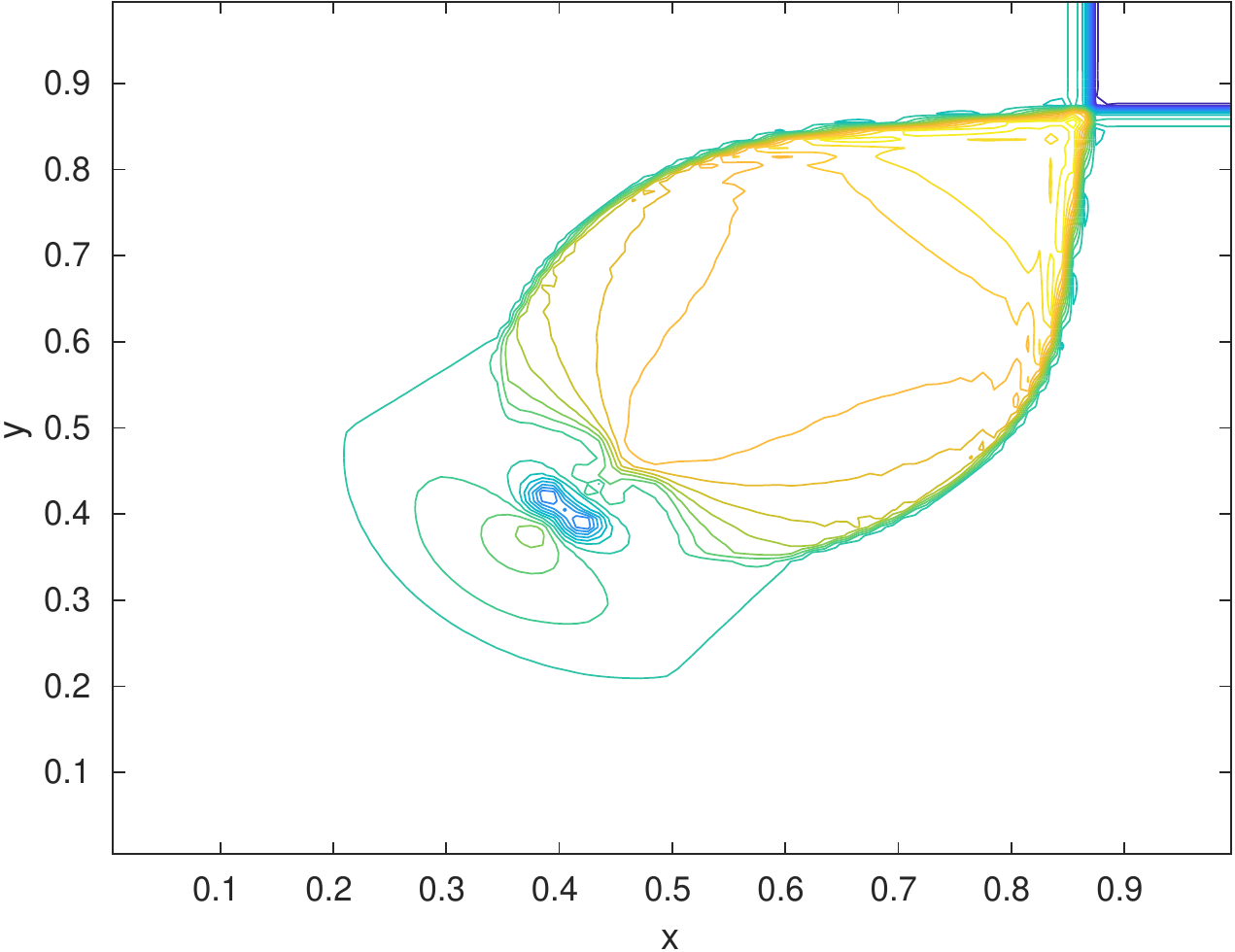}
			\caption{$\ln(p)$ using ESDG-O3}
		\end{subfigure}
		\caption{Test Problem \ref{test12} (Two-dimensional Riemann problem 4): Plots of $\ln(\rho)$ and $\ln(p)$ at time $t=0.4$ using $100\times100$ mesh.}
		\label{fig:pb16}
	\end{figure}
\end{example}

\section{Conclusion}
\label{sec:con}
In this article, we have considered the equation of special relativistic hydrodynamic with the ideal equation of state. We also present the entropy and entropy flux for the system.  Then we design the higher-order entropy stable discontinuous Galerkin scheme for the system in both one- and two-dimensions. This is achieved using an entropy conservative numerical flux from \cite{deepak2019entropy} in cells and an entropy stable numerical flux at the cell interfaces. Furthermore, following \cite{chen2017entropy}, we use  Gauss-Lobatto quadrature rules which have SBP property. The resulting schemes are shown to be entropy stable at the semi-discrete level. For the time discretization, we have used SSP Runge Kutta methods. These schemes are then tested on several test cases in one- and two-dimensions. 
\section*{Acknowledgment} Harish Kumar has been funded in part by SERB, DST MATRICS  grant with file No. MTR/2019/000380. 
\bibliography{main}
\pagebreak

\end{document}